\documentclass[11pt]{article}
\usepackage[a4paper, margin=2.5cm]{geometry}
\usepackage[utf8]{inputenc}
\usepackage{tikz}
\usepackage{mathtools}
\usepackage{mathrsfs,amssymb,amsfonts,amsthm,amsmath,amsbsy}
\usepackage{stmaryrd}
\usepackage[T1]{fontenc}
\usepackage{hyperref}
\usepackage{float}
\usepackage{amssymb}
\usepackage{verbatim}
\usepackage{dsfont}
\usepackage{stmaryrd}
\usepackage{soul}
\usepackage{titlesec} 
\usepackage{listings}
\usepackage{lipsum}
\usepackage[export]{adjustbox}
\usepackage[
backend=biber,
style=alphabetic, 
citestyle=ieee-alphabetic,
maxnames=10
]{biblatex}
\newtheorem{theorem}{Theorem}[section]

\newtheorem{proposition}[theorem]{Proposition} 
\newtheorem{corollary}[theorem]{Corollary}
\newtheorem{lemma}[theorem]{Lemma}

\newtheorem{definition1}[theorem]{Definition}

\newtheorem{remark1}[theorem]{Remark}

\newcommand{\somme}[3]{\sum_{#1}^{#2} #3}
\newcommand{\indi}[1]{\mathds{1}_{#1}}
\newcommand{\Pf}{\mathbb{P}}

\newcommand{\func}[5]{\begin{array}{l|rcl}

#1: & #2 & \longrightarrow & #3 \\
    & #4 & \longmapsto & #5 \end{array}}
\newcommand{\ent}[1]{\lfloor #1 \rfloor}
\newcommand{\union}[3]{\bigcup_{#1}^{#2} #3}

\DeclarePairedDelimiter\floor{\lfloor}{\rfloor}

\title{Quenched invariance principle for biased random walks in random conductances in the sub-ballistic regime}

\author{Alexander Fribergh \thanks{Universit\'{e} de Montr\'{e}al, Email: \texttt{alexander.fribergh@umontreal.ca}} \and Tanguy Lions \thanks{ENS Paris-Seclay, Email: \texttt{tanguy.lions@ens-paris-saclay.fr}} \and Carlo Scali \thanks{Technische Universität München, Email: \texttt{carlo.scali@tum.de}}}
\date{\vspace{-5ex}}

\begin{document}
\maketitle

\begin{abstract}
We consider a biased random walk in positive random conductances on $\mathbb{Z}^d$ for $d\geq 5$. In the sub-ballistic regime, we prove the quenched convergence of the properly rescaled random walk towards a Fractional Kinetics.
\end{abstract}

% \begin{résumé}
% We consider a biased random walk in positive random conductances on $\mathbb{Z}^d$ for $d\geq 5$. In the sub-ballistic regime, we prove the quenched convergence of the properly rescaled random walk towards a Fractional Kinetics.
% \end{résumé}

\section{Introduction}

Random walks in random environment (RWRE) have been extensively studied since the mid seventies and especially so in the last two decades. We refer the reader to~\cite{Zeitouni},
\cite{Sz1}, \cite{Sz2}, \cite{Kumagai} and~\cite{BA-Frib} for different surveys of the field.

A lot of effort has been directed towards proving invariance principles in various models. The majority of those results have been obtained in isotropic settings. We could mention as an example the numerous works done for the simple random walk on supercritical percolation clusters in $\mathbb{Z}^d$. This started with results in the annealed setting in the eighties, see~\cite{deMasi1} and~\cite{deMasi2} (based on works of~\cite{Kipnis-Var} and~\cite{Kozlov}) which was later generalized into a quenched result, initially for $d\geq 4$ in~\cite{SidoSz} and later in dimensions 2 and 3 in~\cite{MP} and~\cite{Berger_invariance}. 

In models where trapping occurs, the classical invariance principle  is not expected to hold. However, using a different time scaling, it is still possible to obtain a convergence result where the limiting Brownian motion is replaced by a Fractional Kinetics (a Brownian motion time-changed by a stable subordinator). For example, Barlow and \v{C}ern\'{y} prove in~\cite{BarlowCerny} subdiffusive scaling limit in both the Bouchaud trap model (BTM) and the random conductance model. Mourrat gives a different proof for the result of~\cite{BarlowCerny} for the Bouchaud trap model in~\cite{Mourrat}.

All the results mentioned above are for isotropic models. Other cases, such as models with directional transience, have attracted a lot of interest as well. Annealed invariance principles have been obtained for such models on $\mathbb{Z}$ (see~\cite{Sabot_Enriquez_Zindy}), trees (see~ \cite{Frib_trees,Hammond_trees}) and on $\mathbb{Z}^d$ with $d \geq 2$ (see~\cite{Frib_perco,Kious_Frib}). Moreover, there are a few results in the quenched setting as well on $\mathbb{Z}$ (see~\cite{Varadhan_large_deviation}), trees (see~\cite{Bowditch}) and $\mathbb{Z}^d$ (see~\cite{Berger_Zeitouni, RassoulAghaSeppalainen}).

The models with directional transience of interest are typically studied because several of them exhibit slowdown effects due to a trapping phenomenon, such as \cite{Sabot_Enriquez_Zindy, berger2020, Frib_trees,aidekon_trees,Hammond_trees, Frib_perco,Frib_conduc,Kious_Frib, Interlac}. The most precise results for an anisotropic model on $\mathbb{Z}^d$ with $d \geq 2$ were obtained by Fribergh and Kious in \cite{Kious_Frib}, where it is shown that under heavy tail condition for conductances we have an aging phenomenon for the walk and an annealed scaling limit related to Fractional Kinetics.

In light of Mourrat's result in~\cite{Mourrat} it is natural to expect that the annealed scaling limit obtained in~\cite{Kious_Frib} should hold in the quenched setting as well. This is the main result of this article. 

It provides a first example for a directionally transient RWRE which converges to a Fractional Kinetics under the quenched setting.

\subsection{Definition of the model and main result}

Fix $d \geq 5$. Let us consider the following random walk in random environment. Let $\mathbf{P}$ be a measure on the set of nearest neighbour edges of the square lattice $E(\mathbb{Z}^d)$, i.e.\ $e \in E(\mathbb{Z}^d)$ if $e = [x, y]$ for $x \sim y$ (adjacent vertices in $\mathbb{Z}^d$). In particular, let $\mathbf{P}(\cdot) = \mu^{\otimes E(\mathbb{Z}^d)}$, where $\mu$ is a probability measure on $(0,+\infty)$. For $e \in E(\mathbb{Z}^d)$, we denote by $c_{*}(e)$ the random variable of the marginal of the conductance of the edge $e$ under $\mathbf{P}$. Let us call the space of environments $\Omega$ and let $\omega \in \Omega$ be a realisation of the environment. We assume throughout the article that 
\begin{equation*}
    \mathbf{P}\left(c_{*}(e) \ge u \right) = \mu([u,+\infty[) = L(u) u^{-\gamma} \quad \textnormal{for any }u \ge 0,
\end{equation*}
with $\gamma \in (0,1)$ and where $L(\cdot)$ is a slowly-varying function.

We denote by $\{c^{\omega}_{*}(e)\}_{e \in E(\mathbb{Z}^d)}$ the family of conductances in the fixed environment $\omega$. To introduce the bias, consider a unit vector $\vec{\ell} \in \mathbb{S}^{d-1}$ and $\lambda > 0$, denote $\ell \coloneqq \lambda \vec{\ell}$ and set, for $e = [x,y] \in E(\mathbb{Z}^{d})$, 
\begin{equation*}
    c^{\omega}(e) = c^{\omega}_{*}(e)e^{(x+y)\cdot  \ell}.
\end{equation*}
We consider the discrete time random walk $(X_n)_{n \geq 0}$ on $\mathbb{Z}^d$, on the random conductances $(c^\omega(e))_{e \in E(\mathbb{Z}^d)}$  (RWRC). We call $P_{x}^{\omega}$ the quenched law of the random walk on the environment $\omega$, starting from $x \in \mathbb{Z}^d$ and with transition probabilities, for $z \sim y \in \mathbb{Z}^d$,
\begin{equation}\label{QuenchedLawDef}
	p^\omega(z, y) = \frac{c^\omega([z, y])}{\sum\limits_{w \sim z} c^\omega([z, w])},
\end{equation}
and $p^\omega(z, y) = 0$ whenever $z \nsim y$. Let $\mathbb{P}_{x} = \mathbf{E}[P_{x}^{\omega}(\cdot)]$ be the annealed law of the random walk $(X_n)_{n}$ ($\mathbf{E}[\cdot]$ is the expectation with respect to $\mathbf{P}$ and $\mathbf{Var}$ the corresponding variance). For any $T>0$, let $D^d([0, T])$ denote the space of c\`{a}dl\`{a}g functions from $[0, T]$ to $\mathbb{R}^d$. The main theorem of this article is the following.

\begin{theorem}\label{MainTheorem}
	Let $d \ge 5$ and fix any $T > 0$, the following statements hold for $\mathbf{P}$-almost every environment $\omega \in \Omega$. There exists a deterministic $v_0 \in \mathbb{S}^{d-1}$ with $v_0 \cdot \vec{\ell} > 0$ such that, under the quenched law $P^\omega_0(\cdot)$, it holds that
	\begin{equation}\label{FirstStatementMain}
		\left( \frac{X_{\floor{n t}}}{n^\gamma/L(n)} \right)_{0 \le t \le T} \overset{(\mathrm{d})}{\longrightarrow} \left( v_0 C_\infty^{-\gamma} \mathcal{S}_\gamma^{-1}(t) \right)_{0 \le t \le T},
	\end{equation}
 furthermore,
	\begin{equation}\label{SecondStatementMain}
		\left( \frac{X_{\floor{n t}} - (X_{\floor{n t}} \cdot v_0) v_0}{\sqrt{n^\gamma/L(n)}} \right)_{0 \le t \le T} \overset{(\mathrm{d})}{\longrightarrow}\left( M_d B_{\mathcal{S}_\gamma^{-1}(t)} \right)_{0 \le t \le T},
	\end{equation}
	for a deterministic constant $C_\infty >0$, a deterministic $d \times d$ matrix $M_d$ of rank $d-1$, a standard Brownian motion $B$ and a stable subordinator of index $\gamma$, independent of $B$, $\mathcal{S}_\gamma$. The first convergence holds in the uniform topology on $D^d([0, T])$ while the second holds in the Skorohod's $J_1$-topology on $D^d([0, T])$.
\end{theorem}
\noindent It may be possible to extend this result and (partially) the method of our proof to $d=4$. However, the key \textit{heat} kernel estimate of Proposition~\ref{TechnicalInequality} is not strong enough to make our method work in dimension $d=4$. Hence, we believe that proving the theorem in lower dimensions would constitute a non-trivial extension of this work.

\subsubsection{Quenched convergence via series convergence} \label{SectionIntroSeries}

Let us explain the main steps of our strategy. In \cite{Kious_Frib} one can find the definition of a regeneration structure with suitable independence properties, in particular let $(\tau_k)_{k \ge 0}$ be the regeneration times. We recall the rigorous definition of $(\tau_k)_{k \ge 0}$ in Section~\ref{definition_reg_times}. Let $v \coloneqq \mathbb{E}[X_{\tau_2} - X_{\tau_1} ]$, we introduce, for all $t \in [0, T]$,
\begin{equation}\label{EqnThreeMainQuantities}
	Y_n(t) = \frac{X_{\tau_{\floor{tn}}}}{n}, \quad Z_n(t) = \frac{X_{\tau_{\floor{tn}}} - v n t}{n^{1/2}} , \quad S_n(t) = \frac{\tau_{\floor{tn}}}{\mathrm{Inv}(n)}, \quad \textnormal{and} \quad W_n(t) = \left(Z_n(t),S_n(t)\right),
\end{equation}
where $\mathrm{Inv}(u) \coloneqq \inf \left\{ s: \mathbf{P}[c_{*} > s]\le 1/u \right\}$ is the generalised right-continuous inverse of the tail of $c_{*}$, let us point out that $\mathrm{Inv}(n) \approx n^{1/\gamma}$. Our initial aim is to prove that for $\mathbf{P}$-almost all $\omega \in \Omega$ these four quantities converge weakly, then we will use this fact to prove Theorem~\ref{MainTheorem}. We get the quenched convergence of $Y_n(t)$ for free from \cite[Lemma 11.2]{Kious_Frib}, as the limit is deterministic. For technical reasons let us define the process
\begin{equation}\label{ShiftedClock}
	W^*_n(t) = W_n\left(t+\tfrac{1}{n}\right) - W_n\left(\tfrac{1}{n}\right).
\end{equation} 
In order to prove quenched convergence of $W_n$ we will use alternative formulations considered in \cite{Szn_serie} and \cite{Mourrat} to rephrase the problem.
These conditions are stated in the following theorem.

\begin{theorem}\label{TheoremVarianceClockProcess}
	\noindent Let $d \ge 5$ and fix $T>0$, for any positive and bounded Lipschitz function $F_1 \colon (\mathbb{R}^{d+1})^{m} \rightarrow \mathbb{R}$ and $0 \le t_1 \le \cdots \le t_m \le T$ it holds that
	\begin{equation} \label{EquationBoundVarianceClock}
		\sum_{n = 1 }^{+\infty}\mathbf{Var}\left(E_{0}^{\omega}\left[F_1(W^{*}_{b^n}(t_1),\cdots,W^{*}_{b^n}(t_m))\right]\right)< +\infty,
	\end{equation}
	with $b \in (1,2)$. Moreover, by denoting with $Z_n$ the polygonal interpolant of the quantity denoted with the same letter in \eqref{EqnThreeMainQuantities}, we have
	\begin{equation} \label{EquationBoundVarianceTraj}
		\sum_{n = 1 }^{+\infty}\mathbf{Var}\left(E_{0}^{\omega}[F_2(Z_{b^n})]\right)< +\infty,
	\end{equation}
	where $F_2$ is any bounded Lipschitz function from $\mathcal{C}^d([0, T]) \to \mathbb{R}$.
\end{theorem}

\subsubsection{Sketch of proof}
In order to prove Theorem~\ref{TheoremVarianceClockProcess}, we can consider under each environment two independent random walks $X^1_{\cdot}$ and $X^2_{\cdot}$ which have the same law as $X_{\cdot}$ under $P_0^{\omega}$. Let $\mathbb{E}_{0, 0}$ be the natural generalisation of the annealed law to two walks (see \eqref{2WalksSameEnv}). Furthermore, let $Z^1_{\cdot}$ and $Z^2_{\cdot}$ be the polygonal interpolant of the quantity $Z_n$ defined in \eqref{EqnThreeMainQuantities} but constructed, respectively, using $X^1_{\cdot}$ and $X^2_{\cdot}$. The term inside the sum in \eqref{EquationBoundVarianceTraj} is equal to
\begin{equation}\label{EquationCovarianceIndepWalks}
	\mathbb{E}_{0, 0}[F_2(Z_{b^n}^1)F_2(Z_{b^n}^2)] - \mathbb{E}_{0}[F_2(Z_{b^n}^1)]\mathbb{E}_{0}[F_2(Z_{b^n}^2)].
\end{equation}
We will show that the trajectories of the two walks do not meet each other after reaching a distance not too far from the origin, then deduce that the quantity in \eqref{EquationCovarianceIndepWalks} goes quickly to $0$. Indeed, when the two walks do not meet, they behave almost as if they evolved in independent environments.

The important ideas can be grouped in two main parts. The first one deals with \emph{joint regeneration levels} while the second consists in proving the \emph{asymptotic separation} using \emph{joint regeneration levels}. 
\vspace{0.4ex}

\noindent\textbf{Joint regeneration levels}: Section~\ref{JoiRegLev} will be devoted to extending the notion of \emph{regeneration times} to the case where we have two walks. We will construct \emph{joint regeneration levels}: a sequence of levels (in the direction $\vec{\ell}$) such that the hitting time of these levels are regeneration times for both walks. A similar structure in the ballistic setting was constructed in \cite{RassoulAghaSeppalainen}.
\vspace{0.4ex}

\noindent\textbf{Asymptotic separation}: Once we have constructed \emph{joint regeneration levels} the idea is the following: at each regeneration level, both walks have a decent chance of staying far apart for a long time, after which it is unlikely that they will ever meet again. As joint regeneration levels are common, the walks will decouple relatively quickly. In this part we use a method which is inspired by \cite{Berger_Zeitouni}. However, we cannot directly use their results because our model lacks both the uniform ellipticity property of the jump probabilities and it is not ballistic.

\section{Notations} \label{SectionNotation}

We fix an orthonormal basis $\{e_1,...,e_d\}$ of $\mathbb{Z}^d$ such that $e_1\cdot\vec{\ell} \geq e_2\cdot\vec{\ell} \geq ... \geq e_d\cdot\vec{\ell} \geq 0$. In particular, we have that $e_1\cdot\vec{\ell} \geq 1/\sqrt{d}$. We also set the convention that $e_{j + d} = - e_j$ for all $j = 1, \dots, d$. We also fix an orthonormal basis of $\mathbb{R}^d$, $\{f_1,...,f_d\}$ such that $f_1 = \vec{\ell}$.

For any $L \in \mathbb{R}$ we set
\begin{equation}\label{H+-}
    \mathcal{H}^{+}(L) = \mathcal{H}^{+}_L = \{z \in \mathbb{Z}^d \colon z\cdot\vec{\ell} > L\} \quad  \mathrm{and} \quad  \mathcal{H}^{-}(L) =\mathcal{H}^{-}_L= \{z \in \mathbb{Z}^d \colon z\cdot\vec{\ell} \leq L\}.
\end{equation}
Similarly, for $z\in \mathbb{Z}^d$
\begin{equation*}
    \mathcal{H}^{+}_{z} = \mathcal{H}^{+}\left(z\cdot\vec{\ell}\right) \quad \mathrm{and} \quad \mathcal{H}^{-}_{z} = \mathcal{H}^{-}\left(z\cdot\vec{\ell}\right).
\end{equation*}
For any $x \in \mathbb{Z}^d$ we define
\begin{equation*} \label{neighbourhood}
    \mathcal{V}_x = \{x,x+e_1,x-e_1,\cdots,x+e_d,x-e_d\}.
\end{equation*}
For $x,y \in \mathbb{Z}^d$, we recall that if $x = y \pm e_i$ for some $i \in \{1,\cdots,d\}$ then $x$ and $y$ are adjacent and we write $x \sim y$. For $x$ in $\mathbb{Z}^d$ and $e = [e^{+}, e^{-}]$ and edge, we write $x \sim e$ if $x \sim e^{-}$ or $x \sim e^{+}$ and $x \notin \{e^{+},e^{-}\}$. Two distinct edges $e$ and $e'$ are said to be adjacent and we write $e\sim e'$ if they share one endpoint.

For $x \in \mathbb{Z}^d$ we define the following sets of edges
\begin{equation}\label{EquationIncidentEdges}
    \mathcal{E}_x \coloneqq \left\{[x,y] \colon y \sim x\right\}, \,\,     \mathcal{E}^{k}_x \coloneqq \mathcal{E}_{\{x,x-e_1,\cdots,x-ke_1\}}, \,\,  \text{ and for }A\subset \mathbb{Z}^d \text{, }\mathcal{E}_A = \bigcup_{x\in A}\mathcal{E}_x.
\end{equation}
\begin{equation}\label{EquationLeftEdges}
    \mathcal{L}^x \coloneqq \left\{[y,z] \colon y  \in \mathcal{H}^-_{x} \text{ and } z \in \mathcal{H}^-_{x}\right\} \cup \mathcal{E}_x.
\end{equation}
\begin{equation}\label{EquationRightEdges}
    \mathcal{R}^x \coloneqq \left\{[y,z] \colon y\in \mathcal{H}^+_{x} \text{ or } z \in \mathcal{H}^+_{x}\right\} \cup \mathcal{E}_x.
\end{equation}
For any subset $V$ of $\mathbb{Z}^d$ we define its edge-set
\begin{equation*}
E(V) \coloneqq \{[x,y] \colon x \in V \text{ and } y \in V\},
\end{equation*}
its \emph{vertex-boundary}
\begin{equation*}
    \partial V \coloneqq \{x \notin V \colon \exists y \in V, x \sim y \},
\end{equation*}
and its \emph{edge-boundary}
\begin{equation*}
 \partial_{E} V \coloneqq \left\{[x,y]\in E(\mathbb{Z}^d) \colon y \in V, x \notin V \right\}.
\end{equation*}
For any subset $A \subset \mathbb{Z}^d$ we define its width
\begin{equation*}
    W(A)= \max_{1\leq i \leq d}\left(\max_{y \in A}y \cdot e_i - \min_{y \in A}y \cdot e_i\right).
\end{equation*}
For any $L,L' > 0$ and $y \in \mathbb{Z}^d$ we define the box
\begin{equation}\label{eqn:TiltedBox}
    \mathcal{B}_y(L,L') = \left\{x \in \mathbb{Z}^d \colon |(x-y)\cdot\vec{\ell}| \leq L \text{ and } |(x-y)\cdot f_i| \leq L' \text{ for all } i \in \{2,...,d\} \right\},
\end{equation}
and its \emph{positive boundary}
\begin{equation*}
    \partial^{+}\mathcal{B}_y(L,L') = \{ x \in \partial \mathcal{B}_y(L,L') \colon (x-y)\cdot\vec{\ell} > L\},
\end{equation*}
when $y = 0$ we use the notation $\mathcal{B}(L,L')$.

For a general random walk $(X_n)_{n}$ on a graph $(V,E)$ we define for $B \subset V$ the hitting times of $B$,
\begin{equation} \label{T_B}
    T_B = \inf\{n \geq 0 \colon X_n \in B\} \quad \text{ and } \quad T^{+}_B = \inf\{n \geq 1 \colon X_n \in B\},
\end{equation}
and the exit time of $B$
\begin{equation*}
    T^{\mathrm{ex}}_B = \inf\{n \geq 0 \colon X_n \notin B\},
\end{equation*}
when $B$ is a singleton $\{x\}$ we use the notation $T_x$. For $R \in \mathbb{R}$ we define
\begin{equation}\label{HittingHalfSpace}
    T_{R} = \inf \{ n \ \ge 0 : X_n \cdot \vec{\ell} \ge R\}.
\end{equation}
We denote by $\theta_n$ the shift by $n$ units of time, when we consider two random walks $\theta_{s, t}$ denotes the time shift of the first by $s$ time units and of the second by $t$ time units, this operator is defined formally in \eqref{eqn:DoubleTimeShift}.
For $\Pf$ a probability measure, $A_1,\cdots,A_n$ events and $X$ a random variable, the notation $\mathbb{E}[A_1,\cdots,A_n,X]$ denotes $\mathbb{E}[\indi{A_1}\cdots\indi{A_n}X]$.

In most of the paper we will add an index $i \in \{1,2\}$ as superscript when the event depends on $X^i$, where $X^i$ are two realisations of the random walk (independent with respect to the quenched distribution). For example, $T^{i}_R$ will be the hitting time of $\mathcal{H}^{+}(R)$ by the $i$-th random walk.

In the following, we will use the notations $c$, $C$, $a$ to denote constants which only depend on the dimension $d$, the strength of the bias $\lambda$, its direction $\vec{\ell}$ and a parameter $K > 0$.

\paragraph{Probability measures.} Let us recall that $P^\omega_{x}$ denotes the quenched law while $\Pf_{x}$ denotes the annealed law of the walk (started at $x$). As we already mentioned, we will often consider two walks $X^1_{\cdot}$ and $X^2_{\cdot}$, for these walks it is useful to introduce two probability measures. The first one is for two walks that evolve independently (according to the law defined in \eqref{QuenchedLawDef}) in the same environment. Let $U_1, U_2 \in \mathbb{Z}^d$, we define
\begin{align}\label{2WalksSameEnv}
    \Pf_{U_1, U_2}\left(X^1_{\cdot} \in A_1, X^2_{\cdot} \in A_2\right) \coloneqq \mathbf{E}\left[ P_{U_1}^\omega\left( X^1_{\cdot} \in A_1 \right) \times P_{U_2}^\omega\left( X^2_{\cdot} \in A_2 \right) \right].
\end{align}
It is also possible to define a measure according to which the two walks evolve with the same quenched law in two independent environments $\omega_1, \omega_2$ sampled according to $\mathbf{P}$. In formulas,
\begin{align}
    Q_{U_1, U_2}\left(X^1_{\cdot} \in A_1, X^2_{\cdot} \in A_2\right) &\coloneqq \mathbf{E}\left[ P_{U_1}^{\omega_1}\left( X^1_{\cdot} \in A_1 \right)\right] \times \mathbf{E}\left[ P_{U_2}^{\omega_2}\left( X^2_{\cdot} \in A_2 \right) \right] \label{2WalksIndEnv} \\ &= \Pf_{U_1}\left( X^1_{\cdot} \in A_1 \right) \times \Pf_{U_2}\left( X^2_{\cdot} \in A_2 \right).\nonumber
\end{align}
We denote by $\mathbb{E}_{U_1,U_2}[\cdot]$ the expectation with respect to the measure \eqref{2WalksSameEnv}. This is the same notation as the annealed law of one walk (with the exception of the two starting points) because the marginal of $X^1_{\cdot}$ (resp.\ $X^2_{\cdot}$) is the same as the law of the single walk. We may write $\mathbb{P}_{0}(\cdot) = \mathbb{P}_{0, 0}(\cdot)$ and $\mathbb{E}_{0}[\cdot] = \mathbb{E}_{0, 0}[\cdot]$ for simplicity when it is clear from the context that we are considering two walks. $\mathbb{E}^{Q_{U_1,U_2}}[\cdot]$ is the expectation associated to \eqref{2WalksIndEnv}. 

For $v \in \mathbb{R}^{2d}$ and $x \in \mathbb{Z}^d$ we write $\Pf_x^{v}$ (resp.\ $\Pf_{x,y}^{v}$) for the annealed law where we set the conductances at the starting point (resp.\ at the two starting points) to the values of $v$. More precisely, recalling \eqref{EquationIncidentEdges}, we fix for all $[x, x + e_i]\in \mathcal{E}_x$ the conductance $c_{*}^{\omega}([x, x + e_i]) = v_i$ for $i = 1, \dots, 2d$. A special case is when $v = (K,\ldots,K)$ with $K > 0$, in that case we write $\Pf_x^{K}$ (resp. $\Pf_{x,y}^{K}$). The notation $Q_{x,y}^{v}$ refers to the law $\Pf_x^{v}\otimes \Pf_y^{v}$. 

In accordance with the notation of \cite[Equation (5.3)]{Kious_Frib} (we will recall the precise meaning in Section~\ref{definition_reg_times}) we write $\Pf_x^{K}(\cdot|D=+\infty)$ for the law of the walk conditioned to regenerate at $x$ at time $0$. We also write $Q_{x,y}^K(\cdot|D^{\otimes}=+\infty) = \Pf_x^{K}(\cdot|D=+\infty)\otimes\Pf_y^{K}(\cdot|D=+\infty)$.

\section{Joint regeneration levels}\label{JoiRegLev}

For this model a regeneration structure was introduced in \cite{Frib_conduc}. It was further improved in \cite{Kious_Frib} to obtain the independence of trajectories between regeneration times. For our purpose it will be more convenient to have joint regeneration levels. The goal of this section is to introduce a joint regeneration structure for two walks in the same environment. 

\begin{figure}[H]
    \centering
    \includegraphics[scale = 0.35]{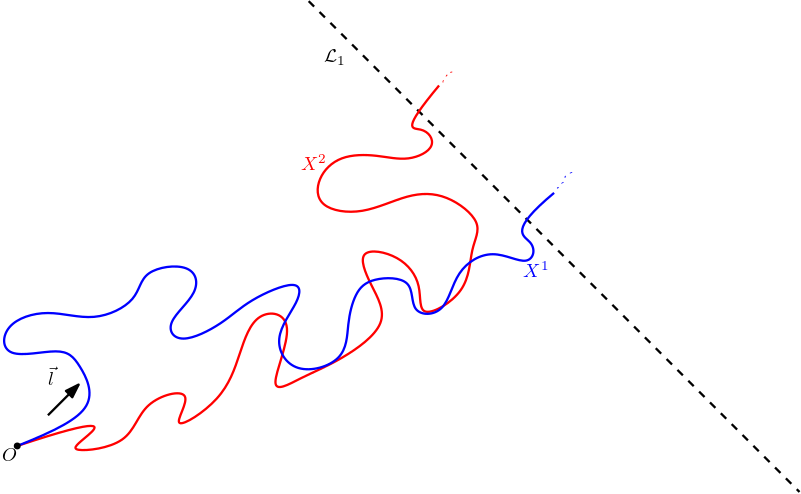}
    \caption{Two trajectories and the first joint regeneration level $\mathcal{L}_1$.}
    \label{fig:figJointRegLevel}
\end{figure}
In principle, a regeneration time is simply a new maximum in the direction of the bias which will be a minimum for the future of the trajectory. However, for technical reasons we want to work with a subset of those regeneration times: those happening on points with bounded conductances. This is why we will adapt the construction of \cite{Kious_Frib} which uses $K$-open ladder points.

The first step is to introduce the notion of \emph{joint $K$-open ladder point} and we do this in the next subsection.

\subsection{Joint $K$-open ladder points}

\subsubsection{Notations and preliminary results} 

\begin{definition1}\label{DefinitionK-Open}
Fix $K\ge 1$, we say that an edge $e$ is $K$-normal if $c_{*}(e) \in [1/K,K]$. A point $z \in \mathbb{Z}^d$ is said to be $K$-open if all the edges incident to it are $K$-normal, i.e.\ if all $e \in \mathcal{E}_z$ are such that $c_{*}(e) \in [1/K,K]$. A vertex that is not $K$-open is said to be $K$-closed.
\end{definition1}

Joint $K$-open ladder points are new maxima in the direction of the bias $\vec{\ell}$ for both walks, such that the points are $K$-open and the maxima are close in the sense that they are both at roughly the same distance, say $\approx R$, from the origin in the bias direction. Note that the distance cannot be exactly the same due to the lattice structure.

The main goal of this subsection is to prove that the first joint $K$-open ladder point is close enough to the origin in the bias direction. The main result, that will be used throughout the paper, is Proposition~\ref{PropositionRegLevelSmall} which follows from Lemma~\ref{LemmaSuccessiveAttepts}. Lemma~\ref{PropFrequentRegLevels} which is related to Proposition~\ref{PropositionShortRegLevel} will also be instrumental for the rest of the proof.

Let us first introduce some useful quantities and results from \cite[Section 6]{Frib_conduc}. We introduce the inner positive boundary of $\mathcal{B}_{x}(L, L')$
\begin{equation*}
    \partial_{\mathrm{in}}^+ \mathcal{B}_{x}(L, L') \coloneqq \{ z \in \mathcal{B}_{x}(L, L'), \text{ where } z \sim y \text{ with } y \in \partial^+ \mathcal{B}_{x}(L, L') \}.
\end{equation*}
We consider two walks $(X^i_n)_{n\geq 0}$ for $i\in \{1,2\}$, set $\alpha > 0$ and define the events
\begin{equation} \label{EventExitRightBoundary}
    A^{i}(L) \coloneqq \left\{ T^i_{\partial \mathcal{B}_{X^{i}_{0}}(L, L^\alpha)} \ge T^i_{\partial^+_{\mathrm{in}} \mathcal{B}_{X^{i}_{0}}(L, L^\alpha)}\right\}, \quad \mathrm{and} \quad A(L) \coloneqq A^{1}(L) \cap A^{2}(L).
\end{equation}
In all the following we write $\mathcal{B}_{i} = \mathcal{B}_{X_{0}^{i}}$ and $\mathcal{B} = \mathcal{B}_{0}$.

\begin{lemma}(Analog of \cite[Lemma 6.1]{Frib_conduc})\label{LemmaA(n)}
For $\alpha > d + 3$ and any $L > 0$ we have
\[
\sup_{U_1, U_2 \in \mathbb{Z}^d} \Pf_{U_1, U_2}\left( A(L)^c \right) \le C e^{-c L},
\]
where the constants $C,c$ only depend on $\lambda,\vec{\ell},d$. The same statement holds under the law $\Pf^K_{U_1, U_2}$.
\end{lemma}
\begin{proof}
Fix any $U_1, U_2 \in \mathbb{Z}^d$. We aim to bound
\begin{align*}
    \Pf_{U_1, U_2}\left( A(L)^c \right) &= \Pf_{U_1, U_2}\left( A^1(L)^c \cup A^2(L)^c \right) \le \Pf_{U_1}\left( A^1(L)^c \right) + \Pf_{U_2}\left( A^2(L)^c \right),
\end{align*}
where the inequality is due to a union bound. By \cite[Theorem 5.1]{Frib_conduc} we have that $\Pf_{0}\left( A^{i}(L)^c \right) \le  C \exp(-c L)$, by the translation invariance of the environment we get
\begin{align*}
    \Pf_{U_1, U_2}\left( A(L)^c \right) \le C e^{-c L},
\end{align*}
as desired. The proof is identical under $\Pf^K_{U_1, U_2}$ but it uses \cite[Theorem 5.2]{Kious_Frib} (which is the analog of \cite[Theorem 5.1]{Frib_conduc} under the measure $\Pf^K_{0}$. 
\end{proof}
Fix $z \in \mathbb{Z}^d$, we say that a vertex $x \in \mathcal{B}_z(L, L^\alpha)$ is $K$-$L$-closed with respect to $z$ with if there exists a neighbour $y \not \in \mathcal{H}^+(L + z \cdot \vec{\ell})$ of $x$ such that $c_{*}([x, y]) \not \in [1/K, K]$. Let $\bar{K}_{x, z}(L)$ be the $K$-$L$-closed connected component of $x$ (with respect to $z$, set it to $\{x\}$ in case it is empty or $x \not \in \mathcal{B}_z(L, L^\alpha)$).  To lighten the notation, we will write $\bar{K}_{x}(L)$ in place of $\bar{K}_{x, z}(L)$ whenever the choice of the point $z$ is clear from the context and fixed. Let us define the event
\begin{equation} \label{event_B}
B(L) \coloneqq \left\{ \textnormal{for all } x \in \partial_{\mathrm{in}}^+ \mathcal{B}_1(L, L^\alpha) \cup \partial_{\mathrm{in}}^+ \mathcal{B}_2(L, L^\alpha), \textnormal{ } |\bar{K}_{x,X_0^1} (L) \cup \bar{K}_{x, X_0^2}(L)| \le \log(L) \right\}.
\end{equation}
\begin{lemma}(Analog of \cite[Lemma 6.2]{Frib_conduc}) \label{LemmaB(n)}
For any $M<\infty$ we can find $K_0>1$ such that for all $K \ge K_0$, 
\[
\sup_{U_1,U_2 \in \mathbb{Z}^d}\Pf_{U_1,U_2}\left( B(L)^c \right) \le C L^{-M},
\]
where $C$ only depends on $\lambda,\vec{\ell},d,K$.
\end{lemma}
\begin{proof}
Fix $U_1,U_2 \in \mathbb{Z}^d$. Using \cite[Lemma 6.2]{Frib_conduc} we have the bound for the event 
\begin{equation*}
    \widehat{B}(L) = \{ \textnormal{for all } x \in \partial_{\mathrm{in}}^+ \mathcal{B}(L, L^\alpha), \textnormal{ we have } |\bar{K}_{x, 0}(L)| \le \log(L)/2 \},
\end{equation*}
\begin{equation*}
    \mathbf{P}\left( \widehat{B}(L)^c \right) \le C L^{-M}.
\end{equation*}
As in the proof of the previous lemma, we get the result by union bound and translation invariance of the the measure $\mathbf{P}$.
\end{proof}
\noindent We recall the following useful lemma.

\begin{lemma}\cite[Lemma 6.3]{Frib_conduc}\label{LemmaBoundaryHitting}
Take $G \ne \varnothing$ to be a finite connected subset of $\mathbb{Z}^d$. Assume each edge $e$ of $\mathbb{Z}^d$ is assigned a positive conductance $c(e)$ and that there exists a constant $c_1>0$, $x \in \partial G$ and $y \in G$ such that $x \sim y$ and $c([x, y]) \ge c_1 c(e)$ for any $e \in \partial_{E} G$. We have
\begin{equation*}
    P_y\left( T_x \le T_{\partial G} \right) \ge \frac{c_1}{4d}|G|^{-1},
\end{equation*}
where $P_y$ is the law of the random walk on the family of conductances $\{c(e)\}_{e \in E(\mathbb{Z}^d)}$, starting from $y$.
\end{lemma}

\subsubsection{Properties of joint $K$-open ladder points}

For $i \in \{1,2\}$ we recall the notation $T^i_R$ defined in \eqref{HittingHalfSpace}. We introduce 
\begin{equation}\label{definition_M}
\begin{split}
    \mathcal{M}^{\bullet} = \mathcal{M}^{\bullet (K)} &= 
    \inf \Big\{ R > \max\{X^{1}_0\cdot \vec{\ell}, X^{2}_0\cdot \vec{\ell}\} \colon \text{ for } i\in \{1,2\},\text{ } X^i_{T^i_{R}} \text{ is } K\text{-open, }  \\ 
    &  X^i_{j} \cdot \vec{\ell} < X^i_{T^i_{R}-2} \cdot \vec{\ell} \text{ for } j < T^i_{R}-2, \, X^i_{T^i_{R}} = X^i_{T^i_{R}-1} + e_1 = X^i_{T^i_{R}-2} + 2e_1, \\ %\Big\}.\\ 
    & X^1_{T^1_{R}} \cdot \vec{\ell} \wedge X^2_{T^2_{R}} \cdot \vec{\ell} = R  \Big\}.
\end{split}
\end{equation}
In general, we will put a bullet when referring to quantities related to two walks evolving in the same environment to avoid confusion with similar quantities for a single walk or two walks in independent environments. Since the definition of $\mathcal{M}^{\bullet}$ corresponds to a level, the joint $K$-open ladder point is then $(X^1_{T^1_{\mathcal{M}^{\bullet}}},X^2_{T^2_{\mathcal{M}^{\bullet}}})$. However, as we will see in the rest of the article, it is much more convenient to work with levels rather than time or points when we consider two walks. We will keep the $K$ in the notation $\mathcal{M}^{\bullet (K)}$ when we want to stress the dependence on it, and drop it to lighten the notation otherwise. Let us define
\begin{equation}\label{eqn:USet}
    \mathcal{U} \coloneqq \left\{ U \in \mathbb{Z}^d : |U\cdot \vec{\ell}| < e_1 \cdot \vec{\ell} \right\}.
\end{equation}
The main proposition of this subsection is the following.
\begin{proposition} \label{PropositionRegLevelSmall}
For any $M > 0$ there exists $K_0 \ge 1$ such that for all $K\geq K_0$ we have
\begin{equation*}
    \sup_{U_1, U_2 \in \mathcal{U}}\Pf_{U_1,U_2}(\mathcal{M}^{\bullet(K)} > n) \leq Cn^{-M},
\end{equation*}
where $C$ only depends on $\lambda,\vec{\ell},d,K$.
\end{proposition}
\noindent In order to prove this result we require a technical lemma, for $1 \le k \le n$, let us define the events 
\begin{equation}
    R^{(K)}(n k) \coloneqq \left\{\mathcal{M}^{\bullet(K)} > kn + 3 \right\}.
\end{equation}
\begin{lemma} \label{LemmaSuccessiveAttepts}
For any $\varepsilon' > 0$ and $M > 0$, there exists $K_0 \ge 1$ such that for all $U_1,U_2 \in \mathcal{U}$, all $K \ge K_0$, $1< k \le n$ we have
\begin{equation*}
    \displaystyle \Pf_{U_1,U_2}\left(R^{(K)}(n k)\right)\leq \left(1-cn^{-\varepsilon'}\right)\Pf_{U_1,U_2}\left(R^{(K)}(n (k-1))\right) + Cn^{-M},
\end{equation*}
where the constants $C,c$ only depend on $\lambda,\vec{\ell},d,K$.
\end{lemma}
\noindent We momentarily assume Lemma~\ref{LemmaSuccessiveAttepts} and use it to prove Proposition~\ref{PropositionRegLevelSmall}.
\begin{proof}[Proof of Proposition~\ref{PropositionRegLevelSmall}]
\noindent One can iterate Lemma~\ref{LemmaSuccessiveAttepts} to get a tail estimate on the quantity $\mathcal{M}^{\bullet (K)}$. Indeed, for any $M<\infty$ and fixing $\varepsilon' = 1/2$ we can inductively see that 
\begin{equation*}
     \Pf_{U_1,U_2}\left(R^{(K)}(n^2)\right)\leq \left(1-cn^{-1/2}\right)^n + Cn^{-M + 1} \le 2Cn^{-M + 1}.
\end{equation*}
Thus, we have that
\begin{equation*}
     \Pf_{U_1,U_2}\left(\mathcal{M}^{\bullet (K)} > n^2 + 3\right) \le 2Cn^{-M + 1},
\end{equation*}
or equivalently
\begin{equation*}
     \Pf_{U_1,U_2}\left(\mathcal{M}^{\bullet (K)} > n\right) \le Cn^{-(M + 1)/2},
\end{equation*}
which is enough to conclude since $M$ is arbitrary.
\end{proof} 

\begin{figure}[H]
    \centering
    \includegraphics[scale =0.25]{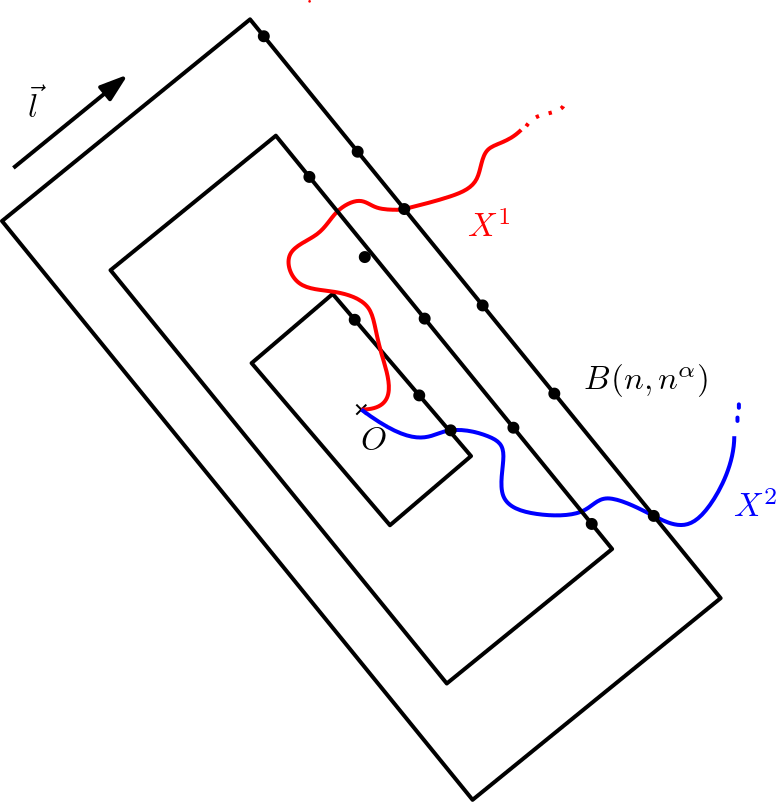}
    \caption{Black points represent the $K$-open points. We look at the successive attempts to reach a joint $K$-open ladder point when the walks exit boxes.}
    \label{fig:figK_open_points}
\end{figure}
\begin{proof}[Proof of Lemma~\ref{LemmaSuccessiveAttepts}]
For simplicity let us write $R(n k) = R^{(K)}(n k)$ in the proof. The argument is similar to the one of \cite[Lemma 6.4]{Frib_conduc}. Fix two arbitrary starting points $U_1,U_2 \in \mathcal{U}$ as in the statement. Let us recall that $\mathcal{B}_{i}(L, L') = \mathcal{B}_{X_{0}^{i}} (L, L')$ and the latter is defined in \eqref{eqn:TiltedBox}. For $i = 1, 2$ let us denote  
\begin{equation}\label{eqn:defWeirdBox}
    \widetilde{\mathcal{B}}^i_n \coloneqq \{\textnormal{largest box of the form }\mathcal{B}_i(m, m^\alpha) \subset \mathcal{H}^{-}(n), \,\, m \in \mathbb{R}_+\}.
\end{equation}
Note that $ \mathcal{B}_i(n-1, (n-1)^\alpha) \subseteq \widetilde{\mathcal{B}}^i_n \subseteq \mathcal{B}_i(n+1, (n+1)^\alpha)$.

Recall that $\bar{K}_{x}(n) $ denotes the $K$-$n$-closed connected component of $x$ defined just before \eqref{event_B}. We introduce 
\begin{equation}\label{eqn:DefinitionK(n)}
    \mathcal{K}^i(n) \coloneqq \bar{K}_{X^{i}_{T^i_{\partial_{\mathrm{in}}^+ \widetilde{\mathcal{B}}^i_n}}, U_i} (n) \subseteq \widetilde{\mathcal{B}}^i_n, \quad i \in \{1, 2\}, \quad \text{and} \quad \mathcal{K}(n) = \mathcal{K}^1(n) \cup \mathcal{K}^2(n).
\end{equation}
In case $T^i_{\partial_{\mathrm{in}}^+ \widetilde{\mathcal{B}}^i_n} =  + \infty$ we set $\mathcal{K}^i(n) = \varnothing$ and $\partial\mathcal{K}^i(n) = \varnothing$. We recall the definition of $A(L)$, $B(L)$ in \eqref{EventExitRightBoundary},\eqref{event_B}. We slightly modify them and define,
\begin{equation*}
    A^{i}(n) = \left\{ T^i_{\partial \widetilde{\mathcal{B}}^i_n} \ge T^i_{\partial^+_{\mathrm{in}} \widetilde{\mathcal{B}}^i_n}\right\}, \quad \mathrm{and} \quad A(n) \coloneqq A^{1}(n) \cap A^{2}(n),
\end{equation*}
\begin{equation*}
B(n) = \left\{ \textnormal{for all } x \in \partial_{\mathrm{in}}^+ \widetilde{\mathcal{B}}^1_n\cup \partial_{\mathrm{in}}^+ \widetilde{\mathcal{B}}^2_n, \textnormal{ we have } |\bar{K}_{x, U_1}(n) \cup \bar{K}_{x, U_2}(n)| \le \log(n) \right\}.
\end{equation*}
Note that, here, with little abuse of the notation, we called the events with the same letters as the ones in \eqref{EventExitRightBoundary},\eqref{event_B}. However, the events are very similar, and observing that $\mathcal{B}_i(n-1, (n-1)^\alpha) \subseteq \widetilde{\mathcal{B}}^i_n \subseteq \mathcal{B}_i(n+1, (n+1)^\alpha)$ guarantees that the estimates derived in Lemma~\ref{LemmaA(n)} and Lemma~\ref{LemmaB(n)} are still valid for the events considered in this proof.
Moreover, we introduce the events: 
\begin{align*}
    &\bullet \hspace{0.3em}C(n) \coloneqq \left\{ \text{for all } x \in \partial \mathcal{K}(n)\cap \mathcal{H}^+(n), \text{ the vertex } x \text{ is } K\textnormal{-open}  \right\}. \\
    \vspace{0.5ex}
    &\bullet \hspace{0.3em} D^{i}(n)  \coloneqq \left\{ T^i_{\partial\mathcal{K}^{i}(n)} \circ \theta_{T^i_{\partial_{\mathrm{in}}^+ \widetilde{\mathcal{B}}^i_n}} = T^i_{\partial\mathcal{K}^{i}(n) \cap \mathcal{H}^+(n)} \circ \theta_{T^i_{\partial_{\mathrm{in}}^+ \widetilde{\mathcal{B}}^i_n}} \right\}. \\
    \vspace{0.5ex}
    &\bullet \hspace{0.3em} E^{i}(n)  \coloneqq \left\{ X^i_{T^i_{\partial\widetilde{\mathcal{B}}^i_n} + 2} =  X^i_{T^i_{\partial\widetilde{\mathcal{B}}^i_n} + 1} + e_1 =  X^i_{T^i_{\partial\widetilde{\mathcal{B}}^i_n}} + 2 e_1 \textnormal{ and }  X^i_{T^i_{\partial\widetilde{\mathcal{B}}^i_n} + 2},  X^i_{T^i_{\partial\widetilde{\mathcal{B}}^i_n} + 1} \textnormal{ are $K$-open}\right\}.
\end{align*}
Furthermore, let $D(n) = D^{1}(n) \cap D^{2}(n)$ and $E(n) = E^{1}(n) \cap E^{2}(n)$. Let us consider the event $A(n) \cap C(n) \cap D(n) \cap E(n)$. It is possible to observe that on $A(n)$, for $i = 1, 2$, we have $T^i_{\partial \widetilde{\mathcal{B}}^i_n} \ge T^i_{\partial^+_{\mathrm{in}} \widetilde{\mathcal{B}}^i_n}$ so that
    \begin{equation*}
        T^i_{\partial^+ \widetilde{\mathcal{B}}^i_n} \ge T^i_{\partial \mathcal{K}^{i}(n)} \circ \theta_{T^i_{\partial^+_{\mathrm{in}} \widetilde{\mathcal{B}}^i_n}} +  T^i_{\partial^+_{\mathrm{in}} \widetilde{\mathcal{B}}^i_n}.
    \end{equation*}
Then, on $A(n) \cap C(n) \cap D(n) \cap E(n)$, by the definition of the events, for $i \in \{1,2\}$, $X^i_{T^i_{\partial \widetilde{\mathcal{B}}^i_n}}$ is a new maximum in the direction $\vec{\ell}$ which is also $K$-open and the next two steps are in the direction $e_1$ on $K$-open points, so that
\begin{equation}\label{JointLadderDiffEvents}
    \{A(n) \cap C(n) \cap D(n) \cap E(n)\} \subset \{ \mathcal{M} \le n + 3 \}.
\end{equation}
Applying standard union bounds we get the following inequality
\begin{align}
    \Pf_{U_1,U_2}\left( R(kn) \right) &\le \Pf_{U_1,U_2}\left( A(kn)^c \right) + \Pf_{U_1,U_2}\left( B(kn)^c \right)   \label{ProofFri6.4Final2}\\
&+ \Pf_{U_1,U_2}\left( R((k-1)n), A(kn), B(kn), C(kn)^c \right)   \label{ProofFri6.4ThreeSets}  \\
    & +\Pf_{U_1,U_2}\left( R((k-1)n), A(kn), B(kn), C(kn), D(kn)^c \right)  \label{ProofFri6.4FourSets}\\
    &+\Pf_{U_1,U_2}\left( R((k-1)n), A(kn), B(kn), C(kn), D(kn), E(kn)^c \right) \label{ProofFri6.4FiveSets}.
\end{align}
The first two terms in \eqref{ProofFri6.4Final2} are controlled, respectively by Lemma~\ref{LemmaA(n)} and Lemma~\ref{LemmaB(n)}. In order to finish the proof we need to control the terms in \eqref{ProofFri6.4ThreeSets}, \eqref{ProofFri6.4FourSets} and \eqref{ProofFri6.4FiveSets}.

\vspace{1ex}
\textit{Step 1: Control of the term \eqref{ProofFri6.4ThreeSets}.} Recall the definition of $\mathcal{K}(kn)$ in \eqref{eqn:DefinitionK(n)}. For $k \le n$, on the event $A(kn) \cap B(kn)$ it holds that
\begin{equation*}
    |\mathcal{K}(kn)| \le 4 \log(n),
\end{equation*}
since for both $i = 1, 2$ it holds that $|\mathcal{K}^i(kn)| \le 2 \log(n)$. We deduce that the term $\Pf_{U_1,U_2}( R((k-1)n), A(kn), B(kn), C(kn)^c )$ can be rewritten as
\begin{align*}
    &\sum_{i =1, 2} \sum_{\substack{F^i \subset \mathbb{Z}^d, \\ |F^i|\le 2 \log(n)}} \Pf_{U_1,U_2}\left( R((k-1)n), A(kn), B(kn), C(kn)^c, \mathcal{K}(kn) = F^1 \cup F^2 \right) \\
    & = \sum_{\substack{i =1, 2, \\F^i \subset \mathbb{Z}^d, \\ |F^i|\le 2 \log(n)}} \mathbf{E} \left[ P_{U_1,U_2}^\omega \left( R((k-1)n), A(kn), B(kn), \mathcal{K}(kn) = F^1 \cup F^2 \right) \mathds{1}_{\substack{\{ \text{some }x \in \partial (F^1 \cup F^2) \cap \mathcal{H}^+(kn)  \\ \text{ is } K\text{-closed} \}}} \right].
\end{align*}
We may notice that:
\begin{enumerate}
    \item $P_{U_1,U_2}^\omega \left( R((k-1)n), A(kn), B(kn), \mathcal{K}(kn) = F^1 \cup F^2 \right)$ is a $\mathbf{P}$-random variable measurable with respect to\ the sigma algebra $\sigma\left(c_{*}([x, y]), \,\, x, y \not \in \mathcal{H}^+(kn) \right)$.
    \item The event $\left\{ \text{some }x \in \partial(F^1 \cup F^2) \cap \mathcal{H}^+(kn) \text{ is closed} \right\}$ is measurable with respect to the sigma algebra $\sigma\left(c_{*}([x, y]), \,\, x \in \mathcal{H}^+(kn)\right)$.
\end{enumerate}
We have $\mathbf{P}$-independence between these two random objects. Hence, applying standard manipulations, the last sum can be written as
\begin{align*}
 \sum_{i =1, 2} \sum_{\substack{F^i \subset \mathbb{Z}^d,\\ |F^i|\le 2 \log(n)}} &\Pf_{U_1,U_2}\left( R((k-1)n), A(kn), B(kn), \mathcal{K}(kn) = F^1 \cup F^2 \right) \\& \times \left( 1 - \mathbf{P}\left( \text{all }x \in \partial(F^1 \cup F^2) \cap \mathcal{H}^+(kn) \text{ are } K\text{-open} \right) \right).
\end{align*}
Since $|F^{i}| \leq 2\log(n)$ and $|F^1 \cup F^2| \le |F^1| + |F^2|$ we have
\begin{align*}
    \mathbf{P}\left( \text{all }x \in \partial(F^1 \cup F^2) \cap \mathcal{H}^+(kn) \text{ are } K\text{-open}  \right) & \ge \left(1 - \varepsilon(K)\right)^{2d |F^1 \cup F^2|}\\
    & \ge \left(1 - \varepsilon(K)\right)^{8d\log(n)} \\
    & = n^{8d \log(1 - \varepsilon(K))},
\end{align*}
where $\varepsilon(K) \coloneqq \mathbf{P}\left( c_* \not \in [1/K, K] \right) $. For any $\varepsilon' > 0$ it is possible to fix $K$ large enough such that $8d \log(1 - \varepsilon(K)) \ge - \varepsilon'$. Then we have 
\begin{align*}
    \displaystyle \Pf_{U_1,U_2}\left( R((k-1)n), A(kn), B(kn), C(kn)^c \right) \le ( 1 - n^{-\varepsilon'}) \Pf_{U_1,U_2}\left( R((k-1)n), A(kn), B(kn)\right),
\end{align*}
which implies
\begin{align}
    n^{-\varepsilon'} \Pf_{U_1,U_2}\left( R((k-1)n), A(kn), B(kn)\right) \le \Pf_{U_1,U_2}\left( R((k-1)n), A(kn), B(kn), C(kn) \right). \label{Proof6.4FriPoly}
\end{align}

\vspace{1ex}
\textit{Step 2: Control of the term \eqref{ProofFri6.4FourSets}.} Notice that, on the event $A(n)$, we have that, for $i = 1, 2$, $X^{i}$ reaches $\partial^+_{\text{in}} \widetilde{\mathcal{B}}^i_{kn}$ before exiting $\widetilde{\mathcal{B}}^i_{kn}$ and by definition of $\mathcal{K}^{i}(kn)$, $X^{i}_{T^i_{\partial^{+}_{\text{in}} \widetilde{\mathcal{B}}^i_{kn}}} \in \mathcal{K}^{i}(kn)$. Moreover, on $B(kn)$ we also have that $|\mathcal{K}^{i}(kn)| \le 2 \log(n)$. Hence,
\begin{align*}
    &\Pf_{U_1,U_2}\left( R((k-1)n), A(kn), B(kn), C(kn), D(kn)^c \right) \\ 
    &\le \sum_{i =1, 2}\sum_{z^i} \sum_{F^i}  \Pf_{U_1,U_2}\left( R((k-1)n), A(kn), B(kn), \bigcap_{i = 1, 2} \{ X^{i}_{T^i_{\partial^{+}_{\text{in}} \widetilde{\mathcal{B}}_{kn}}} = z^i,  \mathcal{K}^{i}(kn) = F^i \}, C(kn), D(kn)^c \right),
\end{align*}
where $\displaystyle \sum_{z^i} \sum_{F^i}$ is short for $\displaystyle \sum_{z^i \in \partial^+_{\text{in}} \widetilde{\mathcal{B}}^i_{kn}} \sum_{\substack{F^i \subset \mathbb{Z}^d \\ |F^i| \le 2 \log(n)\\ z^i \in F^i}}$. 

For $\omega$ fixed, the events $\displaystyle R((k-1) n ), A(kn), B(kn), \bigcap_{i = 1, 2} \left\{ X^{i}_{T^i_{\partial^{+}_{\text{in}} \widetilde{\mathcal{B}}_{kn}}} = z^i,  \mathcal{K}^{i}(kn) = F^i \right\}$ are $P^\omega$-measurable with respect to $\sigma\Big(\{X^1_j, j \le T^1_{\partial^{+}_{\text{in}} \widetilde{\mathcal{B}}^1_{kn}}\},  \{X^2_k, k \le T^2_{\partial^{+}_{\text{in}} \widetilde{\mathcal{B}}^2_{kn}} \}\Big)$. Thus, by Lemma~\ref{Lemma2StoppingTimes}, that generalises the strong Markov property to two independent chains, and the quenched independence of the two walks, we can write
\begin{equation}\label{EquationD(n)^c}
    \begin{split}
    &\Pf_{U_1,U_2}\left( R((k-1)n), A(kn), B(kn), C(kn), D(kn)^c \right) \\ 
    &\le \sum_{i =1, 2} \sum_{z^i} \sum_{F^i} \mathbf{E}\bigg[ P_{U_1,U_2}^\omega \Big( R((k-1)n), A(kn), B(kn), \bigcap_{i = 1, 2} \{ X^{i}_{T^i_{\partial_{\text{in}} \widetilde{\mathcal{B}}^i_{kn}}} = z^i,  \mathcal{K}^{i}(kn) = F^i \} \Big) \\
    & P_{z^1, z^2}^\omega\left(T^1_{\partial F^1} < T^1_{\partial F^1 \cap \mathcal{H}^+(kn)} \textnormal{ or } T^2_{\partial F^2} < T^2_{\partial F^2 \cap \mathcal{H}^+(kn)}\right) \mathds{1}_{\{\text{all } x \in \partial (F^1 \cup F^2) \cap \mathcal{H}^+(kn) \text{ are } K\text{-open} \}} \bigg].  
    \end{split}
\end{equation}
For both $F^1$ and $F^2$ we observe, by the definition of $\mathcal{K}(kn)$ and the fact that we are considering the event $\{\text{all } x \in \partial (F^1 \cup F^2) \cap \mathcal{H}^+(kn) \text{ are } K\text{-open} \}$, that all the edges $e \in \partial_E (F^1 \cup F^2)$ are $K$-normal. As a remark, note that $F^1$ and $F^2$ are either the same set or disjoint sets, hence we get that $\partial_E F^1$ and $\partial_E F^2$ are two (possibly one) sets of normal edges. We obtain a bound for $P_{z^1}^\omega(T^1_{\partial F^1} \ge T^1_{\partial F^1 \cap \mathcal{H}^+(kn)})$, the case $P_{z^2}^\omega(T^2_{\partial F^2} \ge T^2_{\partial F^2 \cap \mathcal{H}^+(kn)})$ can be treated by following the same reasoning. For any $z^1 \in F^1 \cap \partial^+_{\text{in}} \widetilde{\mathcal{B}}^1_{kn}$, there exists a neighbour $y^1 \in \mathcal{H}^+(kn)$. Using the fact that $F^1 \subset \widetilde{\mathcal{B}}^1_{kn}$ we have that, for any $x \in \partial F^1$, $(y^1 - x) \cdot \vec{\ell} \ge - 1$ and $(z^1 - x) \cdot \vec{\ell} \ge - 2$. Using this observation and the definition of $K$-normal edges, we get
\begin{equation*}
    c^\omega(e) \le K^2 e^{3 \lambda} c^{\omega}([z^1, y^1]), \quad \textnormal{for } e \in \partial_E F^1.
\end{equation*}
Thus, we can apply Lemma~\ref{LemmaBoundaryHitting} with $G= F^1$ (resp.\ $F^2$) and obtain, by quenched independence,
\begin{align*}
    P_{z^1, z^2}^\omega&\left(T^1_{\partial F^1} < T^1_{\partial F^1 \cap \mathcal{H}^+(kn)} \textnormal{ or } T^2_{\partial F^2} < T^2_{\partial F^2 \cap \mathcal{H}^+(kn)}\right) \\ &\le P_{z^1, z^2}^\omega\left(T^1_{\partial F^1} < T^1_{y^1} \textnormal{ or } T^2_{\partial F^2} < T^2_{y^2}\right) \\ 
    & = \left( 1 - P_{z^1}^\omega\left( T^1_{y^1} \le T^1_{\partial F^1}\right) P_{z^2}^\omega\left( T^2_{y^2} \le T^2_{\partial F^2}\right) \right) \\
    &\le (1 - c\log(n)^{-2}),
\end{align*}
using $|F^i| \le 2 \log(n)$. By plugging it in Equation~\eqref{EquationD(n)^c} we get
\begin{align*}
    \Pf_{U_1,U_2}&\left( R((k-1)n), A(kn), B(kn), C(kn), D(kn)^c \right) \\ 
    &\le \sum_{i =1, 2} \sum_{z^i} \sum_{F^i} \mathbf{E}\bigg[ P_{U_1,U_2}^\omega \bigg( R((k-1)n), A(kn), B(kn), \bigcap_{i = 1, 2} \Big\{ X^{i}_{T^i_{\partial^{+}_{\text{in}} \widetilde{\mathcal{B}}_{kn}}} = z^i,  \mathcal{K}^{i}(kn) = F^i \Big\} \bigg) \\
    & \times \mathds{1}_{\{ x \in \partial (F^1 \cup F^2) \cap \mathcal{H}^+(kn) \text{ is } K\text{-open} \}} \bigg] (1 - c\log(n)^{-2})\\
    &\le (1 - c\log(n)^{-2}) \Pf_{U_1,U_2}\left( R((k - 1)n ), A(n), B(n), C(n) \right).
\end{align*}
Which readily implies that, for some positive constant $c>0$,
\begin{align}
    c\log(n)^{-2} & \Pf_{U_1,U_2}\left( R((k - 1)n ), A(n), B(n), C(n) \right) \nonumber \\ 
    &\le \Pf_{U_1,U_2}\left( R((k-1)n), A(kn), B(kn), C(kn), D(kn) \right). \label{Proof6.4FriLog}
\end{align}

\vspace{1ex}
\textit{Step 3: Control of the term \eqref{ProofFri6.4FiveSets}.} On the event $A(n)\cap B(n) \cap C(n) \cap D(n)$ we know that both $X^1_{T^1_{\partial \widetilde{\mathcal{B}}^1_{kn}}} \in \partial^+ \widetilde{\mathcal{B}}^1_{kn}$ and $X^2_{T^2_{\partial \widetilde{\mathcal{B}}^2_{kn}}} \in \partial^+ \widetilde{\mathcal{B}}^2_{kn}$ are $K$-open new maxima in the direction $\vec{\ell}$. Moreover, the event $E(kn)$ does not depend on what happens in the ``double box'' $\widetilde{\mathcal{B}}^1_{kn} \cup \widetilde{\mathcal{B}}^2_{kn}$. Let $x, y \in \mathbb{Z}^d$ be fixed, we say that a vertex $z \in \mathbb{Z}^d$ is $x, y$-open if all the edges in the set $\mathcal{E}_{z} \backslash (\mathcal{E}_{x} \cup \mathcal{E}_{y})$ are $K$-normal, if a vertex is not $x, y$-open then is $x, y$-closed. Let us also define
\begin{equation*}
    R'((k - 1)n) = R((k - 1)n) \cap A(n)\cap B(n) \cap C(n) \cap D(n).
\end{equation*}
Recall (see \eqref{JointLadderDiffEvents} and considerations above) that on $R'((k - 1)n)$ both walks are at new $K$-open maxima when they hit $\mathcal{H}^{+}(kn)$. We can write
\begin{align*}
    &\Pf_{U_1,U_2}\left(R((k - 1)n), A(kn), B(kn), C(kn), D(kn), E(kn)^c \right) \\
    &\le \sum_{i = 1, 2} \sum_{x^i \in \partial^+ \widetilde{\mathcal{B}}^i_{kn}} \Pf_{U_1,U_2}\left( R'((k - 1)n), X^1_{T^1_{\partial \widetilde{\mathcal{B}}^1_{kn}}} = x^1, X^2_{T^2_{\partial \widetilde{\mathcal{B}}^2_{kn}}} = x^2, \{x^1, x^2\} \textnormal{ are }K\textnormal{-open}, E(kn)^c \right)
\end{align*}
\begin{align*}
    &\le \sum_{i = 1, 2} \sum_{x^i \in \partial^+ \widetilde{\mathcal{B}}^i_{kn}} \mathbf{E}\bigg[ P_{U_1,U_2}^\omega \left( R'((k - 1)n), X^1_{T^1_{\partial \widetilde{\mathcal{B}}^1_{kn}}} = x^1, X^2_{T^2_{\partial \widetilde{\mathcal{B}}^2_{kn}}} = x^2 \right) \mathds{1}_{\{\{x^1, x^2\} \textnormal{ are }K\textnormal{-open}\}} \\
    &\big( \mathds{1}_{\{\textnormal{one vertex in }\{x^1 + e_1, x^1 + 2e_1, x^2 + e_1, x^2 + 2e_1\} \textnormal{ is }x^1, x^2\textnormal{-closed}\}} + \mathds{1}_{\{\{x^1 + e_1, x^1 + 2e_1, x^2 + e_1, x^2 + 2e_1\} \textnormal{ are }x^1, x^2\textnormal{-open}\}} \\
    &\times P^\omega_{x^1, x^2}\left( X^1_1 \ne x^1 + e_1 \text{ or } X^1_2 \ne x^1 + 2e_1 \text{ or } X^2_1 \ne x^2 + e_1 \text{ or } X^2_2 \ne x^2 + 2e_1 \right) \big) \bigg].
\end{align*}
On the event $\{\{x^1 + e_1, x^1 + 2e_1, x^2 + e_1, x^2 + 2e_1\} \textnormal{ are }x^1, x^2\textnormal{-open}\} \cap \{ \{x^1, x^2\} \textnormal{ are }K\textnormal{-open}\}$ one has that $P^\omega_{x^i}\left( X^i_1 = x^i + e_1, X^i_2 = x^i + 2e_1 \right) \ge c > 0$ for both $i = 1, 2$. Then
\begin{align*}
    \Pf&_{U_1,U_2}\left(R((k - 1)n), A(kn), B(kn), C(kn), D(kn), E(kn)^c \right) \\
    &\le \sum_{i = 1, 2} \sum_{x^i \in \partial^+ \widetilde{\mathcal{B}}^i_{kn}} \mathbf{E}\bigg[ P_{U_1,U_2}^\omega \left( R'((k - 1)n), X^1_{T^1_{\partial \widetilde{\mathcal{B}}^1_{kn}}} = x^1, X^2_{T^2_{\partial \widetilde{\mathcal{B}}^2_{kn}}} = x^2 \right) \mathds{1}_{\{\{x^1, x^2\} \textnormal{ are }K\textnormal{-open}\}} \\
    & \big( \mathds{1}_{\{\textnormal{one vertex in }\{x^1 + e_1, x^1 + 2e_1, x^2 + e_1, x^2 + 2e_1\} \textnormal{ is }x^1, x^2\textnormal{-closed}\}} \\ &+ (1-c^2)\mathds{1}_{\{\{x^1 + e_1, x^1 + 2e_1, x^2 + e_1, x^2 + 2e_1\} \textnormal{ are }x^1, x^2\textnormal{-open}\}} \big) \bigg].
\end{align*}
The $\mathbf{P}$-random variables $P_{U_1,U_2}^\omega \Big( R'((k - 1)n), X^1_{T_{\partial \widetilde{\mathcal{B}}^1_{kn}}} = x^1, X^2_{T_{\partial \widetilde{\mathcal{B}}^2_{kn}}} = x^2 \Big)$ and $\mathds{1}_{\{\{x^1, x^2\} \textnormal{ are }K\textnormal{-open}\}}$ are measurable with respect to $\sigma( e \in E(\widetilde{\mathcal{B}}^1_{kn} \cup \widetilde{\mathcal{B}}^2_{kn})\cup \mathcal{E}_{x^1} \cup \mathcal{E}_{x^2})$, we recall that $\mathcal{E}_x, x \in \mathbb{Z}^d$ is defined in \eqref{EquationIncidentEdges}. On the other hand the events $\{\textnormal{one vertex in }\{x^1 + e_1, x^1 + 2e_1, x^2 + e_1, x^2 + 2e_1\} \textnormal{ is }x^1, x^2\textnormal{-closed}\}$ and $\{\{x^1 + e_1, x^1 + 2e_1, x^2 + e_1, x^2 + 2e_1\} \textnormal{ are }x^1, x^2\textnormal{-open}\}$ are measurable with respect to $\sigma( e \not \in E(\widetilde{\mathcal{B}}^1_{kn} \cup \widetilde{\mathcal{B}}^2_{kn})\cup \mathcal{E}_{x^1} \cup \mathcal{E}_{x^2} )$, due to the fact that we are summing over $x^i \in \partial^+ \widetilde{\mathcal{B}}^i_{kn}$. This implies $\mathbf{P}$-independence between these two groups of random variables. This yields
\begin{align*}
    \Pf_{U_1,U_2}&\left(R((k - 1)n), A(kn), B(kn), C(kn), D(kn), E(kn)^c \right) \\
    & \le \sum_{i = 1, 2} \sum_{x^i \in \partial^+ \widetilde{\mathcal{B}}^i_{kn}} \mathbf{E}\left[ P_{U_1,U_2}^\omega \left( R'((k - 1)n), X^1_{T^1_{\partial \widetilde{\mathcal{B}}^1_{kn}}} = x^1, X^2_{T^2_{\partial \widetilde{\mathcal{B}}^2_{kn}}} = x^2 \right) \right]  \\ 
    & \times \bigg( \mathbf{P}\left( \textnormal{one vertex in }\{x^1 + e_1, x^1 + 2e_1, x^2 + e_1, x^2 + 2e_1\} \textnormal{ is }x^1, x^2\textnormal{-closed} \right) +  \\  
    &+ (1 - c') \mathbf{P}\left( \{x^1 + e_1, x^1 + 2e_1, x^2 + e_1, x^2 + 2e_1\} \textnormal{ are }x^1, x^2\textnormal{-open} \right) \bigg) \\
    & \le (1 - c') \Pf_{U_1,U_2}\left(R((k - 1)n), A(kn), B(kn), C(kn), D(kn) \right),
\end{align*}
since $\mathbf{P}\left( \{x^1 + e_1, x^1 + 2e_1, x^2 + e_1, x^2 + 2e_1\} \textnormal{ are }x^1, x^2\textnormal{-open} \right) > 0$ and its value does not depend on $x^1,x^2$. This implies that there exists $c>0$ such that 
\begin{align}
    c \Pf_{U_1,U_2}&\left(R((k - 1)n), A(kn), B(kn), C(kn), D(kn) \right) \nonumber \\  &\le \Pf_{U_1,U_2}\left(R((k - 1)n), A(kn), B(kn), C(kn), D(kn), E(kn)\right). \label{Proof6.4FriConst}
\end{align}

\vspace{1ex}
\textit{Step 4: Conclusion.} For any $\varepsilon' > 0$, $M > 0$, one can find a $K$ chosen large enough such that, putting together \eqref{ProofFri6.4Final2} with Lemma~\ref{LemmaA(n)} and Lemma~\ref{LemmaB(n)} and the estimates obtained in \eqref{Proof6.4FriPoly}, \eqref{Proof6.4FriLog} and \eqref{Proof6.4FriConst}, implies that, for all $k \in \{2, \dots, n \}$
\begin{equation*}
    \Pf_{U_1,U_2}\left( R(kn) \right) \le \Pf_{U_1,U_2}\left( R((k-1)n) \right)\left( 1 - c\log(n)^{-2} n^{-\varepsilon'} \right) + n^{-M},
\end{equation*}
which implies the result since $\varepsilon'$ is arbitrary.
\end{proof}
For the construction of joint regeneration times, we need that joint $K$-open ladder points are close to each other in the direction $\vec{\ell}$. We recall that $\theta_{s, t}$ denotes, in the context of two walks, the canonical time shift of the first walk by $s$ units of time and the second one by $t$ units of time. For any $R > 0$, we introduce the variable
\begin{align*}
    \mathcal{M}^{\bullet(K)}_R = \mathcal{M}^{\bullet}_R \coloneqq \mathcal{M}^{\bullet(K)} \circ \theta_{T^1_R, T^2_R} - R,
\end{align*}
where we recall that $\mathcal{M}^{\bullet (K)}$ is defined in \eqref{definition_M}. We introduce the event
\begin{equation} \label{DefinitionUpsilon}
    \Upsilon(n) = \bigcap_{R = 0}^{n}\left\{\mathcal{M^\bullet}_R\leq n^{\frac{1}{2}}\right\}.
\end{equation}
On $\Upsilon(n)$, before reaching $\mathcal{H}^{+}(n)$, the successive joint $K$-open ladder points are always at distance less than $n^{\frac{1}{2}}$ in the direction $\vec{\ell}$. 
\begin{proposition} \label{PropFrequentRegLevels}
For any $M > 0$ there exists $K_0 > 1$ such that for $K \ge K_0$ we have 
\begin{equation*}
    \sup_{U_1, U_2 \in \mathcal{U}}\Pf_{U_1, U_2}(\Upsilon(n)^{c}) \leq Cn^{-M},
\end{equation*}
where $C$ only depends on $\lambda,\vec{\ell},d,K$ and $\mathcal{U}$ is defined in \eqref{eqn:USet}.
\end{proposition}
\begin{proof}
The proof is similar to the proof of \cite[Lemma 6.6]{Frib_conduc}. Let us consider the events, for $j \in \{1, \dots, n\}$
\begin{equation*}
    \Upsilon_j(n) = \bigcap_{R = 0}^{j - 1}\left\{\mathcal{M^\bullet}_R\leq n^{\frac{1}{2}}\right\} \quad \textnormal{and} \quad \Theta_j(n) = \left\{\mathcal{M^\bullet}_j > n^{\frac{1}{2}}\right\}.
\end{equation*}
We can split the event $\Upsilon(n)^{c}$ over the minimal $j \in \{ 1, \dots, n\}$ such that $\Theta_j(n)$ is true and get, by Lemma~\ref{LemmaA(n)}
\begin{equation} \label{EqnProofFrequentRegLevels}
    \Pf_{U_1, U_2}(\Upsilon(n)^{c}) \le \sum_{j = 0}^{n} \Pf_{U_1, U_2} \left( \Upsilon_j(n), \Theta_j(n), T^1_{\partial \widetilde{\mathcal{B}}^1_n} = T^1_{\partial^+ \widetilde{\mathcal{B}}^1_{n}} , T^2_{\partial \widetilde{\mathcal{B}}^2_n} = T^2_{\partial^+ \widetilde{\mathcal{B}}^2_n} \right) + C e^{-c n}.
\end{equation}
On the event $\Upsilon_j(n) \cap \Theta_j(n) \cap \{ T^1_{\partial \widetilde{\mathcal{B}}^1_n} = T^1_{\partial^+ \widetilde{\mathcal{B}}^1_n} \} \cap \{ T^2_{\partial \widetilde{\mathcal{B}}^2_n} = T^2_{\partial^+ \widetilde{\mathcal{B}}^2_n} \}$ one is sure that $X^i_{T_{j}^{i}} \in \widetilde{\mathcal{B}}^i_n$ for $i = 1, 2$ then by union bound argument
\begin{align} \label{ladder_point_shifted}
    \nonumber \Pf_{U_1, U_2} &\left( \Upsilon_j(n), \Theta_j(n), T^1_{\partial \widetilde{\mathcal{B}}^1_n} = T^1_{\partial^+ \widetilde{\mathcal{B}}^1_n}, T^2_{\partial \widetilde{\mathcal{B}}^2_n} = T^2_{\partial^+ \widetilde{\mathcal{B}}^2_n} \right) \\
    &\nonumber \le \sum_{x_1, x_2 \in \widetilde{\mathcal{B}}^1_n, \widetilde{\mathcal{B}}^2_n} \mathbf{E}\left[ P^{\omega}_{U_1, U_2} \left( X^1_{T^1_j} = x_1, X^2_{T^2_j} = x_2, \mathcal{M}_{j}^{\bullet (K)} > n^{1/2} \right) \right] \\
    &\le 2 |\mathcal{B}_{n+1}|^2 \sup_{x_1, x_2}\Pf_{x_1,x_2}\left(\mathcal{M}^{\bullet(K)} - j \ge n^{1/2}\right).
\end{align}
To get \eqref{ladder_point_shifted} we applied Lemma~\ref{Lemma2StoppingTimes} to the quenched law and the supremum is over all $x_1,x_2 \in \widetilde{\mathcal{B}}^1_n, \widetilde{\mathcal{B}}^2_n$ such that $x_1,x_2$ are compatible with $\{X^1_{T^1_j} = x_1, X^2_{T^2_j} = x_2\}$.
Now we conclude using the fact that
\begin{equation*}
    \Pf_{x_1,x_2}\left(\mathcal{M}^{\bullet(K)} - j \ge n^{1/2}\right) = \Pf_{0,x_2-x_1}\left(\mathcal{M}^{\bullet(K)} - j+ x_1\cdot \vec{\ell} \ge n^{1/2}\right),
\end{equation*}
by the translation invariance of the environment. Using the fact that $|x_1\cdot \vec{\ell} - j| < e_{1} \cdot \vec{\ell}$ and $|(x_2-x_1)\cdot \vec{\ell}| < e_{1} \cdot \vec{\ell}$, the assumptions of Proposition~\ref{PropositionRegLevelSmall} are verified. We conclude that
\begin{align*}
    \sup_{x_1, x_2}\Pf_{x_1,x_2}\left(\mathcal{M}^{\bullet(K)} - j \ge n^{1/2}\right) \le Cn^{-M},
\end{align*}
where $C$ only depends on $\lambda,\vec{\ell},d,K$. We conclude the domination of \eqref{ladder_point_shifted} using the fact that $|\mathcal{B}_{n+1}| \le Cn^{2d\alpha}$.
\end{proof}

\subsection{Enhanced walks}\label{enhanced_walk}

The definition of regeneration time as a ``cut point'' for the trajectory has the drawback that it does not guarantee independence. Indeed, the knowledge of the conductances incident to $x$ influences the walk at any time it reaches a neighbour of $x$ (even when $y \sim x$ is such that $y\cdot \vec{\ell}> x \cdot \vec{\ell}$). To obtain stronger independence properties, we introduce a variant in the construction of regeneration times. Such a task was already done in \cite{Kious_Frib} using a so-called ``enhanced-walk''. 

To make this work self-contained let us recall the definition of the enhanced walk. We defined the quenched law of our RWRC in \eqref{QuenchedLawDef}, for any environment $\omega$ we define the modification 
\begin{equation}\label{eqn:KtransitionsEnhanced}
   p_K^{\omega}(x,y) \coloneqq \frac{(c_{*}(x,y)\wedge K^{-1})e^{(x+y)\cdot \ell}}{\displaystyle\somme{ z \sim x}{}{(c_{*}(x,z)\lor K)e^{(x+z)\cdot \ell}}},
\end{equation}
and $0$ if $x$ and $y$ are not adjacent. Note that, for any $x,y \in \mathbb{Z}^d$, we have that $p^{\omega}_{K}(x,y) \le p^{\omega}(x,y)$ (and the probabilities $\{p^{\omega}_{K}(x,y)\}_{y \sim x}$ do not necessarily sum to $1$), moreover the family $\{p^{\omega}_{K}(x,y)\}_{y \sim x}$ is deterministic if the point $x$ is $K$-open.

For any $(x,z) \in \mathbb{Z}^d \times \{0,1\}$ we define the Markov chain $(\tilde{X}_{n})_{n \geq 0} = (X_{n}, Z_{n})_{n \geq 0}$ to have transitions probabilities $p^{\omega}((x_1,z_1),(x_2,z_2))$ defined by
\begin{align*}
    &1. \,\, \tilde{X}_{0} = (x,z), \quad P^{\omega}_{(x,z)}\text{-a.s},\\
    &2.\,\, p^{\omega}((x,z),(y,1)) = p_K^{\omega}(x,y),\\
    &3.\,\, p^{\omega}((x,z),(y,0)) = p^{\omega}(x,y)  - p_K^{\omega}(x,y),
\end{align*}
and denote its quenched law $P^\omega_{(x,z)}$. This definition ensures that the first coordinate has the same law as our random walk $(X_n)_n$ in the environment $\omega$. We refer to the annealed law of the enhanced random walk writing $\Pf_{(x,z)}$. We observe that the evolution of the first coordinate does not depend on the second one, thus we may simply write $X$ instead of $\tilde{X}$ and $P_{x}^{\omega}$ instead of $P^{\omega}_{(x,z)}$ when referring to the marginal of the first coordinate. Note that we can also define the quenched law of two independent enhanced random walks evolving in the same environment.
\begin{definition1}\label{def:2EnhancedWalks}
    For any environment $\omega$ and a starting configuration $(x_1,z_1),(x_2,z_2) \in \mathbb{Z}^d \times \{0,1\}$ we define 
    \begin{equation*}
        P^{\omega}_{\substack{(x_1,z_1)\\(x_2,z_2)}} (\cdot) \coloneqq P^{\omega}_{(x_1,z_1)} (\cdot) \otimes P^{\omega}_{(x_2,z_2)}(\cdot)
    \end{equation*}
    to be the quenched law of two independent enhanced random walks evolving in the same environment. We also denote by $\mathbb{P}_{\substack{(x_1,z_1)\\(x_2,z_2)}}$ the corresponding annealed law.
\end{definition1}

\subsection{Regeneration times of a single walk} \label{definition_reg_times}
We now recall the rigorous definition of the regeneration times introduced \cite{Kious_Frib}. This definition is essential, as we discussed in Section~\ref{SectionIntroSeries}, one of our main goals is to prove Theorem~\ref{TheoremVarianceClockProcess} which deals with quantities in which this definition plays an important role. We also highlight the fact that the definition of joint regeneration levels in Section~\ref{joint_reg_levels} is very similar to the definition of regeneration times. 

First, we introduce the time at which the single walk reaches a $K$-open ladder point,
\begin{align*}
    \mathcal{M}^{(K)} \coloneqq \inf\Big\{ &i \ge 2 \colon X_i \text{ is }K\text{-open, }X_j\cdot \vec{\ell} < X_{i-2} \cdot \vec{\ell} \text{ for any } j < i-2   \\& \text{ and } X_i = X_{i-1} + e_1 = X_{i-2} + 2e_1 \Big\}.
\end{align*}
This quantity is closely related to \eqref{definition_M}, the main difference is that while $\mathcal{M}^{\bullet}$ is a level, $\mathcal{M}^{(K)}$ is a time. We discuss the choice of working with levels and not times in Section~\ref{joint_reg_levels} (see Remark~\ref{remark:Space}). We also introduce
\begin{align*}
    & \mathbf{BACK} \coloneqq \inf \left \{ n \ge 1 \colon X_n \cdot \vec{\ell} \le X_0 \cdot \vec{\ell }\right \},\\
    & \mathbf{ORI} \coloneqq \inf \left\{ n \ge 1 \colon X_{n-1} \in \mathcal{V}_{X_0} \text{ and } Z_n = 0\right\},
\end{align*}
and finally 
\begin{equation}\label{eqn:DSingleRegeneration}
    D \coloneqq \mathbf{BACK} \wedge \mathbf{ORI}.
\end{equation}
The event $\{\mathbf{BACK} > n\}$ means that the walk $X$ does not backtrack below $X_0$ before time $n$. The event $\{\mathbf{ORI} > n\}$ is essential to get independence between the regeneration times. In particular, under $\{\mathbf{ORI} > n\}$, it is possible to build a coupling which shows that the conductances around the starting point do not affect the future of the walk if the starting point is $K$-open. That is why we will be interested in the event $\{D = +\infty\}$ to regenerate.

We define the sequence of the configuration dependent stopping times $(S_n)_{n \ge 0}$ and $(R_n)_{n \ge 0}$, together with $(M_n)_{n \ge 0}$, successive levels of the maximum of the walk in direction $\vec{\ell}$. Note that these quantities are not finite a priori.
\begin{align*}
    &\hspace{4em}S_0 = 0\text{,     }M_0 = X_{0} \cdot \vec{\ell} \text{  and }\\
    & \text{ for } n \ge 0\text{,  } S_{n+1} \coloneqq \mathcal{M}^{(K)} \circ \theta_{T_{\mathcal{H}^{+}(M_n)}} + T_{\mathcal{H}^{+}(M_n)},
\end{align*}
where 
\begin{align*}
    M_n \coloneqq \sup \{X_k \cdot \vec{\ell} \text{ with } 0 \le k \le R_n \},
\end{align*}
with 
\begin{align*}
    R_n \coloneqq S_n + D \circ \theta_{S_n}.
\end{align*}
Then, we define the first regeneration time as 
\begin{align*}
    \tau_1^{(K)} = S_N \quad \text{with} \quad N \coloneqq \inf\{n \ge 1 \colon S_n < +\infty \text{ and } M_n = +\infty\}.
\end{align*}
We highlight that $\{S_{n+1} < \infty\}$ almost surely on $\{R_{n} < \infty\}$. Conversely, on the event $\{R_{n} = \infty\}$, we get that $\{M_n = +\infty\}$ and the iteration stops, hence the quantity $N$ is well-defined.
\begin{remark1}
Most of the time we will forego the superscript $K$ in $\tau_1^{(K)}$ since it will be fixed to a finite but large enough value. 
\end{remark1}
Using \cite[Theorem 5.1]{Kious_Frib} we obtain that $\tau_1$ is $\Pf_0$-almost surely finite and we can define inductively the sequence of finite regeneration times as
\begin{align}\label{eqn:RegTimeSingleWalk}
    \tau_{k+1} = \tau_1 + \tau_{k}((X_{\tau_1+\cdot } - X_{\tau_1}),\omega( \cdot +X_{\tau_1})).
\end{align}
The key result \cite[Theorem 5.4]{Kious_Frib} states that the $\tau_k$ are $\Pf_0$-almost surely finite and that $(X_{\tau_{k}+\cdot} - X_{\tau_k})$ is $\Pf_0$-independent of $(X_i)_{0\le i \le \tau_k}$ and distributed as a random walk under $\Pf_0^K( \, \cdot \text{ | }D=+\infty)$.
We introduce the notation 
\begin{equation} \label{recall_chi}
    \chi_k \coloneqq \inf \left\{m \ge 0 \colon \{X_{j}-X_{\tau_k}\}_{\tau_k \le j \le \tau_{k+1}} \subset \mathcal{B}(m,m^{\alpha})\right\}.
\end{equation} 

In the next section we will generalize these notions to two walks. In this context, regeneration times will turn into regeneration levels and the independence property will become a Markov-type property.

\begin{remark1}
The reader should be careful in the following since we are going to define \emph{joint regeneration levels}. We will use the notation $\mathcal{D}^{^{\bullet}}$ which is not the same as $D$ and $D^{\otimes}$, which are the natural counterparts for a single walk or for two independent walks under the annealed measure.
\end{remark1}

\subsection{Enhanced walks and joint regeneration}

We consider two enhanced random walks $\tilde{X}^1$ and $\tilde{X}^2$ independent under the quenched law. We recall the notation set in Section~\ref{neighbourhood}. For $i \in \{1,2\}$ we introduce the random variables
\begin{align}
    &\mathbf{BACK}^i_{\le R} \coloneqq \inf \left \{ 0  < n < T^{i}_{\mathcal{H}^{+}_R} \colon X^i_n \cdot \vec{\ell} \leq X^{i}_0 \cdot \vec{\ell} \right\},\label{definition_BACK} \\
    & \mathbf{ORI}^i_{\le R} \coloneqq \inf \left \{ 0  < n < T^{i}_{\mathcal{H}^{+}_R} \colon X^i_{n-1} \in \mathcal{V}_{X^1_0} \cup \mathcal{V}_{X^2_0} \text{ and } Z^i_n = 0, \text{ or } X^i_{n} \in \mathcal{V}_{X^1_{0} - e_1} \cup \mathcal{V}_{X^2_{0} - e_1}  \right\}. \label{definition_ENV} 
\end{align}
Then, we introduce
\begin{equation}\label{definition_Di}
    \mathcal{D}^{\bullet i}_{\le R}=\mathbf{BACK}^i_{\le R} \wedge \mathbf{ORI}^i_{\le R}\quad \text{and} \quad \mathcal{D}^{\bullet i} \coloneqq \lim_{R \to +\infty} \mathcal{D}^{\bullet i}_{\le R}.
\end{equation}
Finally, we set
\begin{equation}\label{definition_D}
    \mathcal{D}^{\bullet} \coloneqq \mathcal{D}^{\bullet 1} \wedge \mathcal{D}^{\bullet 2}.
\end{equation}
Note that in \eqref{definition_BACK} and \eqref{definition_ENV} we assume the convention that $\inf \varnothing = + \infty$. Moreover, we observe that the function $R \mapsto \mathcal{D}^{\bullet i}_{\le R}$ is decreasing and is only taking two values at most: $\mathcal{D}^{\bullet i}_{\le R} \to +\infty$ as $R \to -\infty$ and, as $R \to +\infty$, either $\mathcal{D}^{\bullet i}_{\le R} \to \mathcal{D}^{\bullet i}_{\le R_{0}} < \infty$ if such $R_{0} \in (0, \infty)$ exists or $\mathcal{D}^{\bullet i}_{\le R} \to + \infty$. Hence, the limiting quantity $\mathcal{D}^{\bullet i} \in \mathbb{N} \cup \{+\infty\}$ is well-defined.

The notation $\mathbf{BACK}^i \coloneqq \lim_{R \to +\infty} \mathbf{BACK}^i_{\le R}$ refers to the first time the walk $X^i$ backtracks. The notation $\mathbf{ORI}^i \coloneqq \lim_{R \to +\infty} \mathbf{ORI}^i_{\le R}$ encapsulates information around the starting points in such a way that Proposition~\ref{omegaK} holds. The extra condition $ X^i_{n} \in \mathcal{V}_{X^1_{0} - e_1} \cup \mathcal{V}_{X^2_{0} - e_1} $ is needed to get the independence properties in Theorem~\ref{IndependenceJointReg}. It motivates the notation 
\begin{align}\label{definition_M_other}
    M^{i} \coloneqq \inf\left\{R \in \mathbb{R} \colon \mathcal{D}^{\bullet i}_{\le R} < +\infty \right\} \text{ and } M \coloneqq M^1 \wedge M^2.
\end{align} 
The random variable $M \in \mathbb{R}\cup \{+\infty\}$ has the following geometric interpretation (see Figure~\ref{illustation_Mk}): if one only sees the trajectories of the two walks until they reach $\mathcal{H}^{+}(R)$, then $M$ is the first level such that we can conclude that $\mathcal{D}^{\bullet} < +\infty$. We highlight the fact that $\{\mathcal{D}^{\bullet}=+\infty\} = \{M=+\infty\}$.
\begin{figure}[H] \label{illustation_Mk}
    \centering
    \includegraphics[scale = 0.4]{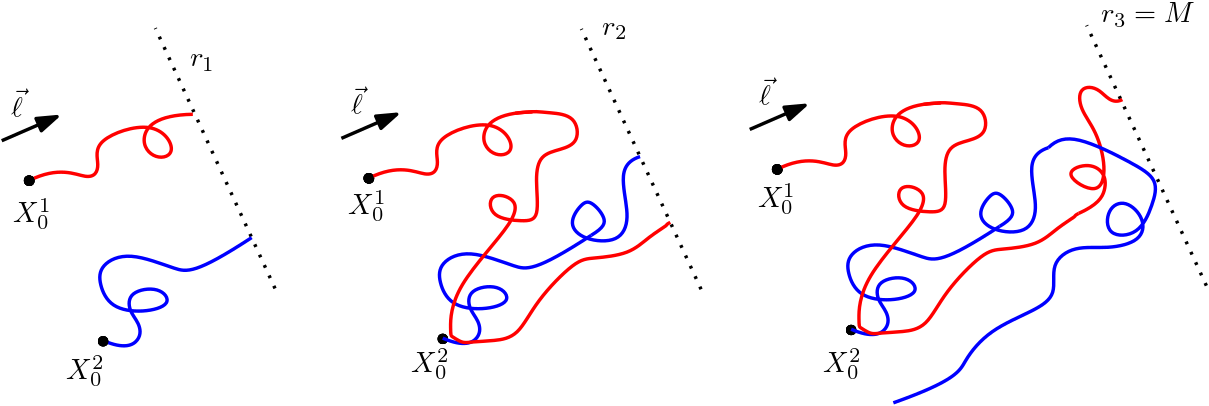}
    \caption{On the left, until they reach $\mathcal{H}^{+}_{r_1}$, the two walks did not backtrack, it follows that $M \ge r_1$. In the middle, the trajectory of $X^1$ has reached $\mathcal{V}_{X^2_0}$ but $\mathbf{ORI}^i_{\le r_2} = +\infty$, it follows that $M \ge r_2$. On the right, the trajectory of $X^2$ is backtracking below the starting point just after reaching level $r_3$, we thus conclude that $M = r_3$.}
    \label{fig:enter-label}
\end{figure}
\begin{remark1}
It may seem that the definition of the quantities with the $R$ could be avoided here. Indeed, they will be useful in Section~\ref{joint_reg_levels} when we will define joint regeneration levels.
\end{remark1}

We are interested in the case $\mathcal{D}^{\bullet} = +\infty$ and $X_0^1,X_0^2$ are $K$-open. In that case, the two walks never backtrack and they behave as if for all $e \in \mathcal{E}_{X_0^{1}} \cup \mathcal{E}_{X_0^{2}}$, $c_{*}(e) = K$. We remark that $\mathcal{D}^{\bullet 1}$ is not measurable with respect to $\sigma(X_0^{1},(\tilde{X}^{1}_{n})_{n \geq 1})$ but with respect to $\sigma(X_0^{1},X_0^{2},(\tilde{X}^{1}_{n})_{n \geq 1})$. 
\begin{remark1}
Note that under $P^{\omega}_{\substack{(x_1,z_1)\\(x_2,z_2)}}$, $\mathcal{D}^{\bullet 1}$ and $\mathcal{D}^{\bullet 2}$ are independent. Indeed, the variables $X^1_0$ and $X^2_0$ are almost surely constant under $P^{\omega}_{\substack{(x_1,z_1)\\(x_2,z_2)}}$ and the two quenched walks are independent of each other.
\end{remark1}

\begin{definition1}\label{def:OmegaK}
    For any environment $\omega$ and a starting configuration $(x_1,z_1),(x_2,z_2) \in \mathbb{Z}^d \times \{0,1\}$, we define $\omega_{K}$ to be the environment such that $c_{*}^{\omega_K}(e) = K$ if $e \in \mathcal{E}_{x_1}\cup \mathcal{E}_{x_2}$ while $c_{*}^{\omega_{K}}(e) = c_{*}^{\omega}(e)$ otherwise. In the environment $\omega_K$, we define $P^{\omega_{K}}_{\substack{(x_1,z_1)\\(x_2,z_2)}}$ to be the analogous of $P^{\omega}_{\substack{(x_1,z_1)\\(x_2,z_2)}}$ in the environment $\omega$ (see Definition~\ref{def:2EnhancedWalks}). Moreover, $p_{K}^{\omega_K}$ denotes the obvious analogous of \eqref{eqn:KtransitionsEnhanced} in the environment $\omega_{K}$ when the two special points $x_1, x_2$ are fixed. In the case of a single walk started at $x \in \mathbb{Z}$ with law $P^{\omega_{K}}_{(x,z)}$, $\omega_{K}$ fixes to the value $K$ the conductances $\{c_{*}(e)\}_{e \in \mathcal{E}_x}$.
\end{definition1}

\noindent To lighten the notation, for a vector $\mathbf{x} = (x_{1}, x_{2}) \in \mathbb{Z}^d \times \mathbb{Z}^d$ and $\mathbf{z} = (z_{1}, z_{2}) \in \{0,1\} \times \{0,1\}$, we may write
\begin{equation}\label{eqn:ShorterNotationJointLaw}
    P^{\omega}_{\substack{(x_1,z_1)\\(x_2,z_2)}} = P^{\omega}_{(\mathbf{x}, \mathbf{z})}.
\end{equation}
The same notation can be assumed for the corresponding expectation that becomes $E^{\omega}_{(\mathbf{x}, \mathbf{z})}$.

The following proposition shows that, on the event $ \{\mathcal{D}^{\bullet} = +\infty \}$, the exact values of the conductances incident to $K$-open starting points do not influence the future of the walk. 

\begin{proposition} \label{omegaK}Fix $(x_1,z_1),(x_2,z_2) \in \mathbb{Z}^d \times \{0,1\}$ and recall the notation set at \eqref{eqn:ShorterNotationJointLaw}. If $x_1$ and $x_2$ are both $K$-open points then, for any positive bounded $ \sigma \big(\tilde{X}^{1}_{n},\tilde{X}^{2}_{n}\big)$-measurable function $f(\cdot)$, we have the equality
\begin{equation*}
    E^{\omega }_{(\mathbf{x}, \mathbf{z})}\left[f(\tilde{X}^1_{\cdot},\tilde{X}^{2}_{\cdot})\indi{\mathcal{D}^{\bullet} = +\infty}\right] = E^{\omega_K }_{(\mathbf{x}, \mathbf{z})}\left[f(\tilde{X}^1_{\cdot},\tilde{X}^{2}_{\cdot})\indi{\mathcal{D}^{\bullet} = +\infty}\right].
\end{equation*}
\end{proposition}

\begin{proof}
Let us use the notation set in \eqref{eqn:ShorterNotationJointLaw}. The proof follows the lines of the proof of \cite[Proposition~5.1]{Kious_Frib} that uses a coupling argument. The idea is to construct a pair $(U^1,U^2)$ with the same law as $(\tilde{X}^1,\tilde{X}^2)$ under $P^{\omega}_{(\mathbf{x}, \mathbf{z})}$ and a pair $(V^1,V^2)$ with the same law as  $(\tilde{X}^1,\tilde{X}^2)$ under $P^{\omega_K}_{(\mathbf{x}, \mathbf{z})}$. In addition, we will construct the random variables in such a way that $\displaystyle (U_i^1,U_i^2)_{0\le i \le n} = (V_i^1,V_i^2)_{0 \le i \le n}$ as soon as $\mathcal{D}^{\bullet}(U) > n$, where $\mathcal{D}^{\bullet}(U)$ denotes the quantity $\mathcal{D}^{\bullet}$ defined in \eqref{definition_D} associated to $(U^1,U^2)$. 

We give the construction of $(U^1_{\cdot},V^{1}_{\cdot})$ under a probability measure $P^{\omega,\omega_K}_{(\mathbf{x}, \mathbf{z})}$. We can then build $(U^2_{\cdot},V^{2}_{\cdot})$ independently in the same way under $P^{\omega,\omega_K}_{(\mathbf{x}, \mathbf{z})}$. To simplify the notations and to highlight the fact that $U^1$ has the same law as $\tilde{X}^1 = (X^1,Z^1)$, we will write $U^1 = U = (X_n,Z_n)_{n \geq 0}$ and $V^1 = V = (Y_n,W_n)_{n\geq 0}$. We also call $\mathcal{D}^{\bullet}_U = \mathcal{D}^{\bullet 1}(U)$ and $\mathcal{D}^{\bullet}_V = \mathcal{D}^{\bullet 1}(V)$.

We set $((X_0,Z_0),(Y_0,W_0)) = ((x_1,z_1),(x_1,z_1))$ $P^{\omega,\omega_K}_{(\mathbf{x}, \mathbf{z})}$-almost surely. Given the trajectory up to time $n \geq 0$ the conditional law of $(U_{n+1},V_{n+1})$ is given by the following rules:
\begin{itemize}
    \item[$1.$] If $\{\mathcal{D}^{\bullet}_U > n\} \cap \{\mathcal{D}^{\bullet}_V > n\}$ and if $U_n = V_n = x \in \mathbb{Z}^d$. Then $U$ makes a step according to $P^{\omega}$ and
    \begin{itemize}
        \item[(a)] If $x \notin \{x_1,x_2\}$ and $x \notin \{x_i + e_j \text{ | } j \in \{1,...,2d\}\text{ and }i \in \{1,2\}\}$ then $V_{n+1}= U_{n+1}$.
        \item[(b)] If $x \in \{x_1,x_2\}$ or $x \in \{x_i + e_j \text{ | } j \in \{1,...,2d\}\text{ and }i \in \{1,2\}\}$ and if $Z_{n+1} = 1$ then  $V_{n+1}= U_{n+1}$.
        \item[(c)] If $x \in \{x_1,x_2\}$ or $x \in \{x_i + e_j \text{ | } j \in \{1,...,2d\}\text{ and }i \in \{1,2\}\}$ and if $Z_{n+1} = 0$ then for any $j \in \{1,...,2d\}$, $(Y_{n+1},W_{n+1}) = (x + e_j,0)$ with probability
        \begin{equation*}
            \frac{p^{\omega_K}(x,x+e_j)-p_K^{\omega_K}(x,x+e_j)}{1- \displaystyle \somme{1 \leq j \leq 2d}{}{p_{K}^{\omega_K}(x,x+e_j)}}.
        \end{equation*}
        \end{itemize}
        \item[$2.$] if $\{\mathcal{D}^{\bullet}_U \leq n\} \cup \{\mathcal{D}^{\bullet}_V \leq n\}$ then $U$ and $V$ move independently according to $P^{\omega}$ and $P^{\omega_K}$ respectively.
\end{itemize}
We have to show that the pair $(U,V)$ is well defined. By construction, if $\{\mathcal{D}^{\bullet}_U  > n\} \cap \{\mathcal{D}^{\bullet}_V > n\} \cap \{U_n = V_n\}$ then $P^{\omega,\omega_K}_{(\mathbf{x}, \mathbf{z})}$-almost surely either $\{\mathcal{D}^{\bullet}_U  > n+1\} \cap \{\mathcal{D}^{\bullet}_V > n+1\} \cap \{U_{n+1} = V_{n+1}\}$ or $\{\mathcal{D}^{\bullet}_U = \mathcal{D}^{\bullet}_V = n+1\}$.

By induction it follows that: 
 \begin{itemize}
     \item The event $\{\mathcal{D}^{\bullet}_U > n\}\cap \{\mathcal{D}^{\bullet}_V > n\}\cap\{U_n \neq V_n\}$ never occurs $P^{\omega,\omega_K}_{(\mathbf{x}, \mathbf{z})}$-almost surely. Thus $U$ and $V$ are well defined. 
     \item Under $\{\mathcal{D}^{\bullet}_U > n\}\cap \{\mathcal{D}^{\bullet}_V > n\}$ we have $X_k = Y_k$ for all $k \in \{0,...,n\}$ $P^{\omega,\omega_K}_{(\mathbf{x}, \mathbf{z})}$-almost surely.
     \item We have $\mathcal{D}^{\bullet}_U = \mathcal{D}^{\bullet}_V$   $P^{\omega,\omega_K}_{(\mathbf{x}, \mathbf{z})}$-almost surely.
 \end{itemize} 
\noindent Thus we have
\begin{align} 
      &P^{\omega,\omega_K}_{(\mathbf{x}, \mathbf{z})}\left(\{\mathcal{D}^{\bullet}_U = +\infty\}\cap \left\{\{\exists n \geq 0 \colon U_n \neq V_n\} \cup \{\mathcal{D}^{\bullet}_V < +\infty\}\right\} \right ) \label{coupling}\\
      &\leq P^{\omega,\omega_K}_{(\mathbf{x}, \mathbf{z})}\left(\mathcal{D}^{\bullet}_U \neq \mathcal{D}^{\bullet}_V\right) + \somme{n = 0}{+\infty}{P^{\omega,\omega_K}_{(\mathbf{x}, \mathbf{z})}\left(\mathcal{D}^{\bullet}_U > n+1, \mathcal{D}^{\bullet}_V >n+1, X_{n+1} \neq Y_{n+1}\right)} = 0. \nonumber
\end{align}
We notice that $p_{K}^{\omega}(x,\cdot) = p_{K}^{\omega_K}(x,\cdot)$ for any $x \in \mathbb{Z}^d$ and $p^{\omega}(x,\cdot) = p^{\omega_K}(x,\cdot)$ as soon as $x \notin \{x_1,x_2\} \cup \{x_i + e_j \text{ | } j \in \{1,...,2d\}\text{ and }i \in \{1,2\}\}$, it is easy to check that the law of $U$ (resp.\ $V$) under $P^{\omega,\omega_K}_{(\mathbf{x}, \mathbf{z})}$ is the law of $\tilde{X}^1$ under $P^{\omega}_{(\mathbf{x}, \mathbf{z})}$ (resp.\ the law of $\tilde{X}^1$ under $P^{\omega_K}_{(\mathbf{x}, \mathbf{z})}$).
 
Then take $A_0,\cdots,A_n$ where $A_i$ are in the $\sigma$-algebra generated by the subsets of $\mathbb{Z}^d \times \{0,1\}$ and call $A = A_0 \times \cdots \times A_n$. We have that
 \begin{align*}
     E^{\omega}_{(\mathbf{x}, \mathbf{z})}\left[\indi{}\{(\tilde{X}^1_{0},\cdots,\tilde{X}^1_{n}) \in A\}\indi{\mathcal{D}^{\bullet 1} = +\infty } \right] &= E^{\omega,\omega_K}_{(\mathbf{x}, \mathbf{z})}\left[\indi{}\{(U_{0},...,U_{n}) \in A\}\indi{\mathcal{D}^{\bullet}_U = +\infty } \right]\\
     &= E^{\omega,\omega_K}_{(\mathbf{x}, \mathbf{z})}\left[\indi{}\{(V_{0},...,V_{n}) \in A\}\indi{\mathcal{D}^{\bullet}_V = +\infty } \right] \\
     & = E^{\omega_K}_{(\mathbf{x}, \mathbf{z})}\left[\indi{}\{(\tilde{X}^1_{0},...,\tilde{X}^1_{n}) \in A\}\indi{\mathcal{D}^{\bullet 1} = +\infty } \right],
 \end{align*}
where the second equality follows from \eqref{coupling}. By the Monotone Class Theorem, for any positive bounded $\sigma \big(\tilde{X}^{1}_{n}\big)$-measurable function we have
 \begin{equation*}
         E^{\omega }_{(\mathbf{x}, \mathbf{z})}\left[f(\tilde{X}^1_{\cdot})\indi{\mathcal{D}^{\bullet 1} = +\infty}\right] = E^{\omega_K }_{(\mathbf{x}, \mathbf{z})}\left[f(\tilde{X}^1_{\cdot})\indi{\mathcal{D}^{\bullet 1} = +\infty}\right].
 \end{equation*}
Then, for ant positive bounded measurable functions $f$ and $g$ we have:
\begin{align*}
    E^{\omega }_{(\mathbf{x}, \mathbf{z})}\left[f(\tilde{X}^1_{\cdot})g(\tilde{X}^2_{\cdot})\indi{\mathcal{D}^{\bullet} = +\infty}\right] 
    &= E^{\omega }_{(\mathbf{x}, \mathbf{z})}\left[f(\tilde{X}^1_{\cdot})g(\tilde{X}^2_{\cdot})\indi{\mathcal{D}^{\bullet 1} = +\infty}\indi{\mathcal{D}^{\bullet 2} = +\infty}\right] \\
    & = E^{\omega }_{(\mathbf{x}, \mathbf{z})}\left[f(\tilde{X}^1_{\cdot})\indi{\mathcal{D}^{\bullet 1} = +\infty}\right]        E^{\omega }_{(\mathbf{x}, \mathbf{z})}\left[g(\tilde{X}^2_{\cdot})\indi{\mathcal{D}^{\bullet 2} = +\infty}\right]\\
    &= E^{\omega_K }_{(\mathbf{x}, \mathbf{z})}\left[f(\tilde{X}^1_{\cdot})\indi{\mathcal{D}^{\bullet 1} = +\infty}\right] E^{\omega_K }_{(\mathbf{x}, \mathbf{z})}\left[g(\tilde{X}^2_{\cdot})\indi{\mathcal{D}^{\bullet 2} = +\infty}\right] \\
    &=E^{\omega_K }_{(\mathbf{x}, \mathbf{z})}\left[f(\tilde{X}^1_{\cdot})g(\tilde{X}^2_{\cdot})\indi{\mathcal{D}^{\bullet 1} = +\infty}\indi{\mathcal{D}^{\bullet 2} = +\infty}\right] \\
    &=E^{\omega_K }_{(\mathbf{x}, \mathbf{z})}\left[f(\tilde{X}^1_{\cdot})g(\tilde{X}^2_{\cdot})\indi{\mathcal{D}^{\bullet} = +\infty}\right].
\end{align*}
To extend the previous result to bounded measurable functions of the form $f(\cdot,\cdot)$ we use again the Monotone Class Theorem.
\end{proof}

\subsection{$K$-good points}\label{SectionKGood}
We recall the notion of \emph{$K$-\text{good}} point from \cite[Definition 4.3]{Kious_Frib}. 
\begin{definition1}\label{DefKGood}
    A point $z \in \mathbb{Z}^d$ is said to be \emph{$K$-\text{good}} if there exists an infinite path $(x_0 = z, x_1,x_2,\cdots)$ such that all $x_i$ are $K$-open and $x_{2i+1}-x_{2i} = e_1$ and $x_{2i+2}-x_{2i+1} \in \{e_1,e_2,\cdots,e_d\}$.
\end{definition1}
\begin{proposition} \label{KGoodLemma}
There exists a constant $K_0>1$ and $\eta > 0$ such that, for all $K \ge K_0$ and any $U_1,U_2 \in \mathbb{Z}^d$, we have $\mathbf{P}(U_1 \text{ and }U_2\text{ are }K\text{-good points}) > \eta > 0$.
\end{proposition}
\begin{proof}
For $e^1,\cdots,e^n \in E(\mathbb{Z}^d)$ distinct edges, the events $(c_{*}(e^i) \in [1/K,K])_{1\le i \le n}$ are independent and have the same probability. Thus $\displaystyle (\indi{c_{*}(e) \in [1/K,K]})_{e\in E(\mathbb{Z}^d)}$ is a percolation on $\mathbb{Z}^d$. Since for all $x \in \mathbb{Z}^d$ the event $\{x \text{ is }K\text{-good}\}$ is an increasing event, we can apply FKG inequality and obtain
\begin{align*}
\mathbf{P}(U_1 \text{ and }U_2\text{ are }K\text{-good points}) \ge \mathbf{P}(U_1 \text{ is a }K\text{-good point})\mathbf{P}(U_2 \text{ is a }K\text{-good point}).
\end{align*}
We conclude using \cite[Lemma~4.1]{Kious_Frib} which states that $\mathbf{P}(0 \text{ is a }K\text{-good point})>0$ and the translation invariance property of $\mathbf{P}$.
\end{proof}

\subsection{Construction of joint regeneration levels and main result} \label{joint_reg_levels}
We recall the notations $\mathcal{M}^{\bullet}$ and $M$ defined respectively in \eqref{definition_M},\eqref{definition_M_other}. Furthermore, we recall that $\theta_{s, t}$ denotes, in the context of two walks, the canonical time shift of the first walk by $s$ units of time and the second one by $t$ units of time.

We set
\begin{equation*}
    L_0 = X^1_0 \cdot  \vec{\ell} \wedge X^{2}_{0}\cdot \vec{\ell}, \quad M_{0} = X^1_0 \cdot  \vec{\ell} \wedge X^{2}_{0} \cdot \vec{\ell}
\end{equation*}
and, for $k\geq 0$, 
\begin{equation} \label{definition_Lk}
   L_{k+1} \coloneqq \mathcal{M}^{\bullet}\circ \theta_{T^1_{\mathcal{H}^+(M_k)}, T^2_{\mathcal{H}^+(M_k)}},
\end{equation}
where we set
%\begin{equation*}
%\begin{split}
%    R_k &\coloneqq \min\left\{ R_k^1, R_k^2 \right\}, \\ 
%    R_k^{1} &\coloneqq\mathcal{D}^{\bullet 1} \circ \theta_{T^1_{L_k}, T^2_{L_k}} + T^1_{L_k}, \quad \quad R_k^{2} \coloneqq\mathcal{D}^{\bullet 2} \circ \theta_{T^1_{L_k}, T^2_{L_k}} + T^2_{L_k},
%\end{split}
%\end{equation*}
%and consequently
\begin{equation} \label{definition_Mk}
    M_k = M \circ \theta_{T^1_{L_k},T^2_{L_k}}.
\end{equation}
We define the first joint regeneration level as
\begin{equation}\label{eqn:FirstJointLevel}
    \mathcal{L}_1 = L_{N}, \quad \text{with} \quad N \coloneqq \inf\{k \geq 0 \colon L_k < +\infty \text{ and } M_k =+\infty\}.
\end{equation}
The associated joint regeneration times (respectively points) for $X^1$ and $X^2$ are $T^{1}_{\mathcal{L}_1}$ and $T^{2}_{\mathcal{L}_1}$ (resp.\ $X^{1}_{T^{1}_{\mathcal{L}_1}}$ and $X^{2}_{T^{2}_{\mathcal{L}_1}}$).
\begin{figure}[H] %\label{fig_def_reg_level}
    \centering\includegraphics[scale= 0.45]{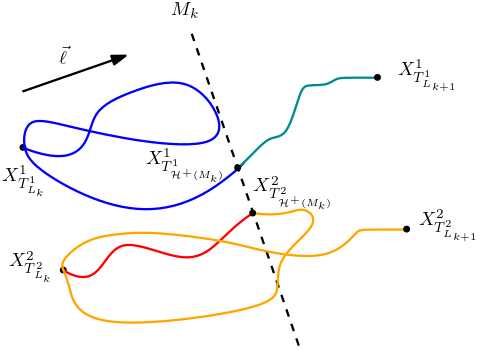}
    \caption{We illustrate the definition $M_k$ and $L_{k+1}$. The portion in blue and red are respectively the trajectory of $X^1$ and $X^2$ between the hitting time of $\mathcal{H}^{+}(M_k)$ and the hitting time of level $M_k$. The portion in cyan and orange are respectively the trajectory of $X^1$ and $X^2$ between the hitting time of $\mathcal{H}^{+}(M_k)$ and the hitting time of level $L_{k+1}$.}
    \label{fig_def_reg_level}
\end{figure}

\begin{remark1}\label{remark:Space}
We would like to highlight why we chose to define the regeneration levels using the quantities $M$ and $\mathcal{D}^{\bullet i}_{\le R}$. In essence, regeneration is more of a spatial property even for a single walk, as, to guarantee independence, we need two pieces of trajectory to evolve in two distinct portions of the space. However, for a single walk it is equivalent to define regeneration times or points as there is a bijection between the two. For two walks, we would need two different regeneration times to encapsulate the information of the joint regeneration structure. However, it seems more natural to us to use a single quantity: a level such that the hitting times (by the two walks) of that level correspond to regeneration times.
\end{remark1}

We observe that the definition is very similar to the definition in Section~\ref{definition_reg_times} but instead of defining a sequence of stopping times $(S_k)_{k}$ we define a sequence of levels $(L_k)_{k}$. The first regeneration level is the first level $L_k$ such that the walks do not backtrack and properly encode the information around the starting points in the sense of Proposition~\ref{omegaK}. 

To have a properly defined regeneration structure it is essential to show that $\mathcal{L}_1$, as defined in \eqref{eqn:FirstJointLevel}, is $\Pf_0$-almost surely finite. Proposition~\ref{PropositionShortRegLevel} addresses this issue and shows that the first joint regeneration level happens closely to the origin. After proving Proposition~\ref{PropositionShortRegLevel}, we will focus on proving the Markov property in Theorem~\ref{IndependenceJointReg}.
\begin{proposition}\label{PropositionShortRegLevel}
For any $M > 0$ there exists $K_0 > 1$ such that, for all $K \ge K_0$,
\begin{equation*}\label{L1decay}
    \sup_{U_1, U_2 \in \mathcal{U}} \Pf_{U_1,U_2}(\mathcal{L}_1 \geq n) \leq Cn^{-M},
\end{equation*}
where the constant $C$ only depends on $\lambda,\vec{\ell},d,K$ and $\mathcal{U}$ is defined in \eqref{eqn:USet}.
\end{proposition}
We will devote the next subsection to the proof of this result, we will need several lemmas in order to be able to prove it. 

\subsection{Proof of Proposition~\ref{L1decay}}

We recall the notation $\Pf_{U_1,U_2}^{K}$ which is the usual annealed measure with the restriction that the conductances around $U_1,U_2$ are set to have value $K$.

\begin{lemma}\label{LemmaDpositive}
There exists $K_0 > 1$ and $ \eta > 0$ such that, for all $K \geq K_0$, we have 
\begin{align*}
    & \inf_{U_1, U_2 \in \mathcal{U}} \Pf_{U_1,U_2}(\mathcal{D}^{\bullet}=+\infty) >  \eta > 0, \\
    & \inf_{U_1, U_2 \in \mathcal{U}} \Pf^{K}_{U_1,U_2}(\mathcal{D}^{\bullet}=+\infty) >  \eta > 0,
\end{align*}
where $\eta$ only depends on $\lambda,\vec{\ell},d,K$ and $\mathcal{U}$ is defined in \eqref{eqn:USet}.
\begin{proof}
We follow the lines of \cite[Lemma 5.1]{Kious_Frib}. We only prove the result for $\Pf_{U_1,U_2}$ since the proof for $\Pf_{U_1,U_2}^{K}$ is analogous.

We have 
\begin{align*}
     \Pf_{U_1,U_2}(\mathcal{D}^{\bullet}=+\infty) > c \Pf_{U_1,U_2}(\mathcal{D}^{\bullet}=+\infty \text{ | }U_1 \text{ and } U_2 \text{ are } K\text{-good}),
\end{align*}
where $c = \mathbf{P}(U_1 \text{ and } U_2 \text{ are } K\text{-good}) > \eta > 0$ using Proposition~\ref{KGoodLemma} and where $\eta$ only depends on $d$. On the event $\{U_1 \text{ and } U_2 \text{ are } K\text{-good}\}$ we fix $\mathcal{P}^i$ to be the path associated to the fact that $U_i$ is a $K$-good point (see Definition~\ref{DefKGood}). By definition we have $\mathcal{P}^{i}(1) = U_{i} + e_1$ and $\mathcal{P}^{i}(2) = U_{i} + e_1 + e_{j_2}$ with $e_{j_2} \in \{e_1,\cdots,e_d\}$. We introduce $L^i_{\partial^{+}\mathcal{B}_{U_i}(n,n^2)} \coloneqq \inf\{k \in \mathbb{N} \colon \mathcal{P}^{i}(k) \in \partial^{+}\mathcal{B}_{U_i}(n,n^2) \}$ and set $x^i = \mathcal{P}^{i}(L^{i}_{\partial^{+}\mathcal{B}_{U_i}(n,n^2)})$.

For $i = 1, 2$ we define the events
\begin{align*}
    &A_1^{i} = \{(X^{i}_1,Z^{i}_1) = (U_{i} + e_1,1)\},\\
    &A_2^{i} = \{(X^{i}_2,Z^{i}_2) = (U_{i} + e_1+e_{j_2},1)\},\\
    &A_3^{i} = \{X^{i}_{k} = \mathcal{P}^{i}(k) \text{ for } 3\leq k \leq L^i_{\partial^{+}\mathcal{B}_{U_i}(n,n^2)}\},\\
    &A_4^{i} = \{T^{i}_{\mathcal{H}^{-}(2)} \circ \theta_{T^i_{x^i}} = +\infty\}.
\end{align*}
Furthermore, we set $A_k = A^1_k \cap A^2_k$ and $ A = \bigcap_{k=1}^{4}A_{k}$. We observe that $A \subset \{\mathcal{D}^{\bullet} =+\infty\}$. As $U_1$ and $U_2$ are $K$-good we have $L^{i}_{\partial^{+}\mathcal{B}_{U_i}(n,n^2)} \leq Cn$ and since all points in $\mathcal{P}^{i}$ are $K$-open, we have 
\begin{align*}
    P_{U_1,U_2}^{\omega}(A_1,A_2,A_3) \geq c^{n},
\end{align*}
with $c > 0$ which only depends on $K$, $d$ and $\ell$. Thus, by Markov property we have
\begin{align*}
    \mathbf{E}[P_{U_1,U_2}^{\omega}&(A_1,A_2,A_3,A_4) \text{ | }U_1 \text{ and } U_2 \text{ are } K\text{-good}]\\
    &= \mathbf{E}[P_{U_1,U_2}^{\omega}(A_1,A_2,A_3)P^{\omega}_{x^1,x^2}(A^1_4,A^2_4) \text{ | }U_1 \text{ and } U_2 \text{ are } K\text{-good}]\\
    &\geq c^n\mathbf{E}[P^{\omega}_{x^1,x^2}(A_4) \text{ | }U_1 \text{ and } U_2 \text{ are } K\text{-good}]\\
    & = c^n \mathbf{E}\left[P^{\omega}_{x^1,x^2}(T^{1}_{\mathcal{H}^{-}(2)} = +\infty,T^{2}_{\mathcal{H}^{-}(2)} = +\infty) \text{ | }U_1 \text{ and } U_2 \text{ are } K\text{-good}\right].
\end{align*}
We also have 
\begin{align}
\mathbf{E}&\left[P^{\omega}_{x^1,x^2}( \{ T^1_{\mathcal{H}^{-}(2)} < +\infty \}\cup \{T^2_{\mathcal{H}^{-}(2)} < +\infty\}) \text{ | }U_1 \text{ and } U_2 \text{ are } K\text{-good}\right] \nonumber\\
     &\leq C\Pf_0\left(T^1_{\mathcal{H}^{-}(-n+2)} < +\infty\right) \nonumber\\
     &\leq C\exp(-cn).\label{eqn:ExponentialBound}
\end{align}
Here we used a union bound, the translation invariance of $\mathbf{P}$, the fact $\mathbf{P}(U_1 \text{ and } U_2 \text{ are } K\text{-good}) > \eta > 0$ and \cite[Lemma 3.1]{Kious_Frib}. 
Hence, we see that there exists $n_0$ large enough which only depends on $\lambda,\vec{\ell},d,K$ such that \eqref{eqn:ExponentialBound} is smaller than $1/2$ for all $n \ge n_0$ and conclude the proof as we showed that, uniformly over the choice of $U_1, U_2$
\begin{equation*}
    \mathbf{E}[P_{U_1,U_2}^{\omega}(\mathcal{D}^{\bullet} =+\infty) \text{ | }U_1 \text{ and } U_2 \text{ are } K\text{-good}] \ge (1/2) c^{n_{0}} > 0.
\end{equation*}
\end{proof}

\end{lemma}
\noindent The following lemmas correspond to the ones in \cite[Appendix A]{Kious_Frib}. We recall the definitions of $L_k$ and $M_k$ in \eqref{definition_Lk}, \eqref{definition_Mk}, then we introduce the events:
\begin{align}
    &\label{definition_Mkn}M^{K}(n) \coloneqq \left\{\text{for }k \text{ such that }M_k < n \text{ we have }L_{k+1} - M_k < n^{1/2} \right\}, \\
    &\label{definition_Sn}S(n) \coloneqq \left\{\text{for } k \text{ such that }L_k < n \text{ and } M_k < +\infty \text{ we have } M_k - L_k < n^{1/2}\right\}.
\end{align}
In the following we may drop the superscript $K$ to lighten the notations.

\begin{lemma}
For any $M > 0$ there exists $K_0 > 1$ and a constant $C > 0$ such that, for all $K \ge K_0$,
\begin{align*}
 \sup_{U_1, U_2 \in \mathcal{U}} \Pf_{U_1,U_2}\left(M^{K}(n)^{c}\right) \leq Cn^{-M},
\end{align*}
where the constant $C$ only depends on $\lambda,\vec{\ell},d,K$ and $\mathcal{U}$ is defined in \eqref{eqn:USet}.
\end{lemma}
\begin{proof}
This result follows follows from Proposition~\ref{PropFrequentRegLevels}. Indeed, $M^{K}(n)^{c}$ implies that there exists $0 \leq R \leq n$ such that $\mathcal{M}^{\bullet}_R \geq n^{\frac{1}{2}}$.
\end{proof}
\noindent We recall the definition of the quantities $M^i, i = 1, 2$ which were introduced in \eqref{definition_M_other}. 
The following proposition shows that under $\mathcal{D}^{\bullet i} < +\infty$, $M^i$ cannot be too large.
\begin{lemma} \label{LemmaMexpdecay}
There exists constants $C,c$ such that
\begin{align*}
     \sup_{U_1, U_2 \in \mathcal{U}} \sup_{i = 1, 2} \Pf_{U_1,U_2}\left(M^i \geq n \text{ }|\text{ }\mathcal{D}^{\bullet i}<+\infty \right) \leq C\exp(-cn),
\end{align*}
where $C,c$ only depend on $\lambda,\vec{\ell},d,K$ and $\mathcal{U}$ is defined in \eqref{eqn:USet}.
\end{lemma}
\begin{proof}
Let us show the proof for $i=1$. First, let us recall that $\Pf_{U_1,U_2}(\mathcal{D}^{\bullet1} < +\infty) > \eta > 0$ with $\eta$ which does not depend on $U_1,U_2$. Indeed
\begin{equation}
    0 < \Pf_{U_1,U_2}(X_1^1 = U_1+e_1,X^1_2 = U_1) = \Pf_{0}(X_1 = e_1 , X_2 = 0) < \Pf_{U_1,U_2}(\mathcal{D}^{\bullet 1} < +\infty),
\end{equation}
where we used translation invariance of the environment for the equality.

We have to dominate $\Pf_{U_1,U_2}(\mathcal{D}^{\bullet 1}<+\infty, M^i \geq n)$. It is not hard to check that when $\mathcal{D}^{\bullet 1}<+\infty$ then $M^1 = \sup_{j \leq \mathcal{D}^{\bullet 1}}X^1_{j}\cdot\vec{\ell}$. By union bound and symmetry, we only have to dominate
\begin{align*}
    \Pf_{U_1,U_2}\left(\mathcal{D}^{\bullet 1}<+\infty,\sup_{j \leq \mathcal{D}^{\bullet 1}}X^1_{j}\cdot\vec{\ell} \geq n\right).
\end{align*}
\noindent Without loss of generality we can suppose $U_1= U_2 =0$, indeed the proof is exactly the same in the general case but instead of considering the boxes $\mathcal{B}(2^k,2^{\alpha k})$ we should consider the boxes $\mathcal{B}_{U_i}(2^k,2^{\alpha k})$.

First, we have
\begin{align*}
    \Pf_{0}\left(2^k \leq \sup_{j \leq \mathcal{D}^{\bullet 1}}X^1_{j}\cdot\vec{\ell} < 2^{k+1} \right) \leq &\Pf_{0}\left(T^1_{\partial \mathcal{B}(2^k,2^{\alpha k})} \neq T^1_{\partial^{+} \mathcal{B}(2^k,2^{\alpha k})}\right) +\\&
    \Pf_{0}\Big(T^1_{\partial \mathcal{B}(2^k,2^{\alpha k})} = T^1_{\partial^{+} \mathcal{B}(2^k,2^{\alpha k})},\\&T^{+}_{\mathcal{H}^{-}(0)}\circ \theta_{T^{1}_{\partial^{+} \mathcal{B}(2^k,2^{\alpha k})}}< T^{+}_{\mathcal{H}^{+}(2^{k+1})}\circ \theta_{T^{1}_{\partial^{+} \mathcal{B}(2^k,2^{\alpha k})}} \Big).
\end{align*}
The first term is smaller than $C\exp(-c2^k)$ by Lemma~\ref{LemmaA(n)}. Using a union bound and translation invariance of the environment the second term is dominated by 
\begin{equation}\label{eqn:exponentialScales}
    (2^{k})^{\alpha d}\Pf_{0}\left(T^1_{\partial \mathcal{B}(2^k,2^{\alpha k})} \neq T^1_{\partial^{+} \mathcal{B}(2^k,2^{\alpha k})}\right) \leq C2^{\alpha kd}\exp(-c2^k),
\end{equation}
where we used again Lemma~\ref{LemmaA(n)}. We conclude using the following domination and \eqref{eqn:exponentialScales}
\begin{align*}
    \Pf_{0}\left(\mathcal{D}^{\bullet 1}<+\infty,\sup_{j \leq \mathcal{D}^{\bullet 1}}X^1_{j}\cdot\vec{\ell} \geq n\right) \leq C\somme{k,2^{k}\geq n}{}{\Pf_{0}\left(2^k \leq \sup_{j \leq \mathcal{D}^{\bullet 1}}X^1_{j}\cdot\vec{\ell} < 2^{k+1}\right)} \le C\exp(-cn).
\end{align*}
\end{proof}
\noindent We recall the definition of $S(n)$ in \eqref{definition_Sn} and of the set $\mathcal{U}$ in \eqref{eqn:USet}.
\begin{lemma}\label{LemmaSdecayexp}
There exists constants $C,c > 0$ such that
\begin{align*}
        \sup_{U_1, U_2 \in \mathcal{U}} \Pf_{U_1,U_2}\left(S(n)^{c}\right) \leq C\exp\left(-c \sqrt{n}\right),
\end{align*}
where $C,c > 0$ only depend on $\lambda,\vec{\ell},d,K$.
\end{lemma}
\begin{proof}
We adapt the argument of \cite[Lemma 7.2]{Frib_conduc}, recall that $\mathcal{B}_i(n,n^{\alpha}) = \mathcal{B}_{X_0^{i}}(n,n^{\alpha})$ for $i$-th the random walk. First, we notice that $\Pf_{0}$-almost surely, $|\{k \geq 0 \text{ such that } L_k < n\}| \le dn$. Hence we can write 
    \begin{align*}
        \displaystyle \Pf_{U_1,U_2}\left(S(n)^{c}\right) \leq &\Pf_{U_1}\left(T^1_{\partial \mathcal{B}_1(n,n^{\alpha})} \neq T^1_{ \partial^{+} \mathcal{B}_1(n,n^{\alpha})} \right) + \Pf_{U_2}\left(T^2_{\partial \mathcal{B}_2(n,n^{\alpha})} \neq T^2_{ \partial^{+} \mathcal{B}_2(n,n^{\alpha})} \right) \\
        &+ \somme{k \leq d n}{}{\somme{\substack{x \in \mathcal{B}_1(n,n^{\alpha})\\y \in \mathcal{B}_2(n,n^{\alpha})}}{}{\Pf_{U_1,U_2}\left(M_k < +\infty, M_k - L_k \geq n^{\frac{1}{2}},X^{1}_{T^{1}_{L_k}} = x, X^{2}_{T^{2}_{L_k}} = y\right)}}.
    \end{align*}
The first two terms are controlled by Lemma~\ref{LemmaA(n)}. We can dominate the term in the sum by
\begin{align*}
      & \Pf_{U_1,U_2}\left(\sup_{j \leq \mathcal{D}^{\bullet 1} \circ \theta_{T^1_{x}} + T^1_{x}}(X^{1}_j - x)\cdot\vec{\ell} \geq n^{\frac{1}{2}},\mathcal{D}^{\bullet 1}\circ \theta_{T^{1}_{x}, T^{2}_{y}} < +\infty,X^{1}_{T^{1}_{L_k}} = x,X^{2}_{T^{2}_{L_k}} = y\right) +\\
    &\Pf_{U_1,U_2}\left(\sup_{j \leq \mathcal{D}^{\bullet 2} \circ \theta_{T^2_{y}} + T^2_{y}}(X^{2}_j - y)\cdot\vec{\ell} \geq n^{\frac{1}{2}},\mathcal{D}^{\bullet 2}\circ \theta_{T^{1}_{x}, T^{2}_{y}}< +\infty,X^{1}_{T^{1}_{L_k}} = x,X^{2}_{T^{2}_{L_k}} = y\right).
\end{align*}
Let us dominate the first term, the second is dominated the same way by symmetry. Using Lemma~\ref{Lemma2StoppingTimes} and the translation invariance of the environment it is bounded by 
\begin{align*}
\mathbf{E}\left[P_{x,y}^{\omega}\left(\sup_{j \leq \mathcal{D}^{\bullet 1}}(X^{1}_j -x)\cdot\vec{\ell} \geq n^{\frac{1}{2}},\mathcal{D}^{\bullet 1} < +\infty\right) \right]&
        \leq \mathbf{E}\left[P_{0,y-x}^{\omega}\left(\sup_{j \leq \mathcal{D}^{\bullet 1}} X^{1}_j\cdot\vec{\ell} \geq n^{\frac{1}{2}},\mathcal{D}^{\bullet 1} < +\infty\right) \right] \\&\leq \Pf_{0,y-x}\left(\sup_{j \leq \mathcal{D}^{\bullet 1}} X^{1}_j\cdot\vec{\ell} \geq n^{\frac{1}{2}}\text{ }|\text{ }\mathcal{D}^{\bullet 1}<+\infty\right)\\
        & \leq \Pf_{0,y-x}\left(M^1 > n^{\frac{1}{2}}\text{ }|\text{ }\mathcal{D}^{\bullet 1}<+\infty\right).
\end{align*}
The last term is dominated using Lemma~\ref{LemmaMexpdecay} since we have $|(y-x) \cdot \vec{\ell}| \le e_{1} \cdot \vec{\ell}$ by definition of $L_k$. Taking the sum over $k,x,y$ we conclude that we have
\begin{align*}
    \displaystyle \Pf_{U_1,U_2}(S(n)^{c}) \leq Cn^{\beta}\exp(-cn^{\frac{1}{2}}),
\end{align*}
where $\beta$ only depends on $d,\alpha$. We conclude the proof by taking $C$ large enough and $c$ small enough.
\end{proof}

\noindent The following lemma is the equivalent of \cite[Lemma A.4]{Kious_Frib}. Recall that we defined $N = \inf\{k \geq 0 \colon L_k < +\infty \text{ and } M_{k}=+\infty\}$ in \eqref{eqn:FirstJointLevel}, and the definition of $\mathcal{E}_x, \mathcal{E}^{2}_x, \mathcal{L}^x, \mathcal{R}^x$ as in Equations~\eqref{EquationIncidentEdges}-\eqref{EquationRightEdges}.
\begin{lemma}\label{Ndecay}
There exists $K_0 > 1$ such that for all $K \ge K_0$
\begin{align*}
    \sup_{U_1, U_2 \in \mathcal{U}}\Pf_{U_1,U_2}(N \geq n) \leq C\exp(-cn),
\end{align*}
where $C,c > 0$ only depend on $\lambda,\vec{\ell},d$ and $\mathcal{U}$ is defined in \eqref{eqn:USet}.
\end{lemma}
\begin{proof}
We follow the lines of the proof of \cite[Lemma A.$4$]{Kious_Frib}. The idea of the proof is simple, each time we reach a level $L_k$ there is a positive probability to regenerate and, using the way we have defined $\mathcal{D}^{\bullet}$, we should have good independence properties, hence $N$ should decay at least like a geometric random variable.

We introduce event 
\begin{align*}
    C(n) = \left\{\text{for all }k \leq n \text{ such that }L_k <+\infty \text{ we have } \mathcal{D}^{\bullet} \circ \theta_{T^{1}_{L_k}, T^{2}_{L_k}} < +\infty \right\}.
\end{align*}
We have $\{N \geq n\} \subset C(n)$, hence we only have to dominate $C(n)$.
We can write 
\begin{align*}
  \Pf_{U_1,U_2}(C(n+1)) \le \somme{x,y}{}{\mathbf{E}\left[P_{U_1,U_2}^{\omega}\left(X^1_{T^{1}_{L_{n+1}}}=x,X^2_{T^{2}_{L_{n+1}}}=y,C(n),\mathcal{D}^{\bullet}\circ \theta_{T^{1}_{L_{n+1}}, T^{2}_{L_{n+1}}} < +\infty\right)\right]}.
\end{align*}
The event $\{X^1_{T^{1}_{L_{n+1}}}=x,X^2_{T^{2}_{L_{n+1}}}=y,C(n)\}$ is measurable with respect to $\sigma\{(X^{1}_n,Z_{n}^1),(X^{2}_k,Z_{k}^2)\text{ for }0 \leq n \leq T^1_x \text{ and }0\leq k \leq T^2_y\}$. We can apply Lemma~\ref{Lemma2StoppingTimes} and get 
\begin{align*}
P_{U_1,U_2}^{\omega}&\left(X^1_{T^{1}_{L_{n+1}}}=x,X^2_{T^{2}_{L_{n+1}}}=y,C(n),\mathcal{D}^{\bullet} \circ \theta_{T^{1}_{L_{n+1}}, T^{2}_{L_{n+1}}} < +\infty\right) \\&= P_{U_1,U_2}^{\omega}\left(X^1_{T^{1}_{L_{n+1}}}=x,X^2_{T^{2}_{L_{n+1}}}=y,C(n)\right)P_{x,y}^{\omega}(\mathcal{D}^{\bullet} < +\infty).
\end{align*}
The sum is also equal to 
\begin{align*}
    \somme{x,y}{}{\mathbf{E}&\left[P_{U_1,U_2}^{\omega}\left(X^1_{T^{1}_{L_{n+1}}}=x,X^2_{T^{2}_{L_{n+1}}}=y,C(n)\right)P_{x,y}^{\omega}(\mathcal{D}^{\bullet} < +\infty)\right]}\\
    &= \somme{x,y}{}{\mathbf{E}\left[P_{U_1,U_2}^{\omega}\left(X^1_{T^{1}_{L_{n+1}}}=x,X^2_{T^{2}_{L_{n+1}}}=y,C(n)\right)P_{x,y}^{\omega_{K}}(\mathcal{D}^{\bullet} < +\infty)\right]},
\end{align*}
where we used the fact that both $x$ and $y$ are $K$-open and Proposition~\ref{omegaK}.

Since $P_{U_1,U_2}^{\omega}\big(X^1_{T^1_{L_{n+1}}}=x,X^2_{T^2_{L_{n+1}}}=y,C(n)\big)$ is measurable with respect to  $\sigma\big(c_{*}(e) \text{ with } e \in ((\mathcal{L}^{x} \cap \mathcal{L}^{y}) \cup \mathcal{E}^2_{x} \cup \mathcal{E}^2_{y}\big)\big)$ 
and $P_{x,y}^{\omega_{K}}(\mathcal{D}^{\bullet} < +\infty)$ is measurable with respect to $\sigma \big(c_{*}(e) \text{ with } e \in \mathcal{R}^{x} \cup \mathcal{R}^{y} \backslash (\mathcal{E}^1_{x} \cup \mathcal{E}^1_{y})\big)$ and using the fact that those sets of edges are disjoint (see Figure~\ref{disjoint_sets}) we can write by $\mathbf{P}$ independence that the sum is equal to 
\begin{align*}
\somme{x,y}{}{\mathbf{E}\left[P_{U_1,U_2}^{\omega}\left(X^1_{T^{1}_{L_{n+1}}}=x,X^2_{T^{2}_{L_{n+1}}}=y,C(n)\right)\right]\mathbf{E}\left[P_{x,y}^{\omega_{K}}(\mathcal{D}^{\bullet} < +\infty)\right]}.
\end{align*}
Using Lemma~\ref{LemmaDpositive} with $K_0$ large enough we have $\Pf_{x,y}^{K}(\mathcal{D}^{\bullet}<+\infty) \leq \eta < 1$ with $\eta$ a constant which only depends on $\lambda,\vec{\ell},d$. It follows that
\begin{align*}
    \Pf_{U_1,U_2}(C(n+1)) \leq \eta\Pf_{U_1,U_2}(C(n)).
\end{align*}
The result follows by induction.
\end{proof}
We now have all the tools to prove Proposition~\ref{L1decay}. Indeed, combining Lemma~\ref{LemmaMexpdecay}, Lemma~\ref{LemmaSdecayexp} and Lemma~\ref{Ndecay}, it is unlikely to have, for $k \le n$, $L_{k}$ finite and $L_{k+1} - L_{k} \ge n^{\frac{1}{2}}$.
\begin{proof}[Proof of Proposition~\ref{L1decay}]
We claim that, on $\{\mathcal{L}_1 \geq n\}\cap S(n)\cap M(n)$, we have $N \geq n^{1/3}$ for $n$ large enough. Indeed, on $\{\mathcal{L}_1 \geq n\}\cap S(n)\cap M(n)$ for $k < n$ such that $L_k < n $ and $M_k < n $, we have $L_{k+1} - L_k < 2n^{1/2}$, and since $\mathcal{L}_1 \geq n$, as long as $L_{k+1} < n$ we have $M_{k+1} < +\infty$. Hence, it follows that we must have $N \geq n^{1/3}$ since, in the opposite case, we would have 
\begin{equation*}
    \mathcal{L}_1 = L_{N} = \somme{k=0}{N-1}{L_{k+1}-L_k} < N2n^{1/2} < 2n^{5/6}<n.
\end{equation*} 
The result follows by applying our estimates in Lemma~\ref{LemmaMexpdecay}, Lemma~\ref{LemmaSdecayexp} and Lemma~\ref{Ndecay}, putting them together we have
\begin{align*}
    \Pf_{U_1,U_2}(\mathcal{L}_1 \geq n) &\leq \Pf_{U_1,U_2}(M(n)^c) + \Pf_{U_1,U_2}(S(n)^{c}) + \Pf_{U_1,U_2}(N \geq n^{\frac{1}{3}})\\
    &\leq 2C\exp(-cn^{\frac{1}{3}}) + C n^{-M} \\
    &\leq Cn^{-M}.
\end{align*}
\end{proof}

\subsection{Further properties of joint regeneration levels}
We extend Proposition~\ref{L1decay} to hold under the measure $\Pf_{U_1,U_2}^{K}$. This is the analogous of \cite[Theorem 5.3]{Kious_Frib}.
\begin{proposition} \label{L1decaybis}
For any $M > 0$ there exists $K_0 > 1$ such that for all $K \ge K_0$ 
\begin{align*}
   \sup_{U_1, U_2 \in \mathcal{U}} \Pf_{U_1,U_2}^{K}(\mathcal{L}_1 > n) \leq Cn^{-M},
\end{align*}
where the constant $C$ only depends on $\lambda,\vec{\ell},d,K$ and $\mathcal{U}$ is defined in \eqref{eqn:USet}.
\end{proposition}
\begin{proof}
Using the same techniques as in previous proofs, we may restrict ourselves to the case where $U_1 = U_2 = 0$. Indeed, in the general case, we can replace the boxes $\mathcal{B}(n,n^{\alpha})$ with ${\tilde{\mathcal{B}}}^i_{n}$ introduced in \eqref{eqn:defWeirdBox} and run the same argument. We write
\begin{align*}
   &\Pf_{0}^{K}(\mathcal{L}_1 > n) \leq     \Pf^{K}_{0}\left(T^1_{\partial \mathcal{B}(n/2,(n/2)^{\alpha})} \neq T^1_{\partial^{+} \mathcal{B}((n/2),(n/2)^{\alpha})} \right) + \Pf^{K}_{0}\left(T^2_{\partial \mathcal{B}(n/2,(n/2)^{\alpha})} \neq T^2_{\partial^{+} \mathcal{B}(n/2,(n/2)^{\alpha})} \right)\\&+
    \somme{x,y \in \partial^{+} \mathcal{B}((n/2),(n/2)^{\alpha})}{}{\Pf^{K}_{0}(X^1_{T^1_{\partial \mathcal{B}((n/2),(n/2)^{\alpha})}} = x,X^2_{T^2_{\partial \mathcal{B}((n/2),(n/2)^{\alpha})}} = y,T^1_{\mathcal{H}^{-}_{e_1}} \circ \theta_{T^{1}_x} < +\infty )}\\
    &+ \somme{x,y \in \partial^{+} \mathcal{B}((n/2),(n/2)^{\alpha})}{}{\Pf^{K}_{0}(X^1_{T^1_{\partial \mathcal{B}((n/2),(n/2)^{\alpha})}} = x,X^2_{T^2_{\partial \mathcal{B}((n/2),(n/2)^{\alpha})}} = y,T^2_{\mathcal{H}^{-}_{e_1}} \circ \theta_{T^2_y} < +\infty )}\\&+\somme{x,y \in \partial^{+} \mathcal{B}((n/2),(n/2)^{\alpha})}{}{\Pf^{K}_{0}(X^1_{T^1_{\partial \mathcal{B}((n/2),(n/2)^{\alpha})}} = x,X^2_{T^2_{\partial \mathcal{B}((n/2),(n/2)^{\alpha})}} = y,\\&\hspace{4.5cm} T^1_{\mathcal{H}^{-}_{e_1}} \circ \theta_{T^1_x} = +\infty ,T^2_{\mathcal{H}^{-}_{e_1}} \circ \theta_{T^2_y} = +\infty ,\mathcal{L}_1 \circ \theta_{T^1_{x},T^2_{y}} \geq n/2)}.
\end{align*}
The first two terms are dominated using Lemma~\ref{LemmaA(n)}.

Let us dominate the third term, the fourth can be dominated in the same way by symmetry. We can re-write it using the strong Markov property
\begin{align*}
    \mathbf{E}&\left[P^{\omega_K}_{0}\left(X^1_{T^1_{\partial \mathcal{B}((n/2),(n/2)^{\alpha})}} = x,X^2_{T^2_{\partial \mathcal{B}((n/2),(n/2)^{\alpha})}} = y\right)P_{x}^{\omega^{0}_K}\left(T^1_{\mathcal{H}^{-}_{e_1}} < +\infty \right)\right] \\ &\leq \mathbf{E}\left[P_{x}^{\omega^{0}_K}\left(T^1_{\mathcal{H}^{-}_{e_1}} < +\infty \right)\right],
\end{align*}
where the notation $\omega^{0}_{K}$ means that we have set the conductances of $\mathcal{E}_0$ to $K$. Then, we notice that $P_{x}^{\omega^{0}_K}(T^1_{\mathcal{H}^{-}_{e_1}} < +\infty )$ does not depend on the conductances of $\mathcal{E}_0$, so it is equal to $P_{x}^{\omega}(T^1_{\mathcal{H}^{-}_{e_1}} < +\infty )$. Using translation invariance of the environment and \cite[Lemma 3.1]{Kious_Frib} the sum is bounded by $Cn^{2d\alpha}\exp(-cn)$.

Let us dominate the fifth term. Using Lemma~\ref{Lemma2StoppingTimes} the term in the sum can be written as
\begin{align*}
    &\mathbf{E}\left[P_{0}^{\omega_{K}}\left(X^1_{T^1_{\partial \mathcal{B}((n/2)),(n/2)^{\alpha})}} = x,X^2_{T^2_{\partial \mathcal{B}((n/2),(n/2)^{\alpha})}} = y\right)P_{x,y}^{\omega^{0}_{K}}\left(T^1_{\mathcal{H}^{-}_{e_1}} = +\infty ,T^2_{\mathcal{H}^{-}_{e_1}} = +\infty ,\mathcal{L}_1\geq n\right)\right]\\
    &\leq \mathbf{E}\left[P_{x,y}^{\omega^{0}_{K}}\left(T^1_{\mathcal{H}^{-}_{e_1}} = +\infty ,T^2_{\mathcal{H}^{-}_{e_1}} = +\infty ,\mathcal{L}_1\geq n\right)\right].
\end{align*}
On the event $\{T^1_{\mathcal{H}^{-}_{e_1}} = +\infty ,T^2_{\mathcal{H}^{-}_{e_1}} = +\infty\}$ both walks never meet conductances in $\mathcal{E}_{0}$ and so we can replace $\omega^{0}_{K}$ with $\omega$. We write
\begin{align*}
    \mathbf{E}\left[P_{x,y}^{\omega^{0}_{K}}\left(T^1_{\mathcal{H}^{-}_{e_1}} = +\infty ,T^2_{\mathcal{H}^{-}_{e_1}} = +\infty ,\mathcal{L}_1\geq n\right)\right]&\leq \Pf_{x,y}(\mathcal{L}_1\geq n) \\&\leq \Pf_{0,y-x}(\mathcal{L}_1\geq n/4) \le Cn^{-M},
\end{align*}
where we applied Proposition~\ref{L1decay} since we have $|(y-x) \cdot \vec{\ell}| \le e_{1} \cdot \vec{\ell}$ because $x,y \in \partial^{+} \mathcal{B}((n/2),(n/2)^{\alpha})$ and the translation invariance of the environment. Finally, we observe that the sum is bounded by $ C n^{2d\alpha} n^{-(M+2d\alpha)} $ by taking $K$ large enough. This concludes the proof.
\end{proof}
\noindent We can define successive joint regeneration levels as
\begin{equation*}
    \mathcal{L}_{k+1} = \mathcal{L}_{k} \circ \theta_{T^1_{\mathcal{L}_1}, T^2_{\mathcal{L}_1}}.
\end{equation*}
Let us define
\begin{equation}
\begin{split}\label{eqn:SigmaAlgebraG_k}
    \mathcal{G}_k = \sigma &\bigg(\mathcal{L}_0,\dots,\mathcal{L}_k,(\tilde{X}^1_{T^1_{\mathcal{L}_k}\wedge m})_{m \geq 0},(\tilde{X}^2_{T^2_{\mathcal{L}_k}\wedge m})_{m \geq 0}, \\ &c_{*}(e) \text{ with } e \in E(\mathcal{H}^{-}(\mathcal{L}_k)) \cup \mathcal{E}^{2}_{X^1_{T^1_{\mathcal{L}_k}}} \cup \mathcal{E}^{2}_{X^2_{T^2_{\mathcal{L}_k}}}\bigg).
\end{split}
\end{equation}
The definition of $\mathcal{G}_k$ may seem a bit complicated at first sight. It represents the information that the walks see before reaching level $\mathcal{L}_k$ and also the whole environment on the left of $\mathcal{L}_k$.
\begin{figure}[H]
    \centering
    \includegraphics[scale =0.2]{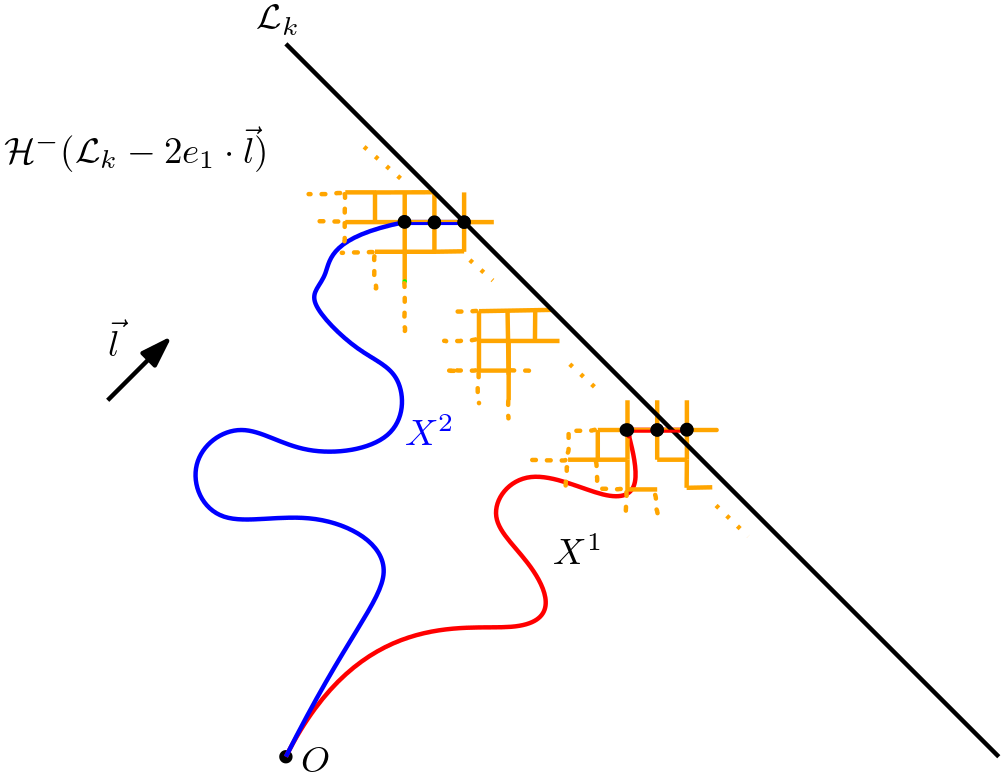}
    \caption{The sigma algebra $\mathcal{G}_k$: In red and blue the trajectories before reaching $\mathcal{L}_k$ and in orange the edges considered.}
    \label{fig:figG_k}
\end{figure}

We want to show that when we consider our walks shifted by $T^{i}_{\mathcal{L}_k}$ they behave like new walks starting from respectively $X^1_{T^1_{\mathcal{L}_k}}$ and $X^2_{T^2_{\mathcal{L}_k}}$ conditioned to never backtrack and such that $c_{*}(e) = K$ for $e \in \mathcal{E}_{X^1_{T^1_{\mathcal{L}_k}}} \cup \mathcal{E}_{X^2_{T^2_{\mathcal{L}_k}}}$. This is the purpose of the next theorem.

\begin{theorem} \label{IndependenceJointReg}
Fix $K$ large enough and $U_1,U_2 \in \mathbb{Z}^d$. For any $k \geq 1$, we have $\mathcal{L}_k < +\infty$ $\Pf_0$-almost surely. Let $f, g, h_k$ be bounded functions which are measurable with respect to $\sigma\left( (X^1_0,X^2_0), (\tilde{X}^1_n,\tilde{X}^2_n)_{n \geq 1}\right)$, $\sigma\left(c_{*}(e) \colon e \in \mathcal{R}^{X^1_0}  \cup \mathcal{R}^{X^2_0} \backslash (\mathcal{E}^1_{X^1_0} \cup \mathcal{E}^1_{X^2_0})\right)$ and $\mathcal{G}_k$, respectively. We have the following equality
\begin{equation*}
    \mathbb{E}_{\substack{(U_1,z_1)\\(U_2,z_2)}}\left[f\left(\tilde{X}^{1}_{T^{1}_{\mathcal{L}_k}+\cdot},\tilde{X}^{2}_{T^{2}_{\mathcal{L}_k}+\cdot}\right)g\circ \theta_{T^1_{\mathcal{L}_k}, T^2_{\mathcal{L}_k}}h_k\right] = \mathbb{E}_{\substack{(U_1,z_1)\\(U_2,z_2)}}\left[h_k\mathbb{E}^K_{X^1_{T^{1}_{\mathcal{L}_k}},X^2_{T^{2}_{\mathcal{L}_k}}}\left[f(\tilde{X}^1_{\cdot},\tilde{X}^{2}_{\cdot})g \text{ } | \text{ } \mathcal{D}^{\bullet} = +\infty\right]\right].
\end{equation*}
\end{theorem}
\begin{remark1}
    In the Theorem, $g$ can be written as $\tilde{g}((X^1_n)_{n \ge 0}, (X^2_n)_{n \ge 0}, \omega)$ where $\tilde{g}$ is a positive measurable function. Thus $g\circ \theta_{T^1_{\mathcal{L}_k}, T^2_{\mathcal{L}_k}} $ designates $\tilde{g}((X^1_n)_{n \ge T^1_{\mathcal{L}_k}}, (X^2_n)_{n \ge T^2_{\mathcal{L}_k}}, \omega)$.
\end{remark1}
\begin{proof}
The proof can be done by induction. We can restrict to the case $k = 1$, but we have to prove the result when the conductances in $\mathcal{E}_{U_1}\cup \mathcal{E}_{U_2}$ are fixed to a deterministic vector $a \in  [1/K,K]^{\mathcal{E}_{U_1}\cup \mathcal{E}_{U_2}}$ 
\begin{equation*}
        \mathbb{E}^a_{\substack{(U_1,z_1)\\(U_2,z_2)}}\left[f\left(\tilde{X}^{1}_{T^{1}_{\mathcal{L}_1}+\cdot},\tilde{X}^{2}_{T^{2}_{\mathcal{L}_1}+\cdot}\right)g\circ \theta_{T^1_{\mathcal{L}_k}, T^2_{\mathcal{L}_k}}h_1\right] = \mathbb{E}^a_{\substack{(U_1,z_1)\\(U_2,z_2)}}\left[h_1\mathbb{E}^K_{X^1_{T^{1}_{\mathcal{L}_1}},X^2_{T^{2}_{\mathcal{L}_1}}}\left[f(\tilde{X}^1_{\cdot},\tilde{X}^{2}_{\cdot})g \text{ } | \text{ } \mathcal{D}^{\bullet} = +\infty\right]\right].
\end{equation*}
Here the notation $\mathbb{E}^{a}$ means that we fix the conductances in $\mathcal{E}_{U_1}\cup \mathcal{E}_{U_2}$ to $a$. The proof with $\mathbb{E}$ follows the same steps.

By definition of $\mathcal{L}_1$, we have
    \begin{align*}
       \displaystyle &\mathbb{E}^a_{\substack{(U_1,z_1)\\(U_2,z_2)}}\left[f\left(\tilde{X}^{1}_{T^{1}_{\mathcal{L}_1}+\cdot},\tilde{X}^{2}_{T^{2}_{\mathcal{L}_1}+\cdot}\right)g\circ \theta_{T^1_{\mathcal{L}_1}, T^2_{\mathcal{L}_1}}h_1\right] \\
        &= \somme{k \geq 0}{}{\mathbb{E}^a_{\substack{(U_1,z_1)\\(U_2,z_2)}}\left[f\left(\tilde{X}^{1}_{T^{1}_{L_k}+\cdot},\tilde{X}^{2}_{T^{2}_{L_k}+\cdot}\right)g\circ \theta_{T^1_{L_k}, T^2_{L_k}}h_1,L_k <+\infty, M_k = +\infty\right]}\\
        &= \somme{\substack{k \geq 0 \\ p \geq 0}, \substack{(x,z)\\(x',z')}}{}{\mathbb{E}^a_{\substack{(U_1,z_1)\\(U_2,z_2)}}\left[f\left(\tilde{X}^{1}_{T^{1}_{p}+\cdot},\tilde{X}^{2}_{T^{2}_{p}+\cdot}\right)g\circ \theta_{T^1_p, T^2_p} h_1,L_k =p, M_k = +\infty,\tilde{X}^1_{T^{1}_{p}} = (x,z),\tilde{X}^2_{T^{2}_p} = (x',z')\right]}.
    \end{align*}
For $p,(x,z),(x',z')$ in the sum and following the arguments of \cite[Proposition 1.3]{Sznitman_Zerner}, there exists $h_1'$, measurable with respect to 
\begin{equation}\label{eqn:SigmaAlgebraCoupling}
    \sigma \left((\tilde{X}^1_{T^1_{p}\wedge m})_{m \geq 0},(\tilde{X}^2_{T^2_{p}\wedge m})_{m \geq 0}, c_{*}(e) \text{ with } e \in (\mathcal{L}^{x} \cap \mathcal{L}^{y}) \cup \mathcal{E}^{2}_{x} \cup \mathcal{E}^{2}_{x'}\right), 
\end{equation}
such that on the event $\{\tilde{X}^1_{T^{1}_p} = (x,z)\}\cap \{ \tilde{X}^2_{T^{2}_p} = (x',z')\} \cap \{\mathcal{L}_1 = p\}$
the functions $h_1$ and $h^{'}_1$ coincide.

Fix $k,p \geq 0$ and $(x,z),(x',z')$, in the sum. We can replace $h_1$ by $h_1'$, but to simplify notations we will still write $h_1$ and suppose it is measurable with respect to the $\sigma$-algebra in \eqref{eqn:SigmaAlgebraCoupling}. We can do the same with $g \circ \theta_{T^1_p, T^2_p}$, indeed we can consider $g$ as a function of $(X^1_0,X^2_0,\omega)$ and then under $\{\mathcal{L}_1 = p,X^1_{T^1_{\mathcal{L}_1}} = x,X^2_{T^2_{\mathcal{L}_1}} = x' \}$ we can replace $g$ by a function $g^{(x,x')}$ which is measurable with respect to
\begin{equation*}
    \sigma\left(c_{*}(e) \colon e \in \mathcal{R}^{x} \cup \mathcal{R}^{x'} \backslash (\mathcal{E}^1_{x} \cup \mathcal{E}^1_{x'})\right).
\end{equation*}
The term in the sum can be written as
    \begin{align*}
        &\mathbb{E}^a_{\substack{(U_1,z_1)\\(U_2,z_2)}}\left[f(\tilde{X}^{1}_{T^{1}_{p}+\cdot},\tilde{X}^{2}_{T^{2}_{p}+\cdot})g^{(x,x')}h_1,L_k =p, M_{k} = +\infty,\tilde{X}^1_{T^{1}_p} = (x,z),\tilde{X}^2_{T^{2}_p} = (x',z')\right]\\
        &=\mathbf{E}\left[E^{\omega_a}_{\substack{(U_1,z_1)\\(U_2,z_2)}}\left[h_1,L_k =p,\tilde{X}^1_{T^{1}_p} = (x,z),\tilde{X}^2_{T^{2}_p} = (x',z')\right]\times g^{(x,x')}E^{\omega_a}_{\substack{(x,z)\\(x',z')}}\left[f(\tilde{X}^{1}_{\cdot},\tilde{X}^{2}_{\cdot}), \mathcal{D}^{\bullet} = +\infty\right]\right].
    \end{align*}
We used Lemma~\ref{Lemma2StoppingTimes} with $\tilde{X}^1$ and $\tilde{X}^2$ and with the stopping times $T^1_{p}$ and $T^2_{p}$. We recall the notation $\omega_a$ which means we fix the conductances in $\mathcal{E}_{U_1}\cup \mathcal{E}_{U_2}$ to $a$. In the following, we will use the more compact notation defined in \eqref{eqn:ShorterNotationJointLaw} with $\mathbf{x} = (x, x')$ and $\mathbf{z} = (z, z')$.

We want to stress that in the expression $E^{\omega_a}_{(\mathbf{x}, \mathbf{z})}[f(\tilde{X}^{1}_{\cdot},\tilde{X}^{2}_{\cdot}), \mathcal{D}^{\bullet} = +\infty]$, with a little abuse of the notation, we fix to the value $a$ the conductances in $\mathcal{E}_{U_1}\cup \mathcal{E}_{U_2}$ and not the ones in $\mathcal{E}_{x}\cup \mathcal{E}_{x'}$ as one might expect. Moreover, thanks to the fact that on $\{\mathcal{D}^{\bullet} = +\infty\}$ $P^{\omega_a}_{(\mathbf{x}, \mathbf{z})}$-almost surely the walks never backtrack below $x$ and $x'$, the value of the conductances in $\mathcal{E}_{U_1}\cup \mathcal{E}_{U_2}$ does not matter in the expectation so we can replace $\omega_a$ with $\omega$ and write $E^{\omega}_{(\mathbf{x}, \mathbf{z})}[f(\tilde{X}^{1}_{\cdot},\tilde{X}^{2}_{\cdot}), \mathcal{D}^{\bullet} = +\infty]$.

Then, since under $\{L_k =p,\tilde{X}^1_{T^{1}_p} = (x,z),\tilde{X}^2_{T^{2}_p} = (x',z')\}$ the points $x$ and $x'$ are $K$-open and using Proposition~\ref{omegaK}, we can replace $\omega$ with $\omega_K$ in $E^{\omega}_{(\mathbf{x}, \mathbf{z})}[f(\tilde{X}^{1}_{\cdot},\tilde{X}^{2}_{\cdot}), \mathcal{D}^{\bullet} = +\infty]$. Notice that $f$ and $\mathcal{D}^{\bullet}$ are measurable with respect to $\sigma \big(  X^1_0,X^2_0,(\tilde{X}^1_{n})_{n \ge 1},(\tilde{X}^2_{n})_{n\ge 1} \big)$ and thus can be written as $\phi(X^1_0,X^2_0,(\tilde{X}^1_{n})_{n \ge 1},(\tilde{X}^2_{n})_{n\ge 1})$ with $\phi$ measurable. Since the law of $(  X^1_0,X^2_0,(\tilde{X}^1_{n})_{n \ge 1},(\tilde{X}^2_{n})_{n\ge 1})$ is the same under $P^{\omega}_{(\mathbf{x}, \mathbf{z})}$ and under $P^{\omega}_{(\mathbf{x}, 0)}$, we replace $z,z'$ by $0$ in $E^{\omega_K}_{(\mathbf{x}, \mathbf{z})}[f(\tilde{X}^{1}_{\cdot},\tilde{X}^{2}_{\cdot}), \mathcal{D}^{\bullet} = +\infty]$.

We observe that $E^{\omega_K}_{\substack{x,x'}}[f(\tilde{X}^{1}_{\cdot},\tilde{X}^{2}_{\cdot}), \mathcal{D}^{\bullet} = +\infty]$ is measurable with respect to 
\begin{equation*}
    \sigma\left(c_{*}(e) \colon e \in \mathcal{R}^{x} \cup \mathcal{R}^{x'} \backslash (\mathcal{E}^1_{x} \cup \mathcal{E}^1_{x'})\right).
\end{equation*}
We highlight the fact that, for this last observation to be true, an essential role is played by the condition $X^i_{n} \in \mathcal{V}_{X^1_{0} - e_1} \cup \mathcal{V}_{X^2_{0} - e_1}$ in \eqref{definition_ENV}. Indeed, without it the sigma algebra above becomes $\sigma(c_{*}(e) \colon e \in \mathcal{R}^{x} \cup \mathcal{R}^{x'} \backslash (\mathcal{E}_{x} \cup \mathcal{E}_{x'}))$ and we would lose the independence properties that follow.
Moreover, $E^{\omega_a}_{\substack{(U_1,z_1)\\(U_2,z_2)}}[h_1,L_k =p,\tilde{X}^1_{T^{1}_p} = (x,z),\tilde{X}^2_{T^{2}_p} = (x',z')]$ is measurable with respect to 
\begin{equation*}
    \sigma \left(c_{*}(e) \text{ with } e \in E(\mathcal{H}^{-}(p)) \cup \mathcal{E}^{2}_{x} \cup \mathcal{E}^{2}_{x'}\right).
\end{equation*}
In addition, $\mathcal{R}^{x} \cup \mathcal{R}^{x'} \backslash (\mathcal{E}^1_{x} \cup \mathcal{E}^1_{x'})$ and $E(\mathcal{H}^{-}(p)) \cup \mathcal{E}^{2}_{x} \cup \mathcal{E}^{2}_{x'}$ can be proved to be disjoint in an elementary way as illustrated in the figure below.
\begin{figure}[H]
    \centering
    \includegraphics[scale = 0.23]{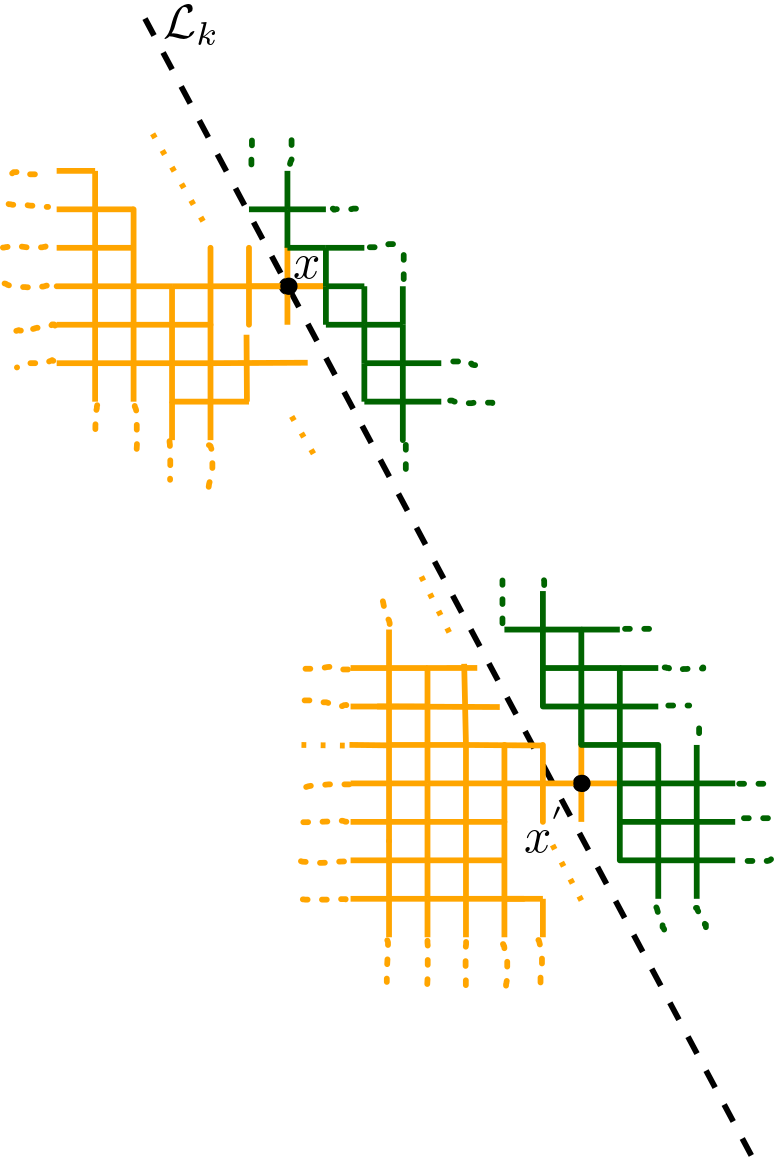}
    \caption{In orange $E(\mathcal{H}^{-}(p)) \cup \mathcal{E}^{2}_{x} \cup \mathcal{E}^{2}_{x'}$ and in green $\mathcal{R}^{x} \cap \mathcal{R}^{x'} \backslash (\mathcal{E}_{x} \cup \mathcal{E}_{x'})$.}
    \label{disjoint_sets}
\end{figure}
\noindent By independence of the two $\sigma$-algebra we deduce that 
\begin{align*}
    &\mathbb{E}^a_{\substack{(U_1,z_1)\\(U_2,z_2)}}\left[f(\tilde{X}^{1}_{T^{1}_{p}+\cdot},\tilde{X}^{2}_{T^{2}_{p}+\cdot})g^{(x,x')}h_1,L_k =p, M_k = +\infty,\tilde{X}^1_{T^{1}_p} = (x,z),\tilde{X}^2_{T^{2}_p} = (x',z')\right]\\
    &=\mathbb{E}^a_{\substack{(U_1,z_1)\\(U_2,z_2)}}\left[h_1,L_k =p,\tilde{X}^1_{T^{1}_p} = (x,z),\tilde{X}^2_{T^{2}_p} = (x',z')\right]\mathbb{E}^{K}_{\substack{x,x'}}\left[gf(\tilde{X}^{1}_{.},\tilde{X}^{2}_{.}),\mathcal{D}^{\bullet} =+\infty\right].
\end{align*}
The last equality, used in the specific case where $f = 1$ and $g = 1$, gives
\begin{equation}\label{eqn:D=inftyOtherSide}
\begin{split}
    &\mathbb{E}^a_{\substack{(U_1,z_1)\\(U_2,z_2)}}\left[h_1,L_k =p, M_k = +\infty,\tilde{X}^1_{T^{1}_p} = (x,z),\tilde{X}^2_{T^{2}_p} = (x',z')\right]\\
    &=\mathbb{E}^a_{\substack{(U_1,z_1)\\(U_2,z_2)}}\left[h_1,L_k =p,\tilde{X}^1_{T^{1}_p} = (x,z),\tilde{X}^2_{T^{2}_p} = (x',z')\right]\Pf^{K}_{\substack{x,x'}}(\mathcal{D}^{\bullet} =+\infty).
\end{split}
\end{equation}
Hence, we can write
\begin{align*}
& \mathbb{E}^a_{\substack{(U_1,z_1)\\(U_2,z_2)}}\left[h_1,L_k =p,\tilde{X}^1_{T^{1}_p} = (x,z),\tilde{X}^2_{T^{2}_p} = (x',z')\right]\mathbb{E}^{K}_{\substack{x,x'}}\left[gf(\tilde{X}^{1}_{\cdot},\tilde{X}^{2}_{\cdot}),\mathcal{D}^{\bullet} =+\infty\right] \\
&= \mathbb{E}^a_{\substack{(U_1,z_1)\\(U_2,z_2)}}\left[h_1,L_k =p, M_k = +\infty,\tilde{X}^1_{T^{1}_p} = (x,z),\tilde{X}^2_{T^{2}_p} = (x',z')\right]\mathbb{E}^{K}_{\substack{x,x'}}\left[gf(\tilde{X}^{1}_{\cdot},\tilde{X}^{2}_{\cdot}) \text{ }| \text{ }\mathcal{D}^{\bullet} =+\infty\right].
\end{align*}
By taking the sum, we obtain
\begin{align*}
    &\somme{\substack{k \geq 0 \\ p \geq 0}, \substack{(x,z)\\(x',z')}}{}{\mathbb{E}^a_{\substack{(U_1,z_1)\\(U_2,z_2)}}\left[h_1,L_k =p, M_k = +\infty,\tilde{X}^1_{T^{1}_p} = (x,z),\tilde{X}^2_{T^{2}_p} = (x',z')\right]\mathbb{E}^{K}_{\substack{x,x'}}\left[gf(\tilde{X}^{1}_{\cdot},\tilde{X}^{2}_{\cdot}) \text{ }| \text{ }\mathcal{D}^{\bullet} =+\infty\right]}\\
    &=  \mathbb{E}^a_{\substack{(U_1,z_1)\\(U_2,z_2)}}\left[h_1\mathbb{E}^K_{X^1_{T^{1}_{\mathcal{L}_1}},X^2_{T^{2}_{\mathcal{L}_1}}}\left[f(\tilde{X}^1_{\cdot},\tilde{X}^{2}_{\cdot})g \text{ } | \text{ } \mathcal{D}^{\bullet} = +\infty\right]\right].
\end{align*}
For the induction, if the formula has been proved for $k$, from Proposition~\ref{PropFrequentRegLevels} we get that $\mathcal{L}_{k+1}$ is almost surely finite. The proof works in a similar way to go from $k$ to $k+1$ and the details of the adaptation can be found in \cite{Lian}.
\end{proof}
\noindent A consequence of this theorem is the analogous of Proposition~\ref{L1decay} but in a more general setting.
\begin{proposition}\label{PropositionFrequentRegeneration}
For any $M > 0$, $U_1,U_2 \in \mathbb{Z}^d$ and any $k > 0$,
\begin{equation*}
    \Pf_{U_1,U_2}(\mathcal{L}_{k+1} - \mathcal{L}_k \geq n) \leq Cn^{-M},
\end{equation*}
where $C > 0$ only depends on $\lambda,\vec{\ell},d,K$.
\end{proposition}

\begin{proof}
We write 
\begin{align*}
    \Pf_{U_1,U_2}(\mathcal{L}_{k+1} - \mathcal{L}_k \geq n)  &= \somme{p,x,x'}{}{\mathbb{E}_{U_1,U_2}\left[\mathcal{L}_k = p , X^{1}_{T^1_p} = x, X^{2}_{T^2_p} = x', \mathcal{L}_{k+1} - \mathcal{L}_k \geq n\right]}\\
    & = \somme{p,x,x'}{}{\mathbb{E}_{U_1,U_2}\left[\mathcal{L}_k = p , X^{1}_{T^1_p} = x, X^{2}_{T^2_p} = x', \mathcal{L}_1 \circ \theta_{T^1_{x},T^2_{x'}} - p  \geq n\right]}\\
    & = \somme{p,x,x'}{}{\mathbb{E}_{U_1,U_2}\left[\mathcal{L}_k = p , X^{1}_{T^1_p} = x, X^{2}_{T^2_p} = x',\Pf^{K}_{x,x'}(\mathcal{L}_1 - p \geq n \text{ } | \text{ }\mathcal{D}^{\bullet} =+\infty)\right]},
\end{align*}
where the third equality follows from Theorem~\ref{IndependenceJointReg}. We only have to dominate $\Pf^{K}_{x,x'}(\mathcal{L}_1 - p \geq n \text{ } | \text{ }\mathcal{D}^{\bullet} =+\infty)$ to conclude.

To do so, using the fact that $\Pf_{x,x'}^{K}(\mathcal{D}^{\bullet} =+\infty) > \alpha > 0$ with $\alpha$ which does not depend on $x$ and $x'$, we can see it suffices to bound $\Pf^{K}_{x,x'}(\mathcal{L}_1 - p \geq n )$. Set $y = x' - x$, by the invariance in law of the environment under translations, we have $\Pf^{K}_{x,x'}(\mathcal{L}_1 - p \geq n ) = \Pf^{K}_{0,y}(\mathcal{L}_1 -p + x\cdot\vec{\ell} \geq n ) $. Notice that we have $|x\cdot \vec{\ell} - p| < e_{1} \cdot \vec{\ell}$, which follows from the fact that $X^1_{T^{1}_p} = x$. The result follows from Proposition~\ref{L1decaybis}.
\end{proof}

\section{Asymptotic separation}

The aim of this section is to prove the following result: the trajectories of the two walks do not cross in the half-space $\mathcal{H}^{+}_R$, with high probability as $R$ grows. We recall that $\mathcal{H}^{+}_{R} = \{x \in \mathbb{Z}^d \text{ }|\text{ } x\cdot \vec{\ell} > R\}$ and let $\{X_n^i\} \coloneqq \{ z \in \mathbb{Z}^d \colon \exists n \ge 0 \text{ s.t. } X^i_n = z\}$ denote the range of the $i$-th random walk. We introduce the event that the traces of the walks in $\mathcal{H}^{+}_R$ get close
\begin{align} \label{definition_MN}
    \mathbf{M}_R \coloneqq \bigg\{ \inf_{\substack{x \in \mathcal{H}^{+}_{R}\cap \{X_n^1\}\\y \in \mathcal{H}^{+}_{R}\cap \{X_n^2\}}}\|x-y\|_{1}\leq 2\bigg\}.
\end{align}
\begin{proposition} \label{LemmaDistantWalks}
For any $R \ge 0$, it holds that
\begin{equation*}
    \mathbb{P}_{0}(\mathbf{M}_R ) \le CR^{-c},
\end{equation*}
where the constants $C,c > 0$ only depend on $\lambda,\vec{\ell},d$.
\end{proposition}

In order to prove this result we will adapt methods from \cite{Berger_Zeitouni}. First, we will deal with the case when the two walks $X^1$ and $X^2$ live in independent environments. Then, when the two walks live in the same environment, we will use the fact that when they do not cross they behave like if they were in independent environments. 

\subsection{Independent environments}
First, we recall some notations from Section~\ref{SectionNotation}:
\begin{itemize}
    \item The law $Q_{U_1,U_2}$ is defined as $Q_{U_1,U_2} = \Pf_{U_1} \otimes \Pf_{U_2}$. Hence $X^1$ and $X^2$ are independent under $Q_{U_1,U_2}$. We also have $Q^{K}_{U_1,U_2} =\Pf^{K}_{U_1} \otimes \Pf^{K}_{U_2}$.
    \item We write $Q_{U_1,U_2}(\cdot \, | \, D^{\otimes}=+\infty) = \Pf_{U_1}(\cdot \, | \, D=+\infty) \otimes \Pf_{U_2}(\cdot \, | \, D=+\infty)$ where $D$ is defined in \eqref{eqn:DSingleRegeneration} in the definition of regeneration time. Analogously, $Q^{K}_{U_1,U_2}(\cdot \, | \, D^{\otimes}=+\infty) = \Pf^{K}_{U_1}(\cdot \, | \, D=+\infty) \otimes \Pf^{K}_{U_2}(\cdot \, | \, D=+\infty)$.
\end{itemize}

\begin{proposition} \label{PropositionDistantIndepEnv}
There exists $K_0 > 1$ such that, for any $U_1, U_2 \in \mathbb{Z}^d$ and $K \ge K_0$, we have
\begin{equation*}
    Q^{K}_{U_1, U_2}\bigg(\inf_{\substack{x \in  \{X_n^1\}\\y \in  \{X_n^2\}}}\|x-y\|_{1}\leq 2 \, \Big| \, \textnormal{}D^{\otimes}=+\infty\bigg) \leq C\|U_1 - U_2\|^{-c},
\end{equation*}
where the constants $C,c > 0 $ only depend on $\lambda,\vec{\ell},d,K$.
\end{proposition}
\begin{remark1}
We observe that if $U_1 = U_2$ the statement of Proposition~\ref{PropositionDistantIndepEnv} remains true with the convention $\|U_1-U_2\|^{-c} = +\infty$.
\end{remark1}

\begin{proof}
For $z \in \mathbb{Z}^d$ we define
\begin{equation}\label{eqn:DefinitionF}
F_i(z) = \Pf^{K}_{U_i}\left(\exists k, X_k^{i} = z \, | \, \text{}D=+\infty \right),
\end{equation}
and
\begin{equation}
F_i^{(R)}(z) = F_i(z)\indi{\{\|z-U_i\|> R\}}.
\end{equation}
\paragraph{Step 1: bounding our quantity using $\|F_i^{(R)}\|_2$ and $\|F^i\|_2$.}

Fix $R =  \|U_1-U_2\|/3$. If $\|X^{1}_p-X^2_q\|_1 \le 2$ then necessarily either $\|U_1-X^1_p\| \ge R$ or $\|U_2-X^2_q\| \ge R$, otherwise we would have $\|U_1-U_2\| \le 2/3\|U_1-U_2\|+  3 < \|U_1-U_2\|$ (as soon as $\|U_1-U_2\| \ge 10$). It follows that
\begin{equation}\label{eqn:FirstBound2WalksIndEnv}
\begin{split}
    Q^{K}_{U_1,U_2} &\left(\exists p,q \in \mathbb{N} \colon \|X_p^1 - X_q^2\|_{1} \leq 2 \, | \, \text{}D^{\otimes}=+\infty \right)\\&\le Q^{K}_{U_1,U_2}\left(\exists p,q \in \mathbb{N} \colon \|X_p^1 - X_q^2\|_{1} \leq 2 \text{ and } \|U_1-X^1_p\| \ge R \, | \, \text{}D^{\otimes}=+\infty \right)\\
   & + Q^{K}_{U_1,U_2}\left(\exists p,q \in \mathbb{N} \colon \|X_p^1 - X_q^2\|_{1} \leq 2 \text{ and } \|U_2-X^2_q\| \ge R \, | \, \text{}D^{\otimes}=+\infty \right).
\end{split}
\end{equation}
We write
\begin{align}
  Q^{K}_{U_1,U_2} &\left(\exists p,q \in \mathbb{N} \colon \|X_p^1 - X_q^2\|_{1} \leq 2 \text{ and } \|U_1-X^1_p\| \ge R \, | \, \text{}D^{\otimes}=+\infty \right)\nonumber\\
 &\le \somme{\|z-U_1\|\ge R}{}{Q^{K}_{U_1,U_2}\left(\exists p,q \in \mathbb{N} \colon X^1_p = z, X^2_q = y, \| z - y \|_{1} \le 2\, | \, \text{}D^{\otimes}=+\infty \right)}\nonumber\\
     & \le  \somme{u \in  \mathbb{Z}^d \colon \|u\|_1 = 2 }{}{}\somme{\|z-U_1\| \ge R}{}{F_1^{(R)}(z)F_2^{}(z+u)} \nonumber\\
     &\le  \somme{u \in  \mathbb{Z}^d \colon \|u\|_1 = 2}{}{}\somme{z \in \mathbb{Z}^d}{}{F_1^{(R)}(z)F_2^{}(z+u)} \nonumber\\
     & \le C\|F^{(R)}_1\|_2\|F_2\|_2, \label{eqn:CauchySchwarz}
\end{align}
where \eqref{eqn:CauchySchwarz} follows from Cauchy-Schwarz inequality and the constant $C > 0$ depends only on $\lambda,\vec{\ell},d,K$. Symmetrically, we can do the same computation for the other term in the r.h.s.\ of \eqref{eqn:FirstBound2WalksIndEnv}. To conclude the proof of the proposition we will show that $\|F_i\|_2 < +\infty$ and $\|F_i^{(R)}\|_2 \le CR^{-c}$ with $C,c > 0$ constants which only depend on $\lambda,\vec{\ell},d,K$.

\paragraph{Step 2: bound on $\|F_i\|_2$ and $\|F_i^{(R)}\|_2$.}

For $n \in \mathbb{N}$ we introduce $\{\tau_n^{i}\}_{n \ge 1}, \, i =1, 2$ to be the sequence of regeneration times associated to the $i$-th walk, in the sense of \eqref{eqn:RegTimeSingleWalk}. Note that, by \cite[Theorem 5.2]{Kious_Frib}, both sequences are almost surely well defined (no $\tau_n^{i}$ is infinite) and they are independent of each other as we are considering the measure $Q^{K}_{U_1, U_2}$. Let us define
\begin{equation*}
H_i(z,n) = \Pf^{K}_{U_i}\left(\exists \tau_n^{i} \leq k \leq \tau_{n+1}^{i} \text{, } X_k^{i} = z \, | \, \text{}D=+\infty\right),
\end{equation*}
and
\begin{equation*}
H_i^{(R)}(z,n) = H_i(z,n)\indi{\{ \|z-U_i \|> R \}}.
\end{equation*}
From \eqref{eqn:DefinitionF}, for all $z \in \mathbb{Z}^d$, we have that
\begin{equation*}
    F_i(z) = \sum_{n=0}^{+\infty} H_i(z,n).
\end{equation*}
Using the triangle inequality we write
\begin{equation}\label{eqn:definitionF_{i}}
 \|F_i\|_{2} \leq \sum_{n=0}^{+\infty} \|H_i(\cdot,n)\|_{2},
\end{equation}
and, similarly,
\begin{equation}\label{eqn:definitionF_{i}(R)}
 \|F_i^{(R)}\|_{2} \leq \sum_{n=0}^{+\infty} \|H_i^{(R)}(\cdot,n)\|_{2}.
\end{equation}
We want to bound $\|H_i(\cdot,n)\|_2$ and $\|H_i^{(R)}(\cdot,n)\|_2$. The quantity $H_i(\cdot,n)$ can be rewritten as a convolution because $X_{\tau_n^{i}}^{i}$ and $\{ X_{\tau_n^{i} + \, \cdot }^{i} - X_{\tau_n^{i}}^{i} \}$ are independent by \cite[Theorem 5.4]{Kious_Frib}. Thus
\begin{equation}
    H_i(\cdot,n) = T_i^{n} \ast J,
\end{equation}
where
\begin{equation*}
    T_i^{n}(z) = \Pf_{U_i}^{K}(X^{i}_{\tau^{i}_n} = z),
\end{equation*} 
and 
\begin{equation*}
 J(z) = \Pf_{0}^{K}\left(\exists  0 \leq k \leq \tau_{1}^{i} \text{, } X_k^{i} = z \, | \, \text{}D=+\infty\right).
\end{equation*}
Young's inequality gives
\begin{equation}\label{Young}
\|H_i(\cdot,n)\|_{2} \leq \| T_i^{n}\|_{2} \|J\|_{1}. 
\end{equation}
We claim that $\|J\|_{1} < +\infty$, indeed
\begin{align*}
    \somme{z \in \mathbb{Z}^d}{}{\Pf_{0}^{K}\left(\exists  0  \leq k \leq \tau_{1}^{i} \text{, } X_k^{i}  = z \, | \, \text{}D=+\infty\right)} \le C\mathbb{E}_{0}^{K}\bigg[\somme{z \in \mathbb{Z}^d}{}{\indi{\{\exists  0  \leq k \leq \tau^i_{1} \text{, } X_k^{i} = z\}}} \bigg],
\end{align*}
where $C$ is a constant which only depend on $\lambda,\vec{\ell},d,K$ using \cite[Lemma 5.1]{Kious_Frib}. The term in the last expectation is equal to $|\{X_k^{i} \colon 0\le k \le \tau_1^{i}\}|$. Let us recall the notation \eqref{recall_chi}, for $i = 1, 2$ let $\chi^{i}_{k}$ be the same quantity for the $i$-th walk and let $\chi^{i} = \chi^{i}_{0}$.
We have $|\{X_k^{i} \colon 0\le k \le \tau_1^{i}\}| \le (\chi^i)^{\beta}$ for $\beta$ large enough which only depends on $d,\alpha$. By \cite[Lemma 6.1]{Kious_Frib}, for $K$ large enough, $\chi^i$ has finite moment of order $\beta$. It follows that $\|J\|_1 < +\infty$.

For the term $\|T_i^{n}\|_2$, using Proposition~\ref{TechnicalInequality}, we have for $K$ large enough, 
\begin{equation}\label{uniform_inequality}
    \displaystyle \sup_{z \in \mathbb{Z}^d}\Pf_{U_i}^{K}(X_{\tau_n^{i}}^{i} = z) \leq Cn^{-\frac{d}{2}+\varepsilon},
\end{equation}
where $C$ only depends on $\lambda,\vec{\ell},d,K$. We get the bound 
\begin{equation*}\begin{split}
\|T_i^{n}\|_{2} &\leq \bigg(\somme{x\in \mathbb{Z}^d}{}{\Pf_{U_i}^{K}(X_{\tau_n^{i}}^{i} = z)^2}\bigg)^{\frac{1}{2}} \\
& \leq \bigg(\sup_{z \in \mathbb{Z}^d}\Pf_{U_i}^{K}(X_{\tau_n^{i}}^{i} = z)\somme{x\in \mathbb{Z}^d}{}{\Pf_{U_i}^{K}(X_{\tau_n^{i}}^{i} = z)}\bigg)^{\frac{1}{2}}\\
& \underset{\eqref{uniform_inequality}}{\leq}  Cn^{-\frac{d}{4}+\frac{\varepsilon}{2}}.
\end{split}
\end{equation*}
Inserting this estimate in \eqref{Young}, together with $\|J\|_{1} < +\infty$, we obtain 
\begin{align}\label{first_bound}
    \|H_i(\cdot,n)\|_2 \le Cn^{-\frac{d}{4}+\frac{\varepsilon}{2}}.
\end{align}
As a consequence, by \eqref{eqn:definitionF_{i}}, in dimension $d \geq 5$ we get
\begin{equation}\label{eqn:F2Finite}
    \|F_i\|_{2} < +\infty.
\end{equation}
We still have to bound $\|H_i^{(R)}(\cdot,n)\|_{2}$. A first bound is given by \eqref{first_bound}, however for $n$ small this bound is not precise enough, thus we compute another bound which is useful for small $n$. Let us bound $\|H_i^{(R)}(\cdot,n)\|_{1}$ and then use the inequality $\|H_i^{(R)}(\cdot,n)\|_{2} \leq \|H_i^{(R)}(\cdot,n)\|_{1}^{1/2}$.
We can rewrite
\begin{align}
    \|H_i^{(R)}(\cdot,n)\|_{1} &=  \sum_{\|z-U_{i}\| \geq R}\mathbb{E}^{K}_{U_i}\left[ \indi{ \left\{\exists \tau_n^{i} \leq k \leq \tau_{n+1}^{i} \colon X_k^{i} = z \right\} } \, \big| \, \text{}D=+\infty  \right] \nonumber \\ &= \mathbb{E}^{K}_{U_i}\left[ \sum_{\|z - U_i\| \geq R}\indi{ \left\{\exists \tau_n^{i} \leq k \leq \tau_{n+1}^{i} \colon X_k^{i} = z \right\} } \, \big| \, \text{}D=+\infty \right]. \label{expectation}
\end{align}
Define $A(n,R) \coloneqq \{ \exists \, \tau_n \le k \le \tau_{n+1} \textnormal{ such that } |X^i_k-X^i_0| \ge R\}$. The term inside the expectation \eqref{expectation} is bounded by $\indi{A(n,R)}|\mathcal{B}(\chi^{i}_n,(\chi^{i}_n)^{\alpha})|$; indeed, it is clearly bounded by $|\mathcal{B}(\chi^{i}_n,(\chi^{i}_n)^{\alpha})|$ and under $A(n,R)^{c}$ it is equal to $0$.
\begin{itemize}
    \item Using the fact that $\mathrm{diam}(\mathcal{B}(L,L^{\alpha})) \le CL^{\alpha}$ where $C$ is a constant which only depends on $\vec{\ell},d$, we have the following inclusion: 
\begin{align*}
    A(n,R) \subset \bigg\{ \somme{k=0}{n}{C(\chi^{i}_k})^{\alpha} \ge R  \bigg\}.
\end{align*}
Since $((\chi^{i}_{k})^{\alpha})_{k \ge 0}$ is a family of independent identically distributed random variables, and for $K$ large enough $(\chi^{i}_0)^{\alpha}$ has finite first moment under $\Pf^K_{U_i}(\cdot \text{ | }D=+\infty)$, by \cite[Lemma 6.1]{Kious_Frib}, we have 
\begin{align*}
    \Pf^{K}_{U_i}\left(A(n,R) \text{ | }D=+\infty\right) &\le \Pf^{K}_{U_i}\left(\somme{k=0}{n}{C(\chi^{i}_k)^{\alpha}} \ge R\text{ | }D=+\infty\right)\\
    &\le C\frac{(n+1)\, \mathbb{E}_{U_i}^{K}\left[(\chi^{i}_{1})^{\alpha}\text{ | }D=+\infty\right]}{R}  = \frac{Cn}{R},
\end{align*}
where $C > 0$ is a constant which only depends on $\lambda,\vec{\ell},d,K$.
    \item[$\bullet$] Using again \cite[lemma 6.1]{Kious_Frib}, for any $M > 0$ there exists $K_0 > 1$ such that for $K \ge K_0$ the variable $\chi^i_n$ has finite moment of order $M$. By chosing $K_0$ big enough, $|\mathcal{B}(\chi^{i}_n,(\chi^{i}_n)^{\alpha})|$ has finite second moment. Indeed, $|\mathcal{B}(\chi^{i}_n,(\chi^{i}_n)^{\alpha})| \le C(\chi^i_n)^{d\alpha}$ with $C$ a constant which only depends on $d$.
\end{itemize}
Using Cauchy-Schwarz inequality we have
\begin{align}
\nonumber \|H_i^{(R)}(\cdot,n)\|_{1} &\leq \Pf_{U_i}^{K}\left( A(n,R)\text{ | }D=+\infty\right)^{\frac{1}{2}}\mathbb{E}^{K}_{U_i}\left[|\mathcal{B}(\chi^{i}_n,(\chi^{i}_n)^{\alpha})|^2 \text{ | }D=+\infty\right]^{\frac{1}{2}}\\
& \label{second_bound}\leq Cn^{\frac{1}{2}}R^{-\frac{1}{2}},
\end{align}
where $C$ only depends on $\lambda,\vec{\ell},d,K$. It follows that we have
\begin{align*}
    \|H_i^{(R)}(\cdot,n)\|_{2} \le \|H_i^{(R)}(\cdot,n)\|^{\frac{1}{2}}_{1} \le  Cn^{\frac{1}{4}}R^{-\frac{1}{4}}.
\end{align*}
Pick $a > 0$ arbitrarily small. We have that
\begin{align*}
    \|F_i^{(R)}\|_{2} &\leq \somme{n=1}{\ent{R^{a}}}{\|H_i^{(R)}(\cdot,n)\|_{2}} +\somme{n=\ent{R^{a}}+1}{+\infty}{\|H_i^{(R)}(\cdot,n)\|_{2}}\\&\underset{\eqref{first_bound},\eqref{second_bound}}{\leq} C \left(\displaystyle \sum_{n = 1}^{\ent{R^{a}}}n^{\frac{1}{4}}R^{-\frac{1}{4}} + \sum_{n = \ent{R^{a}} + 1}^{+\infty}n^{-\frac{d}{4}+\frac{\varepsilon}{2}} \right)\\
    &\le CR^{-\frac{1}{4}}R^{\frac{5}{4}a} + C(R^{a})^{-\frac{d}{4}+1+\frac{\varepsilon}{2}}.
\end{align*}
By taking $a$ small enough such that $\frac{5}{4}a-\frac{1}{4} < 0$ and since $-\frac{d}{4}+1+\frac{\varepsilon}{2} < 0$ we have
\begin{equation}\label{eqn:BoundF(R)}
       \|F_i^{(R)}\|_{2} \leq CR^{-c},
\end{equation}
where $C,c > 0$ only depend on $\lambda,\vec{\ell},d,K$. To conclude one needs to combine statements \eqref{eqn:FirstBound2WalksIndEnv} and \eqref{eqn:CauchySchwarz} with \eqref{eqn:F2Finite} and \eqref{eqn:BoundF(R)}.
\end{proof}
We introduce the notation:
\begin{align}\label{definition_R0}
    R_0 \coloneqq \max(X^{1}_0\cdot \vec{\ell},X^{2}_0 \cdot \vec{\ell}).
\end{align}
The following proposition gives us a result equivalent to Proposition~\ref{LemmaDistantWalks} but considering two walks evolving in independent environments.
\begin{proposition} \label{CorollaryIndepEnv}
There exists $K_0 > 1$ such that, for any $U_1,U_2 \in \mathbb{Z}^d$, $K \ge K_0$ and $R > 0$, we have 
\begin{equation*}
Q^{K}_{U_1,U_2}\left(\mathbf{M}_{R+R_0} \right) \leq CR^{-c},
\end{equation*}
where the constants $C,c > 0$ only depend on $\lambda,\vec{\ell},d,K$.
\end{proposition}
\begin{proof}
Let us fix $M$ large enough. We introduce $A_{i}(R) = \{\|X^i_{0}-X^i_{\tau^i_{1}}\|_1 \le R \}$. Then by \cite[Lemma 6.1]{Kious_Frib}, for $K$ large enough we have 
\begin{align}\label{control_first_ref}
    \Pf^{K}_{U_1,U_2}(A_{i}(R)^{c}) \le cR^{-M},
\end{align}
where $c > 0$ only depends on $\lambda,\vec{\ell},d,K$. Hence, we can write 
\begin{align*}
Q^{K}_{U_1,U_2}\left(\mathbf{M}_{R+R_0}\right) \le  Q^{K}_{U_1,U_2}\left(\mathbf{M}_{R+R_0},A_1(R/2),A_2(R/2)\right) + cR^{-M}.
\end{align*}
We observe that $\mathbf{M}_{R+R_0}\cap A_1(R/2)\cap A_2(R/2)$ implies $\mathbf{M}_{R/2+R_0} \circ \theta_{\tau^1_1,\tau^2_1}$. By \cite[theorem 5.4]{Kious_Frib}, we only have to dominate
\begin{align} \label{term_to_dom_to_conclue}
    \mathbb{E}^{Q^K_{U_1,U_2}}\bigg[Q^K_{X^1_{\tau^1_1},X^2_{\tau^2_2}}\bigg(\mathbf{M}_{R/2+R_0} \text{ | }D^{\otimes}=+\infty\bigg)\bigg].
\end{align}
Using the same notations as in the last proof, for any $x,y \in \mathbb{Z}^d$ we have
\begin{align*}
   Q^{K}_{x,y}\left(\mathbf{M}_{R/2+R_0}\, | \,\text{}D^{\otimes}=+\infty \right) &\le  \somme{u \in \mathbb{Z}^d \colon \|u\|_1 = 2 }{}{} \somme{\substack{z \in \mathbb{Z}^d\\ \|z-x\| \ge R/2 \\ \|z-y\| \ge R/2}}{}{F_1^{(R/2)}(z)F_2^{(R/2)}(z + u)}\\
   &\le \somme{u \in \mathbb{Z}^d \colon \|u\|_1 = 2 }{}{} \somme{\substack{z \in \mathbb{Z}^d}}{}{F_1^{(R/2)}(z)F_2^{(R/2 - 1)}(z + u)} \\&\le C \|F_1^{(R/2)} \|_2 \| F_2^{(R/2 - 1)} \|_2
   \\&\underset{\eqref{eqn:BoundF(R)}}{\le} CR^{-c},
\end{align*}
where the third step holds by Cauchy-Schwarz inequality. The constants $C,c > 0$ only depend on $\lambda,\vec{\ell},d,K$. We dominate \eqref{term_to_dom_to_conclue} using this last inequality and conclude the proof.
\end{proof}
The last result shows that there is little probability that the two walks intersect on points which are far from the starting points. In other words, with high probability, if the two walks do intersect, then it has to be close to one of the starting points.

We need a last proposition before moving on to the case where both walks are in the same environment. We recall the notation $T^i_{B}$, $\mathcal{H}^{+}_R$ defined in \eqref{T_B},\eqref{H+-}. Let us recall the notation $M$ defined in \eqref{definition_M_other}. For $R > 0$, we define the event
\begin{equation} \label{Definition_event_A}
    A(R) \coloneqq \bigg\{M \ge R+R_0\bigg\}\cap \bigg\{\|X_p^1 - X_q^2\|_{1} \geq 3 \quad \forall 2 < p \leq T^1_{\mathcal{H}^{+}_{R+R_0}}  \,\, \textnormal{and} \,\, \forall 2 < q \leq T^2_{\mathcal{H}^{+}_{R+R_0}} \bigg\}.
\end{equation}

\begin{proposition}\label{PropositionEventA(R)Indep}
There exists $\Lambda > 0$ and $K_0 > 1$ such that, for $K \ge K_0$, $U_1,U_2 \in \mathbb{Z}^d$ such that $(U_1-U_2) \in \mathcal{U}$, $\|U_1-U_2\| \ge \Lambda$ and $R > 0$,
\begin{equation*}
Q^{K}_{U_1,U_2}(A(R)) \geq \rho > 0,
\end{equation*}
where the constant $\rho$ only depends on $\lambda,\vec{\ell},d,K$ where $\mathcal{U}$ is defined in \eqref{eqn:USet}.
\end{proposition}

\begin{proof}
We have $Q^{K}_{U_1,U_2}(A(R)) \geq cQ^{K}_{U_1,U_2}(A(R) \, | \, \text{}D^{\otimes}=+\infty)$ with $c = \Pf^{K}_{0}(D=+\infty)^2>0$.
Thus, under the law $Q^{K}_{U_1,U_2}( \cdot\, | \, \text{}D^{\otimes}=+\infty)$ the event $A(R)^{c}$ implies $ \left\{ \exists \, p,q \in \mathbb{N}, \,\, \|X_p^1 - X_q^2\|_{1} \leq 2 \right\}$. It follows from Proposition~\ref{PropositionDistantIndepEnv} that one has the bound
\begin{align*}
Q^{K}_{U_1,U_2}(A(R)^{c} \, | \, \text{}D^{\otimes}=+\infty) &\leq Q^K_{U_1,U_2}\left(\exists p,q \in \mathbb{N}, \,\, \|X_p^1 - X_q^2\|_{1} \leq 2 \, | \, \text{}D^{\otimes}=+\infty \right) \\ &\leq C\|U_1 - U_2\|^{-c}.
\end{align*}
The result is deduced by choosing $\Lambda$ large enough.
\end{proof}

\subsection{Identical environments}

We will exploit the results obtained in the last part in order to deduce estimates for two walks in the same environment. When the two walks are far enough they live in independent environments and so, by using the previous subsection, we will be able to show that they have a positive probability to not hit each other anymore. If they intersect, then we wait for the walks to get far enough and we iterate this way.

First, let us prove a general proposition which enables us to relate two walks in the same environment to two walks in independent environments. Let us recall the notation $\mathcal{E}_A$ defined in \eqref{EquationIncidentEdges}. We introduce
\begin{equation}
    \label{def_trib_temps_inter}
    \begin{split}
        \mathcal{F}_n &\coloneqq \sigma \left(\left\{(\tilde{X}^{1}_{k})_{0\le k \le n},(\tilde{X}^{2}_{k})_{0\le k \le n}\right\} \cup \left\{c^{*}(e)  \colon e \in \mathcal{E}_{(\tilde{X}^{1}_{k})_{0\le k \le n-1}}\cup \mathcal{E}_{(\tilde{X}^{2}_{k})_{0\le k \le n-1}}\right\} \right) \text{   and   }\\ T &\coloneqq \inf\Bigg\{n \ge 0 \colon \inf_{\substack{x \in  \{X_k^1\}_{0\le k \le n}\\y \in  \{X_k^2\}_{0\le k \le n}}}\|x-y\|_{1}\leq 2 \Bigg\}.
    \end{split}
\end{equation} 

\begin{proposition}\label{from_same_to_ind_env}
For any positive function $f$ measurable with respect to $\mathcal{F}_{T}$ and $U_1,U_2 \in \mathbb{Z}^{d}$ we have:
\begin{align*}
    \mathbb{E}^{K}_{U_1,U_2}(f) = \mathbb{E}^{Q_{U_1,U_2}^{K}}(f).
\end{align*}
\end{proposition}

\begin{proof}
First, we restrict to the case $f = \indi{(\tilde{X}^1_{k})_{0 \le k \le n} = (a_k)_{0 \le k \le n}}g_1\indi{(\tilde{X}^2_{k})_{0 \le k \le n} = (b_k)_{0 \le k \le n}}g_2 $ where the paths $(a_k),(b_k)$ are nearest neighbour paths and for $ 0 \le i,j \le n-1$ we have $\|a_i-b_j\|_1 \ge 3$ and $g_1, g_2$ are measurable with respect to $\sigma\big(c^{*}(e) \colon e \in \mathcal{E}_{\{a_k\}_{0\le k \le n-1}}\big)$ and $\sigma\big(c^{*}(e) \colon e \in \mathcal{E}_{\{b_k\}_{0\le k \le n-1}}\big)$, respectively. We can write
\begin{align} \label{generate}
    \mathbb{E}^{K}_{U_1,U_2}(f) &= \mathbf{E}\left[g_1P^{\omega}_{U_1}\left((\tilde{X}^1_{k})_{0 \le k \le n} = (a_k)_{0 \le k \le n}\right)g_2P^{\omega}_{U_2}\left((\tilde{X}^2_{k})_{0 \le k \le n} = (b_k)_{0 \le k \le n}\right)\right].
\end{align}
The term $g_1P^{\omega}_{U_1}\big((\tilde{X}^1_{k})_{0 \le k \le n} = (a_k)_{0 \le k \le n}\big)$ is measurable with respect to $\sigma\big(c^{*}(e) \colon e \in \mathcal{E}_{\{a_k\}_{0\le k \le n}}\big)$. The term $g_2P^{\omega}_{U_2}\big((\tilde{X}^2_{k})_{0 \le k \le n} = (b_k)_{0 \le k \le n}\big)$ is measurable with respect to $\sigma\big(c^{*}(e) \colon e \in \mathcal{E}_{\{b_k\}_{0\le k \le n}}\big)$.
Now since $\mathcal{E}_{\{a_k\}_{0\le k \le n}}\cap \mathcal{E}_{\{b_k\}_{0\le k \le n}} = \emptyset$ and by independence of the conductances we can write \eqref{generate} as
\begin{align*}
    \mathbf{E}\left[g_1P^{\omega}_{U_1}\left((\tilde{X}^1_{k})_{0 \le k \le n} = (a_k)_{0 \le k \le n}\right)\right]\mathbf{E}\left[g_2P^{\omega}_{U_2}\left((\tilde{X}^1_{k})_{0 \le k \le n} = (b_k)_{0 \le k \le n}\right)\right] = \mathbb{E}^{Q_{U_1,U_2}^{K}}(f).
\end{align*}
This concludes the proof since, in the general case, $f$ can be written as an increasing limit of linear combinations of functions of this type.
\end{proof}

\noindent We state the analogous of Proposition~\ref{PropositionEventA(R)Indep}. We recall the definition of $A(R)$ at \eqref{Definition_event_A} and $\mathcal{U}$ at \eqref{eqn:USet}.

\begin{proposition}\label{PropositionJointRegTime}
There exists $\Lambda > 0$, $K_0 > 1$ such that, for any $K \ge K_0$, $U_1,U_2 \in \mathbb{Z}^d$ with $(U_1-U_2) \in \mathcal{U}$, $\|U_1-U_2\| \geq \Lambda$,
\begin{equation}
\Pf_{U_1,U_2}^{K}(A(R) \text{ }| \text{ } \mathcal{D}^{\bullet} = +\infty)\geq \rho,
\end{equation}
the constant $\rho$ only depends on $\lambda,\vec{\ell},d,K$.
\end{proposition}
\begin{proof}
 We recall the notation $T$ and $\mathcal{F}_n$ defined in \eqref{def_trib_temps_inter}. Fix $\|U_1-U_2\| \ge \Lambda$ with $\Lambda$  large enough given by Proposition~\ref{PropositionEventA(R)Indep}. Then under $\Pf_{U_1,U_2}^{K}$, event $A(R)$ coincides with
\begin{align*}
    \{M \ge R+R_0\}\cap \bigg\{\|X_p^1 - X_q^2\|_{1} \geq 3 \quad &\forall 0 \le  p \leq T^1_{\mathcal{H}^{+}_{R+R_0}}  \,\, \textnormal{and} \,\, \forall0 \le  q \leq T^2_{\mathcal{H}^{+}_{R+R_0}} \bigg\}.
\end{align*}
We replace $A(R)$ by this new definition in the proof. Observe that $A(R) = \{M \ge R+R_0\}\cap\{T^1_{\mathcal{H}^{+}_{R+R_0}} \le T\}\cap \{T^2_{\mathcal{H}^{+}_{R+R_0}} \le T\}$. Furthermore,
\begin{itemize}
    \item[$\bullet$] $\{M \ge R+R_0\}$ is measurable with respect to $\sigma\Big((\tilde{X}^i_{k})_{0\le k \le T^i_{\mathcal{H}^{+}_{R+R_0}}} \text{ , }i =1,2\Big)$. On $\{T^1_{\mathcal{H}^{+}_{R+R_0}} \le T\}\cap \{T^2_{\mathcal{H}^{+}_{R+R_0}} \le T\}$ we can replace $\{M \ge R+R_0\}$ with an event $B \in \mathcal{F}_T$.
    \item[$\bullet$] $T^1_{\mathcal{H}^{+}_{R+R_0}}$ and $T^2_{\mathcal{H}^{+}_{R+R_0}}$ are both stopping times with respect to the filtration $(\mathcal{F}_n)_{n \ge 0}$. We deduce $\{T^1_{\mathcal{H}^{+}_{R+R_0}} \le T\}\cap \{T^2_{\mathcal{H}^{+}_{R+R_0}} \le T\} \in \mathcal{F}_T$.
\end{itemize}
From these two observations we have $B\cap\{T^1_{\mathcal{H}^{+}_{R+R_0}} \le T\}\cap \{T^2_{\mathcal{H}^{+}_{R+R_0}} \le T\} \in \mathcal{F}_T$. We write
\begin{align} \label{go_to_ind}
    \nonumber\Pf_{U_1,U_2}^{K}(A(R)) &=\Pf_{U_1,U_2}^{K}\bigg(B\cap\{T^1_{\mathcal{H}^{+}_{R+R_0}} \le T\}\cap \{T^2_{\mathcal{H}^{+}_{R+R_0}} \le T\} \bigg)\\&\underset{\ref{from_same_to_ind_env}}{=}Q_{U_1,U_2}^{K}\bigg(B\cap\{T^1_{\mathcal{H}^{+}_{R+R_0}} \le T\}\cap \{T^2_{\mathcal{H}^{+}_{R+R_0}} \le T\} \bigg)\\ \nonumber&= Q_{U_1,U_2}^{K}(A(R)).
\end{align}
By Lemma~\ref{LemmaDpositive} we have
\begin{align} \label{to_control}
    \Pf_{U_1,U_2}^{K}(A(R) \text{ }| \text{ } \mathcal{D}^{\bullet} = +\infty) \ge C\left(\Pf_{U_1,U_2}^{K}(A(R)) - \Pf_{U_1,U_2}^{K}(A(R)\cap \mathcal{D}^{\bullet} < +\infty)\right).
\end{align}
Furthermore,
\begin{align} \label{error_term}
    \Pf_{U_1,U_2}^{K}(A(R)\cap \mathcal{D}^{\bullet} < +\infty) &\le \Pf_{U_1,U_2}^{K}\bigg(\{M \ge R + R_0\} \cap \mathcal{D}^{\bullet} < +\infty\bigg)\\\nonumber &\le C\exp(-cR),
\end{align}
where the last inequality follows from a union bound, translation invariance of the environment and Lemma~\ref{LemmaMexpdecay}. The constants $C,c > 0$ only depend on $\lambda,\vec{\ell},d,K$.
Using \eqref{go_to_ind},\eqref{to_control},\eqref{error_term} and Proposition~\ref{PropositionEventA(R)Indep} we deduce the proposition for all $R > r$, where $r$ is chosen large enough and depends on $\lambda,\vec{\ell},d,K$. For $R \le r$ we notice that
\begin{equation*}
\begin{split}
    \Pf_{U_1,U_2}^{K}(A(R)\cap \mathcal{D}^{\bullet} = +\infty) &\ge \Pf_{U_1,U_2}^{K}(Z^i_{1} = 1, X^{i}_{k}= X^{i}_{k - 1} + e_{1}, \text{ for } 0< k \le T^i_{\mathcal{H}^{+}_{r+R_0}}, i = 1, 2)\\
    & \ge \rho > 0.
\end{split}
\end{equation*}
The constant $\rho$ can be chosen by asking all the points on the two paths to be $K$-open and the event inside the probability to happen. Note that $\rho > 0$ depends on $r$ and $K$ and, since $r$ only depends on $\lambda,\vec{\ell},d,K$, we obtain the result. 
\end{proof}

\noindent We extend the last result to the case when $U_1$ and $U_2$ may be close. 
\begin{proposition}\label{PositiveProbDontMeet}
 There exists $K_0 > 1$ such that for any $K \ge K_0$ and any $U_1,U_2 \in \mathbb{Z}^d$ with $
 (U_1-U_2) \in \mathcal{U}$ and $R > 0$
\begin{equation*}
\Pf_{U_1,U_2}^{K}(A(R) \text{ }| \text{ } \mathcal{D}^{\bullet} = +\infty)\geq \rho,
\end{equation*}
where the constant $\rho > 0$ only depends on $\lambda,\vec{\ell},d,K$ and $\mathcal{U}$ is defined in \eqref{eqn:USet}.
\end{proposition}
\begin{proof}
We simply need to find a lower bound for $\Pf_{U_1,U_2}^{K}(A(R)\cap \mathcal{D}^{\bullet} = +\infty)$. Fix $\Lambda$ as in Proposition~\ref{PropositionJointRegTime}. We may suppose $\|U_1-U_2\| \le \Lambda$ otherwise we can simply apply Proposition~\ref{PropositionJointRegTime}. We build paths that start from $U_1$ and $U_2$ in the following way:
\begin{itemize}
    \item Fix $V_1$ and $V_2$ such that $\Lambda \le \|V_1-V_2\| < 2\Lambda$ and $ R_0+5< \ent{V_1\cdot \vec{\ell}},\ent{V_2\cdot \vec{\ell}} \le R_0+6$.
    \item Build two finite paths $(x_i)$ and $(y_j)$ starting from $U_1$, $U_2$ respectively and ending at $V_1$, $V_2$, which are:
    \begin{enumerate}
        \item Compatible with having joint regeneration points on $(V_1,V_2)$.
        \item Of length smaller than $C\Lambda$, with $C$ a large enough constant.
        \item Compatible with $A(R)$, i.e.\ on those paths the walks get at distance more than $2$ after the first 2 steps and stay on the right side (in terms of inner product with $\vec{\ell}$) of their starting point.
    \end{enumerate}
\end{itemize}
\noindent It is always possible to find such paths, see the figure below for a visual explanation of why these paths exist.
%Fix $V_1$ and $V_2$ such that $\Lambda \le ||V_1-V_2|| < 2\Lambda$ and $ R_0+5< \ent{V_1\cdot \vec{\ell}},\ent{V_2\cdot \vec{\ell}} < R_0+6$ and two finite paths $(x_i)$ and $(y_j)$ starting from $U_1$ respectively $U_2$ and ending at $V_1$ respectively $V_2$, which are compatible with having joint regeneration points on $(V_1,V_2)$ and of size lower than $C\Lambda$ where $C$ is a large enough constant. The existence of such paths follows from the next figure.
\begin{figure}[H]
    \centering
    \includegraphics[scale = 0.25]{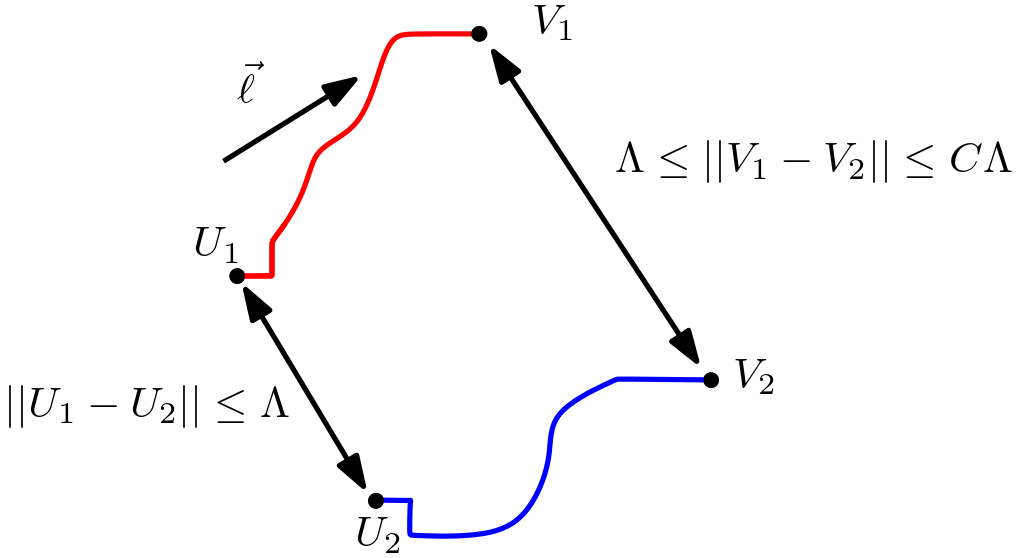}
    \caption{General construction for the points $V_1,V_2$ and for the paths to these points. One may observe that, even if $U_1 = U_2$, the construction still works. Indeed, after three steps the walks get at distance more than $3$.}
    \label{Getting_far}
\end{figure}
By the construction of the paths $(x_i)$ and $(y_j)$, the we have that 
\begin{align}\label{good_path_then_AR}
     \Pf_{U_1,U_2}^{K}\left(A(R)\right) \ge     \Pf_{U_1,U_2}^{K}\left(\substack{
    (X^1_i) = (x_i),(X^2_j)=(y_j),(x_i) \text{ and }(y_j)\text{ are all }K\text{-open points}\\ X^1_{T^1_{\mathcal{L}_1}} = V_1, X^2_{T^2_{\mathcal{L}_1}} = V_2, \,\, A(R) \circ \theta_{T^1_{\mathcal{L}_1},T^2_{\mathcal{L}_1}}}\right).
\end{align}
Using Theorem~\ref{IndependenceJointReg} the term on the r.h.s.~can be rewritten as
\begin{align*}
 &\Pf_{U_1,U_2}^{K}\left((X^1_i) = (x_i),(X^2_j)=(y_j),(x_i) \text{ and }(y_j)\text{ are all }K\text{-open points}, X^1_{T^1_{\mathcal{L}_1}} = V_1, X^2_{T^2_{\mathcal{L}_1}} = V_2\right)\\&\times \Pf_{V_1,V_2}^{K}\left(A(R)\,\,|\,\,\mathcal{D}^{\bullet} = +\infty\right).
    \end{align*}
Using Proposition~\ref{PropositionJointRegTime}, the right term is bounded from below by $\rho > 0$ which only depends on $\lambda,\vec{\ell},d,K$. The left term can be rewritten as (see step \eqref{eqn:D=inftyOtherSide} in the proof of Theorem~\ref{IndependenceJointReg} for details) 
\begin{align*}
    \Pf_{U_1,U_2}^{K}\left((X^1_i) = (x_i),(X^2_j)=(y_j),(x_i) \text{ and }(y_j)\text{ are all }K\text{-open points}\right)\Pf_{V_1,V_2}^{K}\left(\mathcal{D}^{\bullet} = +\infty\right)
\end{align*}
By Lemma~\ref{LemmaDpositive} and using the fact that 
\begin{equation*}
    \Pf_{U_1,U_2}^{K}\big((X^1_i) = (x_i),(X^2_j)=(y_j),(x_i) \text{ and }(y_j)\text{ are all }K\text{-open points}\big) \ge \rho^{*} > 0,
\end{equation*}
where $\rho^{*}$ only depends on $\lambda,K$, this concludes the proof.
\end{proof}

We conclude this subsection with the following proposition.
\begin{proposition} \label{AR_I_unlikely}
There exists $K_0 > 1$ such that for $K \ge K_0$, for $U_1,U_2 \in \mathbb{Z}^d$ such that $(U_1-U_2) \in \mathcal{U}$  and $R > 0$, we have
    \begin{align*}
        \Pf_{U_1,U_2}^{K}(A(R)\cap \mathbf{M}_{R+R_0}) \le CR^{-c},
    \end{align*}
where the constants $C,c > 0$ only depend on $\lambda,\vec{\ell},d,K$ and $\mathcal{U}$ is defined in \eqref{eqn:USet}.
\end{proposition}
\begin{proof}
We recall the notation $T$ defined in \eqref{def_trib_temps_inter}. We introduce $\Delta(R) = \{X^1_T\cdot \vec{\ell}  \ge R_0+ R\}$. On the event $\bigcap_{k=0}^{R^2}\{\mathcal{L}_{k+1}-\mathcal{L}_k \le R/4\} \cap A(R) \cap \mathbf{M}_{R+R_0}$, we have that $\Delta(R/4) \circ \theta_{T^1_{\mathcal{L}_1},T^2_{\mathcal{L}_1}}$ holds. It follows that
\begin{equation}
\begin{split}\label{eqn:OtherTermSeparation}
    \Pf_{U_1,U_2}^{K}\left( A(R)\cap \mathbf{M}_{R+R_0} \right) &\le
          \Pf_{U_1,U_2}^{K}\left(\Delta(R/4)\circ \theta_{T^1_{\mathcal{L}_1},T^2_{\mathcal{L}_1}}\right)    + \somme{k=0}{R^2}{\Pf_{U_1,U_2}^{K}\left( \mathcal{L}_{k+1}-\mathcal{L}_k >R/4\right)}.
\end{split}
\end{equation}
By Proposition~\ref{L1decaybis} and Proposition~\ref{PropositionFrequentRegeneration} we have that
\begin{align}\label{domin_sec_third_term}
    \Pf_{U_1,U_2}^{K}\left( \mathcal{L}_{k+1}-\mathcal{L}_k > R/4\right) \le CR^{-M},
\end{align}
where $M$ is large enough and $C$ only depends on $\lambda,\vec{\ell},d,K$. It follows that
\begin{align}\label{}
    \somme{k=0}{R^2}{\Pf_{U_1,U_2}^{K}\left( \mathcal{L}_{k+1}-\mathcal{L}_k >R/4\right)} \le CR^{-M+2}.
\end{align}
For the first term on the r.h.s.\ of \eqref{eqn:OtherTermSeparation} we use Theorem~\ref{IndependenceJointReg} to write
\begin{align*}
    \Pf_{U_1,U_2}^{K}\left(\Delta(R/4)\circ \theta_{T^1_{\mathcal{L}_1},T^2_{\mathcal{L}_1}}\right) &= \mathbb{E}_{U_1,U_2}^{K}\left[\Pf^{K}_{X^1_{T^1_{\mathcal{L}_1}},X^2_{T^2_{\mathcal{L}_1}}}\left(\Delta(R/4)\text{ | }\mathcal{D}^{\bullet}=+\infty \right)\right]\\&\le C\mathbb{E}_{U_1,U_2}^{K}\left[\Pf^{K}_{X^1_{T^1_{\mathcal{L}_1}},X^2_{T^2_{\mathcal{L}_1}}}\left(\Delta(R/4) \right)\right],
\end{align*}
where $C > 0$ only depends on $\lambda,\vec{\ell},d,K$ and comes from Lemma~\ref{LemmaDpositive}.

Now, we observe that $\Delta(R/4)$ is measurable with respect to $\mathcal{F}_{T}$, where $T$ and $\mathcal{F}_{n}$ are defined in \eqref{def_trib_temps_inter}. By Proposition~\ref{from_same_to_ind_env}, we have 
\begin{align*}
   \mathbb{E}_{U_1,U_2}^{K}\left[\Pf^{K}_{X^1_{T^1_{\mathcal{L}_1}},X^2_{T^2_{\mathcal{L}_1}}}\left(\Delta(R/4)\text{ | }\mathcal{D}^{\bullet}=+\infty \right)\right] = \mathbb{E}_{U_1,U_2}^{K}\left[Q^{K}_{X^1_{T^1_{\mathcal{L}_1}},X^2_{T^2_{\mathcal{L}_1}}}\left(\Delta(R/4)\text{ | }\mathcal{D}^{\bullet}=+\infty \right)\right].
\end{align*}
Observing that $\Delta(R/4)$ implies $\mathbf{M}_{R_0+R/8}$ (for $R$ large enough, otherwise the proposition is obvious), we conclude the proof using Proposition~\ref{CorollaryIndepEnv} and \eqref{domin_sec_third_term}.
\end{proof}

\subsection{Proof of asymptotic separation}
We have everything needed to prove the asymptotic separation stated in Proposition~\ref{LemmaDistantWalks}. We fix $\psi \in (0, 1)$ arbitrarily small. We also introduce $\mathbf{d} = N^{1-\psi}$. We set $\gamma_1 = \mathcal{L}_1$ and for $  k \ge 2 $ we call $\gamma_k$ the first joint regeneration level in $\mathcal{H}^{+}(\gamma_{k-1}+2\mathbf{d})$. Recall the notation $A(R)$ for $R > 0$ defined in \eqref{Definition_event_A}.

We introduce the event $S_k$ which is just event $A(\mathbf{d})$ shifted so that the walks start respectively from $X_{T^1_{\gamma_k}}$ and $X_{T^2_{\gamma_k}}$. We set:
\begin{equation} \label{Event_S_and_Sk}
     \forall k \ge 0, \text{  }S_k = A(\mathbf{d}) \circ \theta_{T^1_{\gamma_k}, T^2_{\gamma_k}} \text{ and }S(N) = \union{k=1}{\ent{N^{\psi}/4}}{S_k}.
\end{equation}
The idea of the proof is the following: each time our walks reach a level $\gamma_k$ we check whether $S_k$ occurs or not. If $S_k$ happens, using Proposition~\ref{AR_I_unlikely}, the walks will never meet after that with high probability. If $S_k$ does not happen, we wait for the walks to reach $\gamma_{k+1}$ and repeat the procedure. Using the independence property of joint regeneration levels (see Theorem~\ref{IndependenceJointReg}) we expect the first success of this process to happen quickly (See Figure~\ref{Succesives_attemps}).
\begin{figure}[H]
    \centering
    \includegraphics[scale = 0.45]{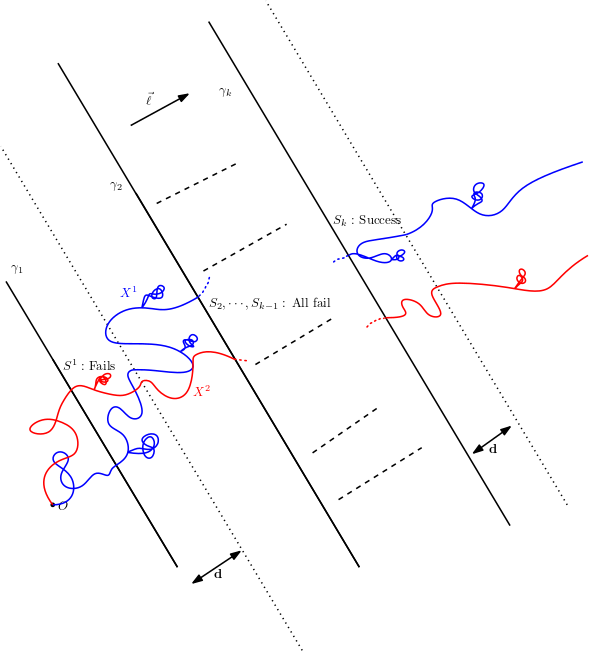}
    \caption{Successive attempts for events $S_k$.}
    \label{Succesives_attemps}
\end{figure}
\begin{proposition}\label{PropositionEventS}
There exists $K_0 > 1$ such that, for any $K \ge K_0$, there exists $\rho > 0$ such that, for any  $k \ge 1$, we have
\begin{equation*}
   \displaystyle \mathbb{P}_{0}(S_k \, | \, S_0^{c},\cdots,S_{k-1}^{c} ) \geq \rho,
\end{equation*}
where $\rho$ only depends on $\lambda,\vec{\ell},d,K$.
\end{proposition}

\begin{proof}
This is a direct application of Theorem~\ref{IndependenceJointReg} and Proposition~\ref{PositiveProbDontMeet}. Indeed, using Theorem~\ref{IndependenceJointReg}, we have
\begin{align*}
    \Pf_0(S_k,S_{k-1}^{c},\cdots,S_0^{c}) &=\somme{U_1,U_2}{}{\Pf_{0}(S_{k-1}^{c},\cdots,S_0^{c},X^{1}_{T^1_{\gamma_k}} = U_1,X^{2}_{T^2_{\gamma_k}} = U_2)\Pf_{U_1,U_2}^{K}(A(\mathbf{d}) \text{ }|\text{ }\mathcal{D}^{\bullet} =+\infty)}.
\end{align*}
Observing that $U_1-U_2 \in \mathcal{U}$ and applying Proposition~\ref{PositiveProbDontMeet} gives the result.
\end{proof}

\noindent We introduce
\begin{equation} \label{RegN}
     \textbf{Reg}(N) \coloneqq \bigcap_{k=1}^{N^2}\{\mathcal{L}_{k} - \mathcal{L}_{k-1} \leq \mathbf{d}\},
\end{equation}
which is the event under which the $N^2$ first joint regeneration levels are close to each other. We also recall the definition of $\mathbf{M}_N$ in \eqref{definition_MN}. We have the following proposition.

\begin{proposition}\label{PropositionIntersectionThreeEvents}
There exists $K_0 > 1$ such that, for any $K \ge K_0$ and any $0 \le k \le \ent{N^{\psi}/4}$, we have the following inequality: 
\begin{center}
    $\Pf_0(\textnormal{\textbf{Reg}}(N)\cap \mathrm{\bf{M}}_N \cap S_k) \leq CN^{-c(1-\psi)}$,
\end{center}
where the constants $C,c > 0$ only depend on $\lambda,\vec{\ell},d,K$.
\end{proposition}
\begin{proof}
We sum over all possible values for $(X^1_{T^1_{\gamma_k}},X^2_{T^2_{\gamma_k}})$ on the event $\mathbf{Reg}(N)$
\begin{align*}
    \Pf_0(\textbf{Reg}(N)\cap \mathrm{\bf{M}}_N \cap S_k) &= \somme{U_1,U_2}{}{\Pf_0\left(X^1_{T^1_{\gamma_k}} = U_1,X^2_{T^2_{\gamma_k}} = U_2,\textbf{Reg}(N)\cap \mathrm{\bf{M}}_N \cap S_k\right)}.
\end{align*}
On $\textbf{Reg}(N)$, $U_1\cdot\vec{\ell} < 3N/4$ and $U_2\cdot\vec{\ell} < 3N/4$. The worst case clearly happens for $k = \ent{N^{\psi}/4}$. We observe that for every $0 \le j \le \ent{N^{\psi}/4}$ on the event $\mathbf{Reg}(N)$, we have $\gamma_{j+1} \le \gamma_j + 3\mathbf{d}$. Indeed, between $\gamma_j$ and $\gamma_{j+1}$ there are at most $C \mathbf{d}$ joint regeneration levels, for some constant depending on the dimension $d$ and $\vec{\ell}$. The claim follows from the fact that, for $0 \le j \le \ent{N^{\psi}/4}$, we are choosing all the $\gamma_{j}$ between the first $N^2$ regeneration levels. We deduce that for $k = \ent{N^{\psi}/4}$ we have $\gamma_k \le 3\mathbf{d}N^{\psi}/4 \le 3N/4$.

%Indeed, on the event $\mathbf{Reg}(N)$, as long as $\gamma_j < N$ and since $\gamma_{j+1}$ is the first regeneration level in $\mathcal{H}^{+}(\gamma_j +2\mathbf{d})$ and $\gamma_j +2\mathbf{d} \le 2N$, we have $\gamma_{j+1} \le \gamma_j + 3\mathbf{d}$. It follows that for $0 \le k \le \ent{N^{\psi}/4}$ we have $\displaystyle \gamma_k \le 3\mathbf{d}N^{\psi}/4 \le 3N/4$.

We recall $S_k = A(\mathbf{d}) \circ \theta_{T^1_{\gamma_k}, T^2_{\gamma_k}}$ and since $\gamma_k \le 3N/4$ on $\mathbf{Reg}(N)$, $\mathrm{\bf{M}}_N$ implies $\mathrm{\bf{M}}_{\mathbf{d}+R_0} \circ \theta_{T^1_{\gamma_k}, T^2_{\gamma_k}}$. We deduce that the last sum is bounded by
\begin{align*}
    \somme{\substack{U_1,U_2}}{}{\Pf_0\left(X^1_{T^1_{\gamma_k}} = U_1,X^2_{T^2_{\gamma_k}} = U_2, \mathrm{\bf{M}}_{\mathbf{d}+R_0} \circ \theta_{T^1_{\gamma_k}, T^2_{\gamma_k}},  A(\mathbf{d}) \circ \theta_{T^1_{\gamma_k}, T^2_{\gamma_k}}\right)}.
\end{align*}
Using Theorem~\ref{IndependenceJointReg} this sum can be rewritten as
\begin{align} \label{MNAd}
    \somme{U_1,U_2}{}{\Pf_0\left(X^1_{T^1_{\gamma_k}} = U_1,X^2_{T^2_{\gamma_k}} = U_2\right)\Pf_{U_1,U_2}^{K}\left(\mathrm{\bf{M}}_{\mathbf{d}+R_0}\cap A(\mathbf{d})\text{ }|\text{ }\mathcal{D}^{\bullet} =+\infty\right)}.
\end{align}
By Lemma~\ref{LemmaDpositive}, the term $\Pf_{U_1,U_2}^{K}\left(\mathbf{M}_{\mathbf{d}+R_0}\cap A(\mathbf{d})\text{ }|\text{ }\mathcal{D}^{\bullet} =+\infty\right)$ is bounded by $C\Pf_{U_1,U_2}^{K}\left(\mathbf{M}_{\mathbf{d}+R_0}\cap A(\mathbf{d})\right)$, where $C > 0$ only depends on $\lambda,\vec{\ell},d,K$. We conclude using Proposition~\ref{AR_I_unlikely} that the sum \eqref{MNAd} is bounded by 
\begin{align*}
    \somme{U_1,U_2}{}{\Pf_0\left(X^1_{T^1_{\gamma_k}} = U_1,X^2_{T^2_{\gamma_k}} = U_2\right)C\mathbf{d}^{-c}} = CN^{(1-\psi)c},
\end{align*}
where $C,c > 0$ only depend on $\lambda,\vec{\ell},d,K$.
\end{proof}
Finally, we prove the asymptotic separation.
\begin{proof}[Proof of Proposition~\ref{LemmaDistantWalks}]
We recall the definition of $S(N)$ in \eqref{Event_S_and_Sk} and of $\mathbf{Reg}(N)$ in \eqref{RegN}. Using Proposition~\ref{PropositionEventS} we have $\Pf_{0}\left(S(N)^c\right) \leq (1-\rho)^{N^{\psi}} $ and by Proposition~\ref{PropositionFrequentRegeneration}, for any $M > 0$ we can take $K$ large enough such that we have $\Pf_{0}(\mathbf{Reg}(N)^{c}) \leq CN^{-M}$, where the constants $\rho,C>0$ only depend on $\lambda,\vec{\ell},d,K$. Hence, 
\begin{align*}
    \Pf_0(\mathrm{\bf{M}}_N) &= \Pf_0\left(\mathrm{\bf{M}}_N\cap \mathbf{Reg}(N) \cap S(N)\right) + \Pf_0\left(\mathrm{\bf{M}}_N \cap (\mathbf{Reg}(N)^c \cup S(N)^c)\right) \\ &\le \Pf_0\left(\mathrm{\bf{M}}_N\cap \mathbf{Reg}(N) \cap S(N)\right) + CN^{-M} + (1-\rho)^{N^{\psi}}.
\end{align*}
We only need to bound $\Pf_0\left(\mathrm{\bf{M}}_N\cap \mathbf{Reg}(N) \cap S(N)\right)$, by the definition of $S(N)$ and union bound we get
\begin{align*}
   \Pf_0\left(\mathrm{\bf{M}}_N\cap \mathbf{Reg}(N) \cap S(N)\right) &= \Pf_0\Bigg(\mathrm{\bf{M}}_N \cap \mathbf{Reg}(N) \cap \bigg( \union{k=1}{\ent{N^{\psi}/4}}{S_k}\bigg)\Bigg)\\& \le \sum_{k=1}^{\ent{N^{\psi}/4}} \Pf_0\left(\mathrm{\bf{M}}_N\cap \mathbf{Reg}(N) \cap S_k\right) \\&\le CN^{\psi}N^{-c(1-\psi)},
\end{align*}
where the last step is a consequence of Proposition~\ref{PropositionIntersectionThreeEvents}. We conclude by choosing $\psi$ small enough, we notice that this choice depends on $c$ which only depends on $\lambda,\vec{\ell},d,K$.
\end{proof}

\section{Convergence of the series}\label{Section5}

\subsection{For the trajectory}
In this section we aim to prove that $\mathbb{E}_{0}[F(Z_n^1)F(Z_n^2)] - \mathbb{E}_{0}[F(Z_n^1)]\mathbb{E}_{0}[F(Z_n^2)]$ converges polynomially fast to $0$ as $n \rightarrow +\infty$. $F$ here is a bounded positive Lipschitz function from $\mathcal{C}([0,T],\mathbb{R}^{d}) \to \mathbb{R}$ and $Z^{i}_n(t) = \displaystyle \frac{X^{i}_{\tau^{i}_{\ent{nt}}} - nvt}{n^{1/2}}$ for $t \in [0,T]$. Actually, with a slight abuse of the notation, in all this section we will refer with $Z^{i}_n(t)$ to the \textbf{polygonal interpolant} of $(X^{i}_{\tau^{i}_{\ent{nt}}} - nvt)/n^{1/2}$, which is needed so that $F(Z^{i}_n)$ is well defined. Throughout this section we assume that the function $F$ is bounded by $1$ and has Lipschitz constant $L=1$, as extending the proofs to other constants follows analogously. We also assume $t \in [0,1]$ for simplicity, as taking $t \in [0,T]$ does not create any significant difficulty.

Let us state the main theorem of this section.
\begin{theorem}\label{decay_term_serie}
There exists $K_0 > 1$ such that for all $K \ge K_0$ the following inequality holds:
\begin{align*}
    \Big|\mathbb{E}_{0}[F(Z_n^1)F(Z_n^2)] - \mathbb{E}_{0}[F(Z_n^1)]\mathbb{E}_{0}[F(Z_n^2)]\Big| \le Cn^{-c},
\end{align*}
where $C,c > 0$ are constants which only depend on $\lambda,\vec{\ell},d,K$.
\end{theorem}

\noindent We postpone the proof to the end of the section but let us give some intuition. Roughly speaking, on the event $\mathbf{M}^c_{n^\vartheta}$ (see \eqref{definition_MN}) and after few regeneration times, the two walks do not intersect anymore. It follows that they should behave like two regenerating random walks in two independent environment. This is the point of Proposition~\ref{shift_inequality} and Proposition~\ref{decondition_I} which are the main steps to prove Theorem~\ref{decay_term_serie}. The rest of the proof basically consists in proving that the beginning of the trajectories does not matter in the re-scaled trajectory $Z_n$. For the sake of brevity we will write $\mathbb{E}_0 = \mathbb{E}_{0, 0}$ when referring to two walks.

We fix $0 < \vartheta < 1/2$ throughout this section. Let us start with an easy estimate which will be useful later on.
\begin{proposition}\label{last_inequality}
There exists $K_0 > 1$ such that for all $K \ge K_0$ the following inequality holds:
\begin{align*}
    \left|\mathbb{E}_0\left[F(Z_n)\right] - \mathbb{E}_0\left[F(Z_n)\text{ }|\text{ }D=+\infty\right]\right| \le Cn^{-c},
\end{align*}
where $C,c > 0$ are constants which only depend on $\lambda,\vec{\ell},d,K$.
\end{proposition}
\begin{proof}
We introduce the scaled and centered position of the walk, starting after the first regeneration time
\begin{equation}\label{eqn:OneShiftedPosition}
Z^*_n(t) = Z_n\left(t+\tfrac{1}{n}\right) - Z_n\left(\tfrac{1}{n}\right).
\end{equation}
We recall the notation $\chi_k$ defined in \eqref{recall_chi}. We set $B_{n,\varepsilon}= \big\{\forall k \in \{0,\cdots,n+1\}\colon\chi_k \le n^{\varepsilon}\big\}$, where $\varepsilon > 0$ is a constant that we choose arbitrarily small. Then, on $B_{n,\varepsilon}$, we have 
\begin{equation*}
    \sup_{t \in [0, 1]}\|Z^*_n(t)-Z_n(t)\| \le Cn^{-\frac{1}{2}}n^{\alpha\varepsilon} \le Cn^{-c},
\end{equation*}
where $c = \alpha\varepsilon - \tfrac{1}{2} < 0$ when choosing $\varepsilon$ small enough (note that $\alpha > d + 3$ is a finite constant).
It follows that, on $B_{n,\varepsilon}$ and using the fact that $F$ is $1$-Lipschitz, we get
\begin{align*}
    |F(Z_n)-F(Z^*_n)| \le Cn^{-c}.
\end{align*}
By \cite[Lemma 6.1]{Kious_Frib}, for $K$ large enough we have $\Pf_0(B_{n,\varepsilon}^{c}) \le Cn^{-c}$, where $C,c > 0$ only depend on $\lambda,\vec{\ell},d,K$. Using this inequality and the fact that $F$ is bounded by $1$ we get that
\begin{align*}
  \left|\mathbb{E}_0[F(Z_n)]-\mathbb{E}_0[F(Z^*_n)]\right| \le Cn^{-c}.
\end{align*}
We conclude the proof by observing that, using \cite[Theorem 5.1]{Kious_Frib},
\begin{align*}
    \mathbb{E}_0[F(Z^*_n)] = \mathbb{E}_0[F(Z_n)\text{ | }D=+\infty].
\end{align*}
\end{proof}
\noindent Let us introduce some events and related estimates:
\begin{itemize}
    \item We recall 
    \begin{equation*}
        \mathrm{\bf{M}}^c_{n^\vartheta} = \bigg\{\displaystyle \inf_{\substack{x \in \mathcal{H}^{+}_{n^{\vartheta}}\cap \{X_n^1\}\\y \in \mathcal{H}^{+}_{n^{\vartheta}}\cap \{X_n^2\}}}\|x-y\|_{1}\ge 3 \bigg\}.
    \end{equation*}
    Using Proposition~\ref{LemmaDistantWalks} we have $\Pf_{0}(\mathrm{\bf{M}}_{n^\vartheta}) \le Cn^{-c}$ where $C,c > 0$ only depend on $\lambda,\vec{\ell},d,K$.
    \item We define $\mathcal{R}_0$ to be the first joint regeneration level in $ \mathcal{H}^{+}_{n^{\vartheta}}$. For $i \in \{1,2\}$ we introduce $\rho_0^{i} = T^{i}_{\mathcal{R}_0}$. In particular, $T^i_{\mathcal{R}_0}$ is a regeneration time for the walk $X^i$. We introduce the event
    \begin{equation} \label{estimate_on_R}
        B_{n,\vartheta} \coloneqq \left\{\mathcal{R}_0 \le 2n^{\vartheta}\right\}.
    \end{equation}
    Using Proposition~\ref{PropositionFrequentRegeneration}, there exists $K > 0$ such that $\Pf_{0}(B_{n,\vartheta}^{c}) \le Cn^{-c}$, where $C,c$ depend on $\vartheta, \lambda,d,\vec{\ell},K$. Indeed, on the event $\{\forall k \in \{1,\cdots,dn^{\vartheta}\}, \text{ }\mathcal{L}_{k+1}-\mathcal{L}_{k} \le n^{\vartheta}\}$, with the convention $\mathcal{L}_{0} = 0$ and using the fact that by construction $\mathcal{L}_{k+1}-\mathcal{L}_k \ge \tfrac{2}{\sqrt{d}}$, we have $\mathcal{R}_0 \le 2n^{\vartheta}$. %hence under $B_i^{c}$ and $\{\forall i \in \{1,\cdots,dn^{\theta}\}, \text{ }\mathcal{L}_{k+1}-\mathcal{L}_{k} \le n^{\theta}\}$, there exists $k \in \{0,\cdots,2dn^{\theta}\}$ such that $\tau^{i}_{k+1}-\tau^{i}_k \ge \epsilon n^{\frac{2\theta}{\gamma}-\theta}$ with $\epsilon$ chosen arbitrarily small. 
    The following inequality follows:
    \begin{align} \label{R0_occurs_soon}
        \Pf_{0}(B_{n,\vartheta}^{c}) & \le \Pf_{0}\left(\exists k \in \{1,\cdots,dn^{\vartheta}\}, \text{ }\mathcal{L}_{k+1}-\mathcal{L}_{k} \ge n^{\vartheta}\right).
    \end{align}
    %\begin{align*}
        %\Pf_{0}(B_i^{c}) & \le \Pf_{0}\bigg(\exists k \in %\{1,\cdots,dn^{\theta}\}, \text{ %}\mathcal{L}_{k+1}-\mathcal{L}_{k} \ge %n^{\theta}\bigg)\\&+\Pf_{0}\bigg(\exists k \in %\{0,\cdots,2dn^{\theta}\}, \text{ %}\tau^{i}_{k+1}-\tau^i_{k} \ge %n^{\frac{2\theta}{\gamma}-\theta}\bigg).
    %\end{align*}
Using Proposition~\ref{PropositionFrequentRegeneration} for all $K$ large enough the last term is bounded from above, for $M$ arbitrarily large, by
\begin{align*}
    Cn^{\vartheta}n^{-M\vartheta}. 
    %+ Cn^{\theta}(n^{-\frac{2\theta}{\gamma}+\theta})^{\gamma-\delta} \le Cn^{-c},
\end{align*}
%where $\delta$ is chosen small enough such that $\theta-(\frac{2\theta}{\gamma}-\theta)(\gamma-\delta) < 0$ and $M$ arbitrarily large. The constants $C,c > 0$ only depend on $\lambda,\vec{\ell},d,K$.
    \item We call $\rho^i_{n}$ the $n^{\text{th}}$ (not joint) regeneration time after $\rho_{0}^{i}$ for $X^{i}$.
    \item We recall the notation \eqref{recall_chi}. Then define $F^{i}_{n,\vartheta} = \big\{\forall k \in \{1,\cdots,n^2\}\text{, }\chi^{i}_k \le n^{\vartheta}\big\}$. Using \cite[Lemma 6.1]{Kious_Frib}, there exists $K > 0$ such that $\Pf_0((F^{i}_{n,\vartheta})^{c}) \le Cn^{-M}$, where $C$ only depends on $\lambda,\vec{\ell},d,K$.
\end{itemize}
We introduce
\begin{align} \label{estime_E}
    E_{n,\vartheta} \coloneqq \mathrm{\bf{M}}^c_{n^\vartheta} \cap B_{n,\vartheta} \cap F^{1}_{n,\vartheta} \cap F^{2}_{n,\vartheta}.
\end{align}
The estimates above can be summarized in the following proposition.
\begin{proposition} \label{proba_E_decay}
There exists $K_0 > 1$ such that, for all $K \ge K_0$, we have
\begin{equation*}
    \Pf_{0}(E_{n,\vartheta}^{c}) \le Cn^{-c},
\end{equation*}
where the constants $C,c >0$
 only depend on $\vartheta, \lambda,\vec{\ell},d,K$.
\end{proposition}
Let us introduce $\tilde{X}^i_{\cdot} = X^{i}_{\rho_{0}^{i}+\cdot} - X^{i}_{\rho^i_0}$, which is just the walk $X^i$ shifted by $\rho_0^{i}$ and re-centered; moreover, call $\tilde{Z}^{i}_{n}(t) =  (\tilde{X}^{i}_{\tau_{\ent{nt}}} - nvt)/n^{1/2}$. Let us show that on $E_{n,\vartheta}$, the trajectories $\tilde{Z}^{i}_n$ and $Z_n^{i}$ stay close.
\begin{proposition} \label{close_trajectories}
The following inequality holds almost surely on the event $E_{n, \vartheta}$
\begin{equation*}
    \forall i \in \{1,2\}, \text{  } \sup_{t \in [0, 1]}\|\tilde{Z}^{i}_n(t) - Z^{i}_n(t)\| \le Cn^{-c},
\end{equation*}
where $C,c$ depend on $\lambda,\vec{\ell}, d, \vartheta$.
\end{proposition}
\begin{proof}
Fix $i \in \{1,2\}$ and $k \in \{0,\cdots,n\}$, let us show that $X^{i}_{\tau_k}$ is close to $X^{i}_{\rho_k}$ after the proper rescaling. We have that $\mathcal{R}_0 \le 2n^{\vartheta}$ and since $\rho_{0}^{i}$ is a regeneration time, it can be written as $\tau^{i}_{k_0}$ with $k_0 \in \{0,\cdots,d\ent{n^{2\vartheta}}\}$. Indeed, at each regeneration time the walk has increased at least $2/\sqrt{d}$ in the direction $\vec{\ell}$.

Using the fact that, under $E_{n,\vartheta}$, $\chi^{i}_k \le n^{\vartheta}$ for all $k \in \{0,\cdots,n^2\}$ we have
\begin{align*}
 \|X^{i}_{\rho_{k}} - X^i_{\rho_{0}}- X^{i}_{\tau_k}\|=\|X^{i}_{\tau_{k+k_0}} - X^{i}_{\tau_{k_0}} - X^{i}_{\tau_k}\|
    & \le Ck_0n^{\alpha\vartheta}\\
    & \le Cn^{(\alpha +2)\vartheta},
\end{align*}
where $C$ only depends on $\lambda,\vec{\ell},d$. The first inequality comes from the fact that the diameter of $\mathcal{B}(m,m^{\alpha})$ is bounded by $Cm^{\alpha}$. Then, we only have to chose $\vartheta$ such that $(\alpha + 2)\vartheta \le 1/4$ to conclude, indeed uniformly over all $k \in \{ 1, \cdots, n\}$:
\begin{align*}
    \|Z^{i}_n(k/n) - \tilde{Z}^{i}_n(k/n)\| &= \bigg\|\frac{X^{i}_{\tau_k} +X^{i}_{\rho_0} - {X}_{\rho_k}^{i}}{n^{1/2}}\bigg\| \le Cn^{-1/4},
\end{align*}
we conclude by taking $c = 1/4$.
\end{proof}
We give the following corollary of the last proposition.
\begin{corollary} \label{corollary_shift}
The following inequality holds:
\begin{align*}
\left|\mathbb{E}_0\left[F(Z^1_n)F(Z^2_n)\right] - \mathbb{E}_0\left[F(\tilde{Z}^1_n)F(\tilde{Z}^2_n)\right]\right| \le Cn^{-c},
\end{align*}
where $C,c > 0$ are constants which only depend on $\lambda,\vec{\ell},d,K$.
\end{corollary}
\begin{proof}
The result follows from Proposition~\ref{close_trajectories}, Proposition~\ref{proba_E_decay} and the fact that $F$ is a bounded Lipschitz function.
\end{proof}

We define $\mathcal{I} = \big\{ \inf_{\substack{p,q \ge 0}}\|X^1_{ p}-X^2_{q}\|_{1}\ge 3 \big\}$ the event that the two walks stay far away from each other. The following proposition is crucial and shows that, when shifted and on $\mathcal{I}$, we can approximate the behavior of our walks as if they lived in independent environments. We also highlight the fact that the following proposition gives an inequality. One can prove it, in the same way, as an equality with an error term of order $Cn^{-c}$, where $C,c > 0$ only depend on $\lambda,\vec{\ell},d,K$. Let us introduce the sigma algebra
\begin{equation}\label{eqn:SigmaAlgebraFuture}
    \sigma\left(\big\{X^1_0,X^2_0,(X^1_{p},Z^1_{p}),(X^2_{q},Z^2_{q}) \,\,| \,\, p,q \ge 1 \big\}\cup \big\{c_{*}(e) \,\, |\,\, e \in \mathcal{E}_{(X^1_n)_{n\ge 0}} \cup \mathcal{E}_{(X^2_n)_{n\ge 0}} \backslash (\mathcal{E}_{X^1_0}\cup \mathcal{E}_{X^2_0})\big\}\right).
\end{equation} 

\begin{proposition} \label{shift_inequality}
For any bounded positive function $f$ measurable with respect to \eqref{eqn:SigmaAlgebraFuture} we have the following inequality:
\begin{align*}
    \mathbb{E}_0\left[f \circ \theta_{\rho^1_0, \rho^2_0} \indi{\mathcal{I}}\circ \theta_{\rho^1_0, \rho^2_0}\right] \le \mathbb{E}_{0}\bigg[\mathbb{E}^{Q^{K}_{X^1_{\rho^1_0},X^2_{\rho^2_0}}}\bigg[f \text{ }|\text{ }\mathcal{I}, D^{\otimes} = +\infty \bigg]\bigg].
\end{align*}
\end{proposition}
\begin{remark1}
Looking at this inequality one may be surprised. Indeed, on the left term $\rho_0$ is using the notion of joint regeneration levels while on the right term, $\{D^{\otimes} = \infty\} = \{D^1=+\infty,D^2=+\infty\}$ refers to the classical notion of regeneration time for a single walk. However, we observe that on event $\mathcal{I}$, the event $\{\mathcal{D}^{\bullet} = +\infty\}$ coincides with $\{D^{\otimes} = + \infty\}$.
\end{remark1}
\begin{proof}
First, using Theorem~\ref{IndependenceJointReg}, we can write
\begin{align}\label{first_eq}
    \mathbb{E}_0\left[f \circ \theta_{\rho^1_0, \rho^2_0} \indi{\mathcal{I}} \circ \theta_{\rho^1_0, \rho^2_0}\right] = \mathbb{E}_{0}\left[\mathbb{E}^{K}_{X^1_{\rho^1_0},X^2_{\rho^2_0}}\left[f\indi{\mathcal{I}} \text{ }|\text{ }\mathcal{D}^{\bullet} =+\infty \right]\right].
\end{align}
Fix $a,b \in \mathbb{Z}^d$ such that $\Pf_{0}\big(X^1_{\rho^1_0}=a,X^2_{\rho^2_0} = b\big) > 0$. Let us prove that 
\begin{align} \label{first_res}
    \mathbb{E}^{K}_{a,b}\left[f\indi{\mathcal{I}} \text{ }|\text{ }\mathcal{D}^{\bullet} =+\infty \right] \le \mathbb{E}^{Q^{K}_{a,b}}\bigg[f \text{ }|\text{ }\mathcal{I}, D^{\otimes} = +\infty \bigg].
\end{align}

Observe that on $\mathcal{I}$, the function $f\indi{\mathcal{D}^{\bullet}=+\infty}$ coincide with a function $g$ measurable with respect to $\mathcal{F}_T$, where $T$ and $\mathcal{F}_n$ are defined in \eqref{def_trib_temps_inter}. Thus, we can write 
\begin{align}\label{go_to_Q}
    \mathbb{E}^{K}_{a,b}\left[f\indi{\mathcal{I}}\indi{\mathcal{D}^{\bullet} = +\infty}\right] \underset{}{=}\mathbb{E}^{K}_{a,b}\left[\indi{\mathcal{I}}g\right]\underset{\ref{from_same_to_ind_env}}{=} \mathbb{E}^{Q^{K}_{a,b}}\left[\indi{\mathcal{I}}g\right]=   \mathbb{E}^{Q^{K}_{a,b}}\left[f\indi{\mathcal{I}}\indi{\mathcal{D}^{\bullet} = +\infty}\right],
\end{align}
where the second equality uses Proposition~\ref{from_same_to_ind_env} and the fact that $\indi{\mathcal{I}}g$ is measurable with respect to $\mathcal{F}_T$. On $\mathcal{I}$, the event $\{\mathcal{D}^{\bullet}=+\infty\}$ coincides with $\{D^{\otimes} = +\infty\}$. It follows that
\begin{align} \label{third_res}
     Q_{a,b}^{K}\left(\mathcal{I},D^{\otimes}=+\infty\right)= \Pf_{a,b}^{K}\left(\mathcal{I},D^{\otimes}=+\infty\right)\le \Pf_{a,b}^{K}\left(\mathcal{D}^{\bullet} =+\infty\right).
\end{align}
Result \eqref{first_res} follows from \eqref{go_to_Q} and \eqref{third_res}. Inserting it in \eqref{first_eq} concludes the proof.
\end{proof}
\noindent The next lemma will enable us to erase the conditioning on the event $\mathcal{I}$.
\begin{lemma} \label{useful_inequality}
We have the following inequality:
\begin{align*}
    \mathbb{E}_{0}\left[Q^{K}_{X^1_{\rho^1_0},X^2_{\rho^2_0}}\left(\mathcal{I}^{c} \text{ }|\text{ }D^{\otimes}=+\infty \right)\right] \le Cn^{-c},
\end{align*}
where $C,c > 0$ are constants which only depend on $\lambda,\vec{\ell},d,K$.
\end{lemma}
\begin{proof}
By \cite[Lemma 5.1]{Kious_Frib} we have $Q^K_{X^1_{\rho^1_0},X^2_{\rho^2_0}}\left(D^{\otimes}=+\infty\right) = \Pf_0(D=+\infty)^2 > \eta > 0$, where $\eta$ only depends on $\lambda,\vec{\ell},d,K$. Hence, we only have to dominate
\begin{align*}
     \mathbb{E}_{0}\left[Q^{K}_{X^1_{\rho^1_0},X^2_{\rho^2_0}}\left(\mathcal{I}^{c},D^{\otimes}=+\infty \right)\right].
\end{align*}
By the same argument as in the last proof, it is equal to
\begin{align*}
     \mathbb{E}_{0}\left[\Pf^{K}_{X^1_{\rho^1_0},X^2_{\rho^2_0}}\left(\mathcal{I}^{c},\mathcal{D}^{\bullet} =+\infty \right)\right].
\end{align*}
Using Lemma~\ref{LemmaDpositive}, the last term is bounded from above by
\begin{align*}
    C\mathbb{E}_{0}\left[\Pf^{K}_{X^1_{\rho^1_0},X^2_{\rho^2_0}}\left(\mathcal{I}^{c}\text{ | }\mathcal{D}^{\bullet} =+\infty \right)\right],
\end{align*}
where $C > 0$ is a constant which only depends on $\lambda,\vec{\ell},d,K$. By Proposition~\ref{IndependenceJointReg}, this expectation is equal to 
\begin{align*}
    \mathbb{P}_{0}\left(\mathcal{I}^{c} \circ \theta_{\rho_0}\right) \le \Pf_0\bigg(\bigg\{\displaystyle \inf_{\substack{x \in \mathcal{H}^{+}_{n^{\vartheta}}\cap \{X_n^1\}\\y \in \mathcal{H}^{+}_{n^{\vartheta}}\cap \{X_n^2\}}}\|x-y\|_{1}\le 2 \bigg\}\bigg) \le Cn^{-c},
\end{align*}
where the last inequality follows from Proposition~\ref{LemmaDistantWalks}.
\end{proof}
\begin{proposition} \label{decondition_I}
For any bounded positive function $f$ measurable with respect to \eqref{eqn:SigmaAlgebraFuture} we have the following inequality:
\begin{align*}
    &\left|\mathbb{E}_{0}\bigg[\mathbb{E}^{Q^{K}_{X^1_{\rho^1_0},X^2_{\rho^2_0}}}\bigg[f \text{ }|\text{ } \mathcal{I},D^{\otimes}=+\infty \bigg]\bigg] - \mathbb{E}_{0}\bigg[\mathbb{E}^{Q^{K}_{X^1_{\rho^1_0},X^2_{\rho^2_0}}}\bigg[f \text{ }|\text{ }D^{\otimes}=+\infty \bigg]\bigg]\right| \le Cn^{-c},
\end{align*}
where $C,c > 0$ are constants which only depend on $\lambda,\vec{\ell},d,K$.
\end{proposition}
\begin{proof}
We can rewrite the term $\mathbb{E}_{0}\bigg[\mathbb{E}^{Q^{K}_{X^1_{\rho^1_0},X^2_{\rho^2_0}}}\bigg[f \text{ }|\text{ }D^{\otimes}=+\infty \bigg]\bigg]$ as
\begin{align*}
    & \mathbb{E}_{0}\bigg[\mathbb{E}^{Q^{K}_{X^1_{\rho^1_0},X^2_{\rho^2_0}}}\bigg[f \text{ }|\text{ }\mathcal{I},D^{\otimes}=+\infty \bigg]Q^{K}_{X^1_{\rho^1_0},X^2_{\rho^2_0}}\bigg(\mathcal{I}\text{ | }D^{\otimes}=+\infty\bigg)\bigg] +\\&\mathbb{E}_{0}\bigg[\mathbb{E}^{Q^{K}_{X^1_{\rho^1_0},X^2_{\rho^2_0}}}\bigg[f \text{ }|\text{ }\mathcal{I}^{c},D^{\otimes}=+\infty \bigg]Q^{K}_{X^1_{\rho^1_0},X^2_{\rho^2_0}}\bigg(\mathcal{I}^{c}\text{ | }D^{\otimes}=+\infty\bigg)\bigg].
\end{align*} 
Using the fact that $f$ is bounded, it follows that the term in the statement of the proposition is bounded from above by
\begin{align*}
    C\mathbb{E}_{0}\bigg[Q^{K}_{X^1_{\rho^1_0},X^2_{\rho^2_0}}\bigg(\mathcal{I}^{c}\text{ | }D^{\otimes}=+\infty\bigg)\bigg].
\end{align*}
We conclude using Lemma~\ref{useful_inequality}.
\end{proof}

We are now able to prove Theorem~\ref{decay_term_serie}.
\begin{proof}[Proof of Theorem~\ref{decay_term_serie}]
We recall that we want to dominate
\begin{equation*}
    \mathbb{E}_{0}\left[F(Z^1_n)F(Z^2_n)\right] - \mathbb{E}_0\left[F(Z^1_n)\right]\mathbb{E}_0\left[F(Z^2_n)\right].
\end{equation*}
First, using Corollary~\ref{corollary_shift} we dominate this by
\begin{align*}
    Cn^{-c}+\mathbb{E}_{0}\left[F(\tilde{Z}^1_n)F(\tilde{Z}^2_n)\right] - \mathbb{E}_0\left[F(Z^1_n)\right]\mathbb{E}_0\left[F(Z^2_n)\right].
\end{align*}
Using the fact that $\Pf_{0}(\mathcal{I}^{c} \circ \theta_{\rho^1_0, \rho^2_0}) \le Cn^{-c}$ by Proposition~\ref{LemmaDistantWalks} and the fact that $F$ is bounded we dominate it by
\begin{align*}
    Cn^{-c}+\mathbb{E}_0\left[\indi{\mathcal{I}} \circ \theta_{\rho^1_0, \rho^2_0} F(\tilde{Z}^1_n)F(\tilde{Z}^2_n)\right] - \mathbb{E}_0\left[F(Z^1_n)\right]\mathbb{E}_0\left[F(Z^2_n)\right].
\end{align*}
Let us bound from above the term
\begin{equation*}
    \mathbb{E}_0\left[\indi{\mathcal{I}} \circ \theta_{\rho^1_0, \rho^2_0} F(\tilde{Z}^1_n)F(\tilde{Z}^2_n)\right].
\end{equation*}
Let us introduce the function $f = F\left(Z^1_n-Z^1_n(0)\right)F\left(Z^2_n-Z^2_n(0)\right)$, then the last term can be rewritten as
\begin{align*}
    \mathbb{E}_0\left[\indi{\mathcal{I}} \circ \theta_{\rho^1_0, \rho^2_0} f\circ \theta_{\rho^1_0, \rho^2_0} \right].
\end{align*}
Furthermore, we observe that $f$ satisfies the assumptions of Proposition~\ref{shift_inequality} and Proposition~\ref{decondition_I}, hence last term is bounded from above by
\begin{align*}
    Cn^{-c} + \mathbb{E}_0\bigg[\mathbb{E}^{Q_{X^1_{\rho^1_0},X^2_{\rho^2_0}}^{K}}\left[f\text{ | }D^{\otimes}=+\infty\right]\bigg].
\end{align*}
Using \cite[Theorem 5.4]{Kious_Frib} and the form of the function $f$ we can write 
\begin{equation*}
    \mathbb{E}^{Q_{X^1_{\rho^1_0},X^2_{\rho^2_0}}^{K}}\left[F\left(Z^1_{n}-Z^1_n(0)\right)F\left(Z^2_{n}-Z^2_n(0)\right)\text{ | }D^{\otimes}=+\infty\right],
\end{equation*}
as 
\begin{equation*}
   \mathbb{E}^K_0\left[F(Z_n) \text{ | }D=+\infty\right]^2.
\end{equation*}
To summarise we have dominated the term
\begin{align*}
     \mathbb{E}_0\left[F(Z^1_n)F(Z^2_n)\right] - \mathbb{E}_0\left[F(Z^1_n)\right]\mathbb{E}_0\left[F(Z^2_n)\right],
\end{align*}
by
\begin{align*}
     Cn^{-c} + \mathbb{E}^K_0\left[F(Z_n) \text{ | }D=+\infty\right]^2 - \mathbb{E}_0\left[F(Z_n)\right]^2.
\end{align*}
We conclude by using the fact that $F$ is bounded and Proposition~\ref{last_inequality}.
\end{proof}

\subsection{For the joint process}
We recall the notations $S_n(t) =\tau_{\floor{tn}}/\mathrm{Inv}(n)$, where $\mathrm{Inv}(n) = \inf\{t \ge 0 \colon \mathbf{P}[c_{*} > t] \leq 1/n\}$, $\displaystyle W_n(t) = (Z_n(t),S_n(t))$ and $\displaystyle W^{*}_n(t) = W_n\left(t+\tfrac{1}{n}\right) - W_n\left(\tfrac{1}{n}\right)$. We fix a positive and bounded Lipschitz function $F : (\mathbb{R}^{d+1})^{m} \rightarrow \mathbb{R}$ and $0 \le t_1 \le \cdots \le t_m \le T$. For simplicity, we will assume $T = 1$ throughout the section. To clarify the notations we will write 
\begin{align*}
    F(W_n^{*}) = F(W_n^{*}(t_1),\cdots,W_n^{*}(t_m)).
\end{align*}
In this section we prove the following result.
\begin{proposition}\label{decay_time_serie}
There exists $K_0 > 1$ such that the following inequality holds for all $K \ge K_0$:
\begin{align*}
    |\mathbb{E}_{0}[F(W^{*1}_n)F(W^{*2}_n)] - \mathbb{E}_{0}[F(W^{*1}_n)]\mathbb{E}_{0}[F(W^{*2}_n)]| \le Cn^{-c},
\end{align*}
where $C,c > 0$ are constants which only depend on $\lambda,\vec{\ell},d,K$. 
\end{proposition}
\noindent Before proving Proposition~\ref{decay_time_serie} we will need to state and prove Lemma~\ref{lemma_time_shift}. Let us also admit the following lemma which is crucial in the proof of Lemma~\ref{lemma_time_shift}.
\begin{lemma}\label{LemmaSmallTimes}
For all $\eta \in (0,1)$ and $\rho < \eta / \gamma$, the probability
\[
\mathbb{P}_{0} \left( S^*_n(n^{-\eta}) > n^{-\rho} \right) = \mathbb{P}_{0} \left( S_n(n^{-\eta}) > n^{-\rho} | D = + \infty \right)
\]
decays polynomially fast to $0$ as $n \to \infty$.
\end{lemma}
We let $\rho^i_0$ be $T^i_{\mathcal{R}_0}$, where $\mathcal{R}_0$ is the first joint regeneration level in $\mathcal{H}^+_{n^\vartheta}$. Let us also recall that we denote $k_0^{i} \ge 1$ the random integer such that $\tau^{i}_{k_0^{i}} = \rho^{i}_0$.
Then let us introduce, for $i \in \{1,2\}$
\begin{equation} \label{Stilde}
    \widetilde{W}^{*i}_n(t) = W^{*i}_n\left(\tfrac{k^i_0}{n}+t\right) - W^{*i}_n\left(\tfrac{k^i_0}{n}\right),
\end{equation}
which is the value of $W^{*i}_n$ for the walk shifted by $\tau_{k^i_0}$.
\begin{lemma} \label{lemma_time_shift}
The following inequality holds:
\begin{align*}
\left|\mathbb{E}_0\left[F(W^{*1}_n)F(W^{*2}_n)\right] - \mathbb{E}_0\left[F(\widetilde{W}^{*1}_n)F(\widetilde{W}^{*2}_n)\right]\right| \le Cn^{-c},
\end{align*}
where $C,c > 0$ are constants which only depend on $\lambda,\vec{\ell},d,K$.
\end{lemma}
\begin{proof}
Since $F$ is Lipschitz we have the following bound for $i \in \{1,2\}$
\begin{align} 
    |F(W^{*i}_n)-F(\widetilde{W}^{*i}_n)|\le C\bigg(&\max_{1\le j \le m}\left(\left|S_n^{*}\left(t_j +\tfrac{k_0^{i}}{n}\right) - S_n^{*}\left(\tfrac{k_0^i}{n}\right)- S_n^{*}(t_j)\right|\right) + \label{first_estim_lips} \\& \max_{1\le j\le m}\left(\left\| Z_n^{*}\left(t_j +\tfrac{k_0^{i}}{n}\right) - Z_n^{*}\left(\tfrac{k_0^i}{n}\right)- Z_n^{*}(t_j)\right\|\right)\bigg).\label{sec_estim_lips}
\end{align}
Let us recall the event $B_{n,\vartheta}$ defined in \eqref{estimate_on_R}, we proved that $\Pf_0(B_{n,\vartheta}^c)\le Cn^{-c}$ in \eqref{R0_occurs_soon}.

Let us put ourselves on the event $B_{n,\vartheta}$ and fix $1 > 1 - \eta > \vartheta > 0$. Then for $i \in \{1,2\}$, $X^i_{\tau^i_{\floor{n^{1-\eta}}}}$ is $\Pf_{0}$-almost surely in $\mathcal{H}^+_{2 n^\vartheta}$, hence $\tau^i_{k^i_0} < \tau^i_{n^{1-\eta}}$ since $\mathcal{R}_0 \le 2n^{\vartheta}$. It follows that event $B_{n, \vartheta}$ implies $C_{n,\eta}^{i} \coloneqq \{k^i_0 \le n^{1 - \eta}\}$ for $i \in \{1,2\}$. In addition, under event $C^{i}_{n,\eta}$ the term $S^{*i}_n\left(t_{j} + k^i_0/n\right) - S^{*i}_n\left(t_{j} \right)$ is dominated by $S^{*i}_n(t_{j} + n^{-\eta}) - S^{*i}_n(t_{j})$ which is distributed as $S^{*i}_n(n^{-\eta})$ under $\Pf_0$ by \cite[Theorem 5.4]{Kious_Frib}.

Fixing $t_0 = 0$, we define
\begin{equation}\label{definition_Bin}
    A^i_{n,\eta,\rho} \coloneqq \left\{\forall j \in \{0,\cdots,m\}\text{, } S^{*i}_n\left(t_{j} + k^i_0/n\right) - S^{*i}_n\left(t_{j}\right) \le n^{\rho}\right\}.
\end{equation}
Then, using \eqref{estimate_on_R}, Lemma~\ref{LemmaSmallTimes}, a union bound and the i.i.d.\ structure of regeneration times one gets
\begin{align} \label{estimate_Ain}
    &\Pf_0\left( (A^i_{n,\eta,\rho})^{c} \right) \le m\Pf_0\left(S^{*i}_n(n^{-\eta}) \ge n^{\rho}   \text{ | }D=+\infty\right) + \Pf_0\left(B_{n,\vartheta}^{c}\right) \le Cn^{-c},
\end{align}
where $\rho < \eta / \gamma < 1/\gamma$. We recall the definition of the event $E_{n,\vartheta}$ in \eqref{estime_E}. Using Proposition~\ref{proba_E_decay} and Proposition~\ref{close_trajectories}, on $E_{n,\vartheta}$, we have  
\begin{align*}
    \sup_{t \in [0, 1]}\|\tilde{Z}^{i}_n(t) - Z^{i}_n(t)\| \le Cn^{-c}.
\end{align*}
In particular,
\begin{align} \label{sec_estim_traj}
    \forall j \in\{0,\cdots,m\}, \text{ } \left\|Z_n^{*}\left(t_j +\tfrac{k_0^{i}}{n}\right) - Z_n^{*}\left(\tfrac{k_0^i}{n}\right)- Z_n^{*}(t_j)\right\| \le Cn^{-c}.
\end{align}
Putting everything together we obtain
\begin{flalign*}
    \Big|&\mathbb{E}_0\left[F(W^{*1}_n)F(W^{*2}_n)\right] - \mathbb{E}_0\left[F(\widetilde{W}^{*1}_n)F(\widetilde{W}^{*2}_n)\right]\Big|
    \\ &  \le\left|\mathbb{E}_0\left[F(W^{*1}_n)F(W^{*2}_n) \mathds{1}_{\{A^1_{n,\eta,\rho} \cap A^2_{n,\eta,\rho} \cap E_{n,\vartheta}\}}\right] - \mathbb{E}_0\left[F(\widetilde{W}^{*1}_n)F(\widetilde{W}^{*2}_n)\mathds{1}_{\{A^1_{n,\eta,\rho} \cap A^2_{n,\eta,\rho} \cap E_{n,\vartheta}\}}\right] \right| + Cn^{-c} \\
    &  \le Cn^{-c}.
\end{flalign*}
In the first inequality we used \eqref{estimate_Ain} and Proposition~\ref{proba_E_decay} and in the second one the fact that $F$ is bounded, \eqref{first_estim_lips}, \eqref{sec_estim_lips}, \eqref{definition_Bin} and \eqref{sec_estim_traj}.
\end{proof}
We are now able to prove the key result.
\begin{proof}[Proof of Proposition~\ref{decay_time_serie}]
Let us fix again $1 > 1 - \eta > \vartheta > 0$ and recall that $\mathcal{I} = \{ \inf_{\substack{p,q \ge 0}}(\|X^1_{ p}-X^2_{q}\|_{1})\ge 2 \}$. By Lemma~\ref{lemma_time_shift} we get that
\begin{equation*}
    \mathbb{E}_{0}[F(W^{*1}_n)F(W^{*2}_n)] - \mathbb{E}_{0}[F(W^{*1}_n)]\mathbb{E}_{0}[F(W^{*2}_n)] \le Cn^{-c} + \mathbb{E}_{0}[F(\widetilde{W}^{*1}_n)F(\widetilde{W}^{*2}_n)] - \mathbb{E}_{0}[F(W^{*1}_n)]\mathbb{E}[F(W^{*2}_n)]
\end{equation*}
Moreover, recalling the definition of $\mathbf{M}_N$ in \eqref{definition_MN} and using Proposition~\ref{LemmaDistantWalks}, we also get that
\begin{align}
    \left| \mathbb{E}_0\left[F(\widetilde{W}^{*1}_n)F(\widetilde{W}^{*2}_n)\right] - \mathbb{E}_0\left[\mathds{1}_{\big\{\mathcal{I} \circ \theta_{\rho^1_0, \rho^2_0}\big\}}F(\widetilde{W}^{*1}_n)F(\widetilde{W}^{*2}_n)\right] \right| &\le \Pf_{0}\left(\mathcal{I}^{c} \circ \theta_{\rho^1_0, \rho^2_0} \right)\nonumber \\&\le \Pf_0\left(\mathbf{M}_{n^{\vartheta}}\right)\nonumber\\ &\le C n^{-c}\label{EqnProofTimeSerie3}.
\end{align}
We introduce $f = F(W^{*1}_n)F(W^{*2}_n)$ and note that
\begin{equation*}
    \mathbb{E}_0\left[\mathds{1}_{\big\{\mathcal{I} \circ \theta_{\rho^1_0, \rho^2_0}\big\}}F(\widetilde{W}^{*1}_n)F(\widetilde{W}^{*2}_n)\right] = \mathbb{E}_0\left[\mathds{1}_{\big\{\mathcal{I} \circ \theta_{\rho^1_0, \rho^2_0}\big\}} f \circ \theta_{\rho^1_0, \rho^2_0}\right].
\end{equation*}
Furthermore, $f$ satisfies the assumptions of Proposition~\ref{shift_inequality} and Proposition~\ref{decondition_I}, hence we get
\begin{align*}
    \mathbb{E}_0\left[\mathds{1}_{\big\{\mathcal{I} \circ \theta_{\rho^1_0, \rho^2_0}\big\}} f \circ \theta_{\rho^1_0, \rho^2_0} \right] \le Cn^{-c} + \mathbb{E}_0\bigg[\mathbb{E}^{Q_{X^1_{\rho^1_0},X^2_{\rho^2_0}}^{K}}\Big[f\text{ | }D^{\otimes}=+\infty\Big]\bigg].
\end{align*}
We can rewrite the last term in the following way
\begin{align}
    \mathbb{E}^{Q_{X^1_{\rho^1_0},X^2_{\rho^2_0}}^{K}}\Big[f\text{ | }D^{\otimes}=+\infty\Big] &= \mathbb{E}^{Q_{X^1_{\rho^1_0},X^2_{\rho^2_0}}^{K}}\Big[F(W^{*1}_n)F(W^{*2}_n)\text{ | }D^{\otimes}=+\infty\Big]\nonumber\\ 
    &= \mathbb{E}_{0}\left[F(W^{*}_n)\text{ | }D=+\infty\right]^2 = \mathbb{E}_{0}\left[F(W^{*}_n)\right]^2.\label{EqnProofTimeSerie4}
\end{align}
Putting together the estimates \eqref{EqnProofTimeSerie3} and \eqref{EqnProofTimeSerie4} we have shown that
\begin{equation*}
    \big| \mathbb{E}_{0} \left[ F(W^{*1}_n)F(W^{*2}_n) \right] - \mathbb{E}_{0} \left[ F(W^{*1}_n)\right] \mathbb{E}_{0} \left[ F(W^{*2}_n)\right] \big| \le Cn^{-c},
\end{equation*}
which concludes the proof.
\end{proof}

\subsection{Small times do not matter}

We end the section by showing the proof of Lemma~\ref{LemmaSmallTimes}.

\begin{lemma}(Fuk-Nagaev inequality, \cite[Theorem 5.1]{BergerFukNagaev}) \label{LemmaFukNagaev}
Let $(X_i)_{i \ge 1}$ be a family of independent identically distributed random variables with regularly varying tails and tail parameter $\displaystyle \gamma \in (0, 1)$ under some probability measure $P$. Let $\displaystyle \mathscr{S}(n) \coloneqq \sum_{i = 1}^{\floor{n}} X_i$ and $\displaystyle \mathscr{M}(n) \coloneqq \max_{i \in \{1,\cdots,\floor{n}\}} X_i$, then there exists a constant $c>0$ such that, for all $y \le x$
\begin{equation}
    P \left( \mathscr{S}(n) > x, \, \mathscr{M}(n) < y \right) \stackrel{\textnormal{F-N}}{\le} \left( c n \frac{y}{x} L(y) y^{-\gamma}\right)^{x/y},
\end{equation}
where $L(y)$ is the slowly varying function associated to the tail of $X_1$.
\end{lemma}

%\begin{lemma}\cite[Lemma~10.1]{Kious_Frib}\label{LemmaFriKiu}
%For all $\delta>0$ there exists $K_0<\infty$ such that for all $K>K_0$,
%\[
%\mathbb{E}_{0}^K\left[ \tau_1^{\alpha - \delta} | D = \infty \right] < \infty.
%\]
%\end{lemma}

\noindent Let us consider $\{\tau_1^{(i)}\}_{i \ge 1}$, a sequence of i.i.d.\ random variables distributed as $\tau_1$ under the measure $\Pf^K_{0}(\cdot \text{ | }D=+\infty)$, let us assume for simplicity that they can be defined under $\mathbb{P}_{0}$. In \cite[Proposition 10.3]{Kious_Frib} it is proved that the sum $\mathscr{S}(n) \coloneqq  \sum_{i = 1}^{n} \tau_1^{(i)}$ is in the domain of attraction of a completely asymmetric stable law of index $\gamma$. This implies, recalling \cite[IX.8, Theorem 1]{Feller}, that $\mathbb{P}_0^K(\tau_1 > t \text{ | }D=+\infty) = \psi(t)t^{-\gamma}$ for some slowly varying function $\psi(\cdot)$.

\begin{lemma}\label{LemmaMaxBound}
For all $\eta \in (0,1)$ set $\displaystyle \mathscr{M}(n^{1-\eta})$ as in Lemma~\ref{LemmaFukNagaev} then for all $\rho < \eta/\gamma$ and all $\varepsilon< (\eta - \rho \gamma)/3\gamma$ there exists $\nu>0$ such that 
\[
\mathbb{P}_{0}\left( \mathscr{M}(n^{1-\eta}) > n^{1/\gamma - \rho - 2\varepsilon} \right) \le Cn^{-\nu}.
\]
\end{lemma}

\begin{proof}
\begin{align*}
    \mathbb{P}_{0}\left( \mathscr{M}(n^{1-\eta}) > n^{1/\gamma - \rho - 2\varepsilon} \right) &\le n^{1 - \eta} \mathbb{P}_{0}^{K}\left( \tau_1 > n^{1/\gamma - \rho - 2\varepsilon} | D=\infty \right) \\
    &\le Cn^{1-\eta} n^{-1+\gamma\rho + 2\varepsilon \gamma} \psi(n^{1/\gamma - \rho - 2\varepsilon}) \\
    &\le Cn^{-\eta + \gamma \rho + 3\varepsilon \gamma}\le n^{-\nu},
\end{align*}
where we used the fact that a slowly varying function can always be bounded from above by any polynomial so that $\psi(n^{1/\gamma}) \le C n^{\varepsilon/\gamma}$ (Potter's bound \cite[Theorem~1.5.6]{RegularVariations}).
Finally, it holds that $\eta - \gamma \rho - 3\varepsilon \gamma > 0$ by our choice of $\varepsilon$, one has that the last inequality holds for some $\nu$ small enough.
\end{proof}

\begin{lemma}\label{LemmaSumBound}
For all $\eta \in (0, 1)$ recall $\displaystyle \mathscr{M}(n^{1-\eta})$ and $ \displaystyle \mathscr{S}(n^{1-\eta})$ from Lemma~\ref{LemmaFukNagaev} then for all $\rho < \eta/\gamma$ and all $\varepsilon< (\eta - \rho \gamma)/3\gamma$ there exists $\nu>0$ such that 
\[
\mathbb{P}_{0}\left( \mathscr{S}(n^{1-\eta}) > n^{1/\gamma - \varepsilon - \rho} \right) \le n^{-\nu}.
\]
\end{lemma}

\begin{proof}
First notice that, by total probability
\begin{align*}
    \mathbb{P}_{0}\left( \mathscr{S}(n^{1-\eta}) > n^{1/\gamma - \varepsilon - \rho} \right) &= \mathbb{P}_{0}\left( \mathscr{S}(n^{1-\eta}) > n^{1/\gamma - \varepsilon - \rho}, \mathscr{M}(n^{1-\eta}) > n^{1/\gamma - \rho - 2\varepsilon} \right) \\ & +\mathbb{P}_{0}\left( \mathscr{S}(n^{1-\eta}) > n^{1/\gamma - \varepsilon - \rho}, \mathscr{M}(n^{1-\eta}) \le n^{1/\gamma - \rho - 2\varepsilon} \right).
\end{align*}
By Lemma~\ref{LemmaMaxBound} 
\[
\mathbb{P}_{0}\left( \mathscr{S}(n^{1-\eta}) > n^{1/\gamma -\varepsilon- \rho}, \mathscr{M}(n^{1-\eta}) > n^{1/\gamma - \rho - 2\varepsilon} \right) \le Cn^{-\nu},
\]
for some $\nu>0$. So we just need to bound the second term of the sum. We can appeal to Lemma~\ref{LemmaFukNagaev} with the pair
\[\quad\quad x = n^{1/\gamma - \varepsilon - \rho}, \quad\quad y=n^{1/\gamma - \rho - 2\varepsilon}.
\]
Plugging these quantities into the formula gives
\begin{align*}
    \mathbb{P}_{0}\left( \mathscr{S}(n^{1-\eta}) > n^{1/\gamma - \varepsilon - \rho}, \mathscr{M} \le n^{1/\gamma - \rho - 2\varepsilon} \right) &\le \left( C n^{1-\eta} n^{-\varepsilon} n^{-1 + \gamma \rho + 2\gamma \varepsilon} \psi(n^{1/\gamma - \rho - 2\varepsilon}) \right)^{n^\varepsilon}.
\end{align*}
Since the exponent is large we need to check that what is inside the brackets is small. This is immediate as
\[
C n^{1-\eta} n^{-\varepsilon} n^{-1 + \gamma \rho + 2\gamma \varepsilon} \psi(n^{1/\gamma - \rho - 2\varepsilon})\le  C n^{2\varepsilon \gamma - \eta + \gamma \rho} \le e^{-1},
\]
as $\psi(n^{1/\gamma}) \le Cn^{\varepsilon}$ using Potter's bound and $2\varepsilon \gamma - \eta + \gamma \rho < 0$ by our choice of $\varepsilon$, this gives a super-polynomial bound for the desired quantity.
\end{proof}

\begin{proof}[Proof of Lemma~\ref{LemmaSmallTimes}]
Let us start with the first statement. Note that, under $\Pf_0$, $\{\tau_{i + 1} - \tau_i\}_{i \ge 1}$ are independent identically distributed random variables distributed as $\tau_1$ under $\Pf^K_{0}(\cdot \, |\, D=+\infty)$. Using this fact, we can re-write
\begin{equation*}
    S^*_n(n^{-\eta}) \stackrel{(\mathrm{d})}{=} \mathrm{Inv}(n)^{-1}\sum_{i = 1}^{n^{1 - \eta}} \tau_{1}^{(i)},
\end{equation*}
where $\{\tau_{1}^{(i)}\}_{i \ge 1}$ are i.i.d.\ copies of $\tau_1$ under the measure $\Pf^{K}_{0}( \cdot \, |\, D = \infty)$. Then Lemma~\ref{LemmaSumBound} is enough to conclude together with the fact that for all $\varepsilon>0$
\begin{equation*}
    \frac{n^{1/\gamma - \varepsilon}}{\mathrm{Inv}(n)} \stackrel{n \to \infty}{\longrightarrow} 0.
\end{equation*}
\end{proof}

\section{Quenched convergence}

We start this section by proving Theorem~\ref{TheoremVarianceClockProcess}.
\begin{proof}[Proof of Theorem~\ref{TheoremVarianceClockProcess}]
The result is a straightforward consequence of Theorem~\ref{decay_term_serie} and Proposition~\ref{decay_time_serie}. In particular one just needs to notice that
\begin{align*}
\mathbf{Var}\left(E_{0}^{\omega}\left[F_2(Z_{b^n})\right]\right) = \mathbb{E}_{0}[F_2(Z_{b^n}^1)F_2(Z_{b^n}^2)] - \mathbb{E}_0[F_2(Z_{b^n}^1)]\mathbb{E}_0[F_2(Z_{b^n}^2)].
\end{align*}
The same decomposition holds for the process $W^*$ and the function $F_1$.
\end{proof}

Recall $D^d([0, T])$ is the space $\mathbb{R}^d$-valued c\`{a}dl\`{a}g functions and set $D([0, T]) = D^1([0, T])$. In the following we consider $U$ the uniform topology and $J_1$, $M_1$ two classical Skorohod topologies.

Let us recall the definition of the vector $v$ and the matrix $\Sigma$ that were introduced in \cite[Equations (11.1),(11.2)]{Kious_Frib}. In particular, $v \coloneqq \mathbb{E}_{0}^K[X_{\tau_1}| D = \infty]$ while $\Sigma$ is the covariance matrix of $X_{\tau_1}$ under $\mathbb{P}_{0}^{K}( \cdot | D = \infty)$. Moreover, let us introduce $v_0 = v / \|v\|$, $I_d$ as the $d$-dimensional identity matrix and $P_{v_0}$ as the projection matrix on $v_0$. Finally let
\[
M_d \coloneqq C_\infty^{-\gamma/2}(I_d - P_{v_0}) \sqrt{\Sigma}.
\]

\begin{proposition}\label{PropositionJointQuenchedConvergence}
Fix any $T > 0$, the followings holds for almost every environment $\omega \in \Omega$. The law of $(Y_n(t))_{0 \le t \le T}$ converges under $P^\omega_0(\cdot)$ in the $U$ topology on $D^d$ towards the distribution of $(v t)_{0 \le t \le T}$. Moreover, the law of $(Z_n(t), S_n(t))_{0 \le t \le T}$ under $P^\omega_0(\cdot)$ converges in $D^d([0, T]) \times D([0, T])$ in the $J_1 \times M_1$ topology as $n \to \infty$ towards the law of $(\sqrt{\Sigma} B_t, C_\infty \mathcal{S}_\gamma(t))_{0 \le t \le T}$, where $B$ is a standard Brownian motion and $\mathcal{S}_\gamma$ is a stable Subordinator of index $\gamma$ independent of $B$. 
\end{proposition}

\begin{proof} The convergence of $(Y_n(t))_{0 \le t \le T}$ is a consequence of \cite[Equation (11.2)]{Kious_Frib}.
We start by proving that Theorem~\ref{TheoremVarianceClockProcess} together with \cite[Lemma 11.2]{Kious_Frib} imply the fact that $S_n$ converges in the $M_1$ sense under the quenched law towards $\mathcal{S}_\gamma$, we follow the proof of \cite[Proposition~10.4]{Mourrat}. First we observe that
\[
S_n(t) = \frac{1}{\mathrm{Inv}(n)} \left( \tau_1 + \sum_{i = 1}^{\floor{nt} - 1} (\tau_{i + 1} - \tau_i) \right).
\]
We have that $\tau_1 < + \infty$ and $\tau_{n+1} - \tau_{n} < + \infty$ $\Pf$-a.s.\ by \cite[Theorem 5.1]{Kious_Frib}, this means that for $\mathbf{P}$-almost all realisation of the environment $\tau_1< + \infty$ $P^{\omega}$-a.s. Then, applying Slutsky's theorem, we only need to prove the quenched convergence of $S^*_n$ towards $\mathcal{S}_\gamma$.

In order to prove convergence of the finite dimensional distributions we introduce  
\begin{equation*}
    F(S_n^{*}(t_1),\cdots,S_n^{*}(t_m)) = \exp\left(-\lambda_1 S_n^{*}(t_1)- \lambda_2\left(S_n^{*}(t_2)- S_n^{*}(t_1)\right)\cdots-\lambda_m\left( S^{*}_n(t_m)-S_n^{*}(t_{m-1})\right)\right),
\end{equation*}
where $\lambda_1,\cdots,\lambda_m \in \mathbb{R}^m$ and $0 < t_1,\cdots,t_m \le T$. By \cite[Lemma 11.2]{Kious_Frib} we have annealed convergence, so that Laplace functionals converge, in particular
\begin{equation*}
    \mathbb{E}_{0}\left[F(S_n^*)\right] = \mathbf{E}\left[ E_{0}^\omega\left[ F(S_n^*) \right]\right] \underset{n \to +\infty}{\rightarrow} \mathbb{E}_{0}\left[ F(\mathcal{S}_\gamma) \right] = \exp\left( - t_1 \psi(\lambda_1) -  \dots - \psi(\lambda_{m})\left(t_m - t_{m-1}\right)  \right),
\end{equation*}
where $\psi$ is the Laplace exponent of the stable Subordinator. Moreover, since $F$ is Lipschitz we can apply Theorem~\ref{TheoremVarianceClockProcess}, more specifically \eqref{EquationBoundVarianceClock}, we know that for $b \in (1, 2)$
\begin{equation*}
   \sum_{k = 1 }^{+\infty}\mathbf{Var}\left[E_{0}^{\omega}(F(S^*_{b^k}))\right]< +\infty.
\end{equation*}
Therefore, by applying Chebyshev inequality and Borel-Cantelli to the sequence of events $|E_{0}^{\omega}(F(S^*_{b^k})) - \mathbb{E}_{0}\left[ F(S^*_{b^k}) \right]| > \varepsilon$ and using the annealed convergence we get that 
\begin{equation}\label{eqn:ConvergenceAlongbk}
    E_{0}^{\omega}\left[ F(S^*_{b^k}) \right] \to \mathbb{E}_{0}\left[ F(\mathcal{S}_\gamma) \right],
\end{equation}
almost surely along $b^k$. Moreover, by countable subadditivity, the convergence holds jointly on a set of full measure for all functions $F$ of the from above with $\lambda_1, \dots, \lambda_m$, $t_1, \dots, t_m$ and $b$ rationals. The convergence can be extended to any $\lambda_1, \dots, \lambda_m$ and $t_1, \dots, t_m$ by the continuity of the limit and the monotonicity of $S^*$. Let us justify this claim for the one-point marginal as extending to the f.d.d.\ follows in a straightforward manner. Fix any $(\lambda, t) \in \mathbb{R}^+ \times [0, 1]$, pick a sequence $(\lambda_{j}, t_{j})_{j \in \mathbb{N}}$ so that $\lambda_{j} \uparrow \lambda$ and $t_{j} \uparrow t$ are rational. Then, by monotonicity of $S^*$ and applying $\eqref{eqn:ConvergenceAlongbk}$ we get that for all $j$
\begin{align*}
    \limsup_{k \to \infty} E_{0}^{\omega}\left[ e^{- \lambda S^*_{b^k}(t)}  \right] \le \limsup_{k \to \infty} E_{0}^{\omega}\left[ e^{- \lambda_{j} S^*_{b^k}(t_{j})}  \right] = \mathbb{E}_{0}\left[ e^{- \lambda_{j} \mathcal{S}(t_{j})}  \right].
\end{align*}
By the fact that $t$ is almost surely a continuity point of $\mathcal{S}$, we observe that $\mathcal{S}(t_{j}) \uparrow \mathcal{S}(t)$. Hence, by the dominated convergence theorem
\begin{equation*}
    \limsup_{k \to \infty} E_{0}^{\omega}\left[ e^{- \lambda S^*_{b^k}(t)}  \right] \le \mathbb{E}_{0}\left[ e^{- \lambda \mathcal{S}(t)} \right].
\end{equation*}
Arguing symmetrically yields 
\begin{equation*}
    \liminf_{k \to \infty} E_{0}^{\omega}\left[ e^{- \lambda S^*_{b^k}(t)} \right] \ge \mathbb{E}_{0}\left[ e^{- \lambda \mathcal{S}(t)} \right],
\end{equation*}
which justifies our claim. On the set of full measure where the joint convergence holds, we want to extend the convergence to hold along $n$. Let $k_n$ be the largest integer such that $b^{k_n} \le n$. Using the definition of $F$, for any two increasing processes $h$ and $h'$ starting from 0 one gets
\begin{equation*}
    |F(h) - F(h')| \le C \max_{1 \le i \le m} |h(t_i) - h'(t_i)| \wedge 1.
\end{equation*}
Observe that
\begin{equation*}
    S^*_n(t_i) = \left( \tfrac{b^{k_n}}{n} \right)^{1/\gamma} S^*_{b^{k_n}} \left( \tfrac{n}{b^{k_n}} t_i\right).
\end{equation*}
By the definition of $k_n$ and the monotonicity of $S^*_{b^{k_n}}(t)$ the latter is smaller than $S^*_{b^{k_n}}(b t_i)$, as a consequence
\begin{equation}\label{MourratInequality}
   (b^{-\frac{1}{\gamma}}-1)S^*_{b^{k_n}}(t_i) \leq  S^*_n(t_i) - S^*_{b^{k_n}}(t_i) \le S^*_{b^{k_n}}(b t_i) - S^*_{b^{k_n}}(t_i).
\end{equation}
The quantity
\begin{equation*}
    \limsup_{n \to \infty} \left| E_{0}^\omega \left[ F(S^*_n(t_i)) \right] - E_{0}^\omega \left[ F(S^*_{b^{k_n}}(t_i)) \right] \right|,
\end{equation*}
is thus bounded from above, applying \eqref{MourratInequality}, up to a multiplicative constant, by
\begin{equation*}
    \limsup_{n \to \infty} E_{0}^\omega \left[ \max_{1 \le i \le m} \bigg(|S^*_{b^{k_n}}(b t_i) - S^*_{b^{k_n}}(t_i)|,|(b^{-\frac{1}{\gamma}}-1)S^*_{b^{k_n}}(t_i)|\bigg) \wedge 1  \right].
\end{equation*}
Using the convergence along $b^{k}$ we get that
\begin{align*}
    &\limsup_{n \to \infty} E_{0}^\omega \left[ \max_{1 \le i \le m} \bigg(|S^*_{b^{k_n}}(b t_i) - S^*_{b^{k_n}}(t_i)|,|(b^{-\frac{1}{\gamma}}-1)S^*_{b^{k_n}}(t_i)|\bigg) \wedge 1  \right] \\&= \mathbb{E}_{0} \left[ \max_{1 \le i \le m} \bigg(|\mathcal{S}_\gamma(b t_i) - \mathcal{S}_\gamma(t_i)|,|(b^{-\frac{1}{\gamma}}-1)\mathcal{S}_\gamma(t_i)| \bigg) \wedge 1 \right].
\end{align*}
By the fact that $\mathcal{S}_\gamma$ is almost surely continuous at deterministic times we get that the quantity on the right converges to $0$ as $b \to 1$ along rationals by dominated convergence theorem, since the quantity is bounded by $1$. The tightness of the process in $M_1$ follows from this f.d.d.\ convergence we just proved, using the fact that all processes involved are increasing (see \cite[Theorem 12.12.3]{whitt}).
\vspace{1ex}

The convergence of the position $Z_n(t)$ (both seen as the polygonal interpolation and in the canonical way) is proved using Theorem~\ref{TheoremVarianceClockProcess} (statement \eqref{EquationBoundVarianceTraj}) and annealed convergence towards Brownian motion in \cite[Lemma 4.1]{Szn_serie} or, alternatively, \cite[Lemma~3.1]{Bowditch}.

\vspace{1ex}
We now need to prove joint convergence of $S_n$ and $Z_n$. First notice that, by the previous two points and \cite[Theorem 11.6.7]{whitt}, tightness of the marginals is enough to have tightness in the product space. Therefore, we have that the law $(S_n, Z_n)$ is tight on $D^d([0, T]) \times D([0, T])$ endowed with the $J_1 \times M_1$-topology on a set of the environment $\Omega_1 \subset \Omega$ of full $\mathbf{P}$-measure. 

We need to prove that on a set of full measure $\Omega_2$, the law of the increments of the process $W_n = (S_n, Z_n)$ converge to the one of $W = (\mathcal{S}_\gamma, \sqrt{\Sigma} B)$. We can once again show the convergence of the process $W_n^{*}$ and then by Slutsty's theorem we obtain the convergence of the process $W_n$. Let $\lambda_1, \dots, \lambda_m \in \mathbb{R}^{d+1}$ and $0 < t_1 < \dots < t_m \le 1$ and define the function
\begin{equation}\label{EqnFourierFunctional}
    G(w) = \exp\Big( i \lambda_1 \cdot w(t_1)+i \lambda_2 \cdot \left( w(t_2) - w(t_1) \right) + \dots + i \lambda_m \cdot \left( w(t_m) - w(t_{m - 1}) \right) \Big).
\end{equation}
By the annealed convergence of \cite[Lemma 11.2]{Kious_Frib} we have that
\begin{equation*}
    \mathbb{E}_{0}\left[ G(W^{*}_n) \right] \underset{n \to +\infty}{\longrightarrow} \mathbb{E}_{0}\left[ G(W) \right].
\end{equation*}
We observe that the real part (respectively imaginary part) of $G$ is bounded and Lipschitz, and thus its positive and negative part also are bounded and Lipschitz. Thus, we can apply Theorem~\ref{TheoremVarianceClockProcess} to these functions and we obtain that for $\mathbf{P}-$almost every environment $\omega \in \Omega$,
\begin{align*}
    E^{\omega}_{0}[G(W^{*}_{b^{k}})] \underset{k \to +\infty}{\longrightarrow} \mathbb{E}_{0}[G(W)].
\end{align*}
Let us improve this convergence to hold along $n$ and not just $b^k$. The convergence along $b^k$ holds on a set of $\mathbf{P}$-probability $1$ jointly for all rational $b \in (1, 2)$. Fix $b$ rational and choose $k_n$ to be the largest integer such that $b^{k_n} \le n$. The key observation is that, due to the Lipschitz property we can bound
    \begin{equation*}
        \begin{split}
            \left| G(W^*_{n}) - G(W^*_{k_n}) \right| \le C \left( \max_{i = 0, \dots, m} \left|S^*_n(t_i) - S^*_{b^{k_n}}(t_i) \right| \wedge 1  +  \sup_{t \in [0, T]}\left\| Z^*_n(t_i) - Z^*_{b^{k_n}}(t_i) \right\| \wedge 1 \right).
        \end{split}
    \end{equation*}
We can take expectation $E^\omega$ on both sides of this expression. Let us observe that the arguments below \eqref{MourratInequality} imply that 
\begin{equation*}
    \lim_{b \to 1}\limsup_{n \to \infty}  E_{0}^\omega \left[ \max_{i = 0, \dots, m} \left|S^*_n(t_i) - S^*_{b^{k_n}}(t_i) \right| \wedge 1 \right] = 0.
\end{equation*}
Thus, we just need to show that
 \begin{equation*}
    \lim_{b \to 1}\limsup_{n \to \infty} E_{0}^\omega \left[ \sup_{t \in [0, T]}\left\| Z^*_n(t_i) - Z^*_{b^{k_n}}(t_i) \right\| \wedge 1 \right] = 0.
\end{equation*}
We omit the details here as this is showed within the proof of \cite[Lemma~3.1]{Bowditch}. Thus on a set of full measure $\Omega_2$ we have the convergence for all $\lambda_1, \dots, \lambda_m \in \mathbb{R}^{d+1}$ and $0 < t_1 < \dots < t_m \le T$ rationals. Finally, by right continuity of $W^{*}_n$ and $W$ and continuity of the Fourier functional it is possible to show that the limit holds for all $0 \le t_0 \le \dots \le t_m \le T$ and for all $\lambda_1, \dots, \lambda_m \in \mathbb{R}^{d+1}$ by arguing as we did for the Laplace functional and the process $\mathcal{S}_{\gamma}$, we omit the details. This finishes the proof as we showed that on a set of full measure $\Omega_1 \cap \Omega_2$ the sequence $W^{*}_n$ is tight and has just one possible limit $W$.
\end{proof}

We are now able to prove the main result, the method is the same as the proof of \cite[Theorem 11.2]{Kious_Frib}.
\begin{proof}[Proof of Theorem~\ref{MainTheorem}]
We fix a realisation of the environment $\omega \in \Omega$, the following convergences in distribution are meant to be understood under the quenched law $P^\omega_{0}$, we assume the result of Proposition~\ref{PropositionJointQuenchedConvergence}.
Recall that $D([0, T])$ and $D^d([0, T])$ denote the sets of c\`{a}dl\`{a}g functions in dimension one (resp.\ $d$), let us drop $[0, T]$ in the notation for simplicity. Let $D_{\uparrow}$ be the set of non-deacreasing c\`{a}dl\`{a}g functions, $D_{\uparrow\uparrow}$ the set of strictly increasing c\`{a}dl\`{a}g functions and $D_u$ the set of unbounded c\`{a}dl\`{a}g functions. Let $\mathcal{C}$ be the set of continuous functions and $U$ the uniform topology. It is a known fact that $\mathcal{S}_\gamma^{-1}$ is almost surely continuous and increasing.
First, we get that
\begin{equation}\label{firstEquationLastProof} 
    \left(Y_{n^\gamma/L(n)}\left( S_{n^\gamma/L(n)}^{-1}\left(\frac{nt}{\mathrm{Inv}\left(n^\gamma/L(n)\right)}\right) \right) \right)_{t \in [0, T]} \to \left( v C_\infty^{-\gamma} \mathcal{S}_\gamma^{-1}(t) \right)_{t \in [0, T]},
\end{equation}
in the uniform topology. This holds using the fact the right continuous inverse is continuous from $(D_{u, \uparrow\uparrow}, M_1)$ to $(\mathcal{C}, U)$ by \cite[Corollary~13.6.4]{whitt}, and if $(x_n, y_n) \to (x, y)$ in $D^d \times D_{\uparrow}$ with $(x, y) \in \mathcal{C}^d \times \mathcal{C}_{\uparrow}$, then $x_n \circ y_n \to x \circ y$ in the $U$ topology (see \cite[Theorem 13.2.1]{whitt}). Lastly, note that $(nt/\mathrm{Inv}(n^\gamma/L(n)))_{t \in [0, T]}$ converges uniformly to $(t)_{t \in [0, T]}$.

Moreover, applying the same results, it holds that
\begin{equation}\label{secondEquationLastProof} 
    \left(Z_{n^\gamma/L(n)}\left( S_{n^\gamma/L(n)}^{-1}\left(\frac{nt}{\mathrm{Inv}\left(n^\gamma/L(n)\right)}\right) \right) \right)_{t \in [0, T]} \stackrel{(\mathrm{d})}{\to} \left( C_\infty^{-\gamma/2} B_{\mathcal{S}_\gamma^{-1}(t)} \sqrt{\Sigma} \right)_{t \in [0, T]},
\end{equation}
in the $J_1$ topology. We need to ensure that $\mathbf{P}$-a.s., with high $P_0^\omega$-probability $X_{\floor{nt}}$ is close to $X_{\tau_{\left \lfloor \frac{n^\gamma}{L(n)} S_{n^\gamma/L(n)}^{-1}(nt/\mathrm{Inv}(n^\gamma/L(n)))\right\rfloor}}$. It is immediate to prove that $\tau_{\left \lfloor \frac{n^\gamma}{L(n)} S_{n^\gamma/L(n)}^{-1}(nt/\mathrm{Inv}(n^\gamma/L(n)))\right\rfloor}$ is the smallest $\tau_i$ such that $\tau_i > tn$. Then we can bound the distance $X_{\tau_{\left \lfloor \frac{n^\gamma}{L(n)} S_{n^\gamma/L(n)}^{-1}(nt/\mathrm{Inv}(n^\gamma/L(n)))\right\rfloor}} - X_{\floor{nt}}$ by a constant times the size of a regeneration block. Hence,
\begin{align*}
    M_n &\coloneqq \max_{t \in [0, T]} \left \| X_{\tau_{\left \lfloor \frac{n^\gamma}{L(n)} S_{n^\gamma/L(n)}^{-1}(nt/\mathrm{Inv}(n^\gamma/L(n)))\right\rfloor}} - X_{\floor{nt}} \right\| \\
    & \le C \max_{k = 1, \dots, \floor{nT} + 1} \left \| X_{\tau_k} - X_{\tau_{k - 1}} \right\|,
\end{align*}
by the fact that $\| X_{\tau_k} - X_{\tau_{k - 1}}\| \le C \chi_k^\alpha$ defined in \eqref{recall_chi}, applying \cite[Lemma 6.1]{Kious_Frib} and using a union bound one gets that, for all $M>0$ and $\delta>0$
\begin{equation*}
    \mathbb{P}_0\left( M_n > \delta n^{\gamma/4} \right) = \mathbf{E}\left[ P_0^\omega (M_n > \delta n^{\gamma/4}) \right] \le n^{-M}.
\end{equation*}
Markov inequality gives that
\begin{equation*}
     \mathbf{P}\left( P_0^\omega (M_n > \delta n^{\gamma/4}) > n^{-1} \right) \le n^{-2}.
\end{equation*}
Borel-Cantelli implies that $\mathbf{P}$-almost surely\ $P_0^\omega [M_n > \delta n^{\gamma/4}] \le 1/n$ for all $n$ large enough, which leads us to conclude that $\|X_{\tau_{\left\lfloor\frac{n^\gamma}{L(n)} S_{n^\gamma/L(n)}^{-1}(nt/\mathrm{Inv}(n^\gamma/L(n)))\right\rfloor}} - X_{\floor{nt}}\| \le \delta n^{\gamma/4}$ in the uniform sense with high $P_0^\omega$-probability. This implies, together with \eqref{firstEquationLastProof} and \eqref{secondEquationLastProof} and Slutsky's theorem, that statement \eqref{FirstStatementMain} holds $\mathbf{P}$-a.s. under the quenched law $P_0^\omega$, and that
\begin{equation}\label{EqnProofMainTheo}
    \left(\frac{X_{\floor{nt}} - v \frac{n^\gamma}{L(n)} S_{n^\gamma/L(n)}^{-1}\left(\frac{nt}{\mathrm{Inv}\left(n^\gamma/L(n)\right)}\right)}{\sqrt{n^\gamma/L(n)}} \right)_{t \in [0, T]} \stackrel{(\mathrm{d})}{\to} \left( C_\infty^{-\gamma/2} B_{\mathcal{S}_\gamma^{-1}(t)} \sqrt{\Sigma} \right)_{t \in [0, T]}.
\end{equation}
Finally, notice that \eqref{SecondStatementMain} follows by applying the linear transformation $(I_d - P_{v_0})$ to \eqref{EqnProofMainTheo} as $(I_d - P_{v_0}) v = 0$. All linear combinations are continuous in the $J_1$-topology at continuous functions.
\end{proof}

\appendix

\section{Extension of Markov property}
The appendix gives some results which are useful in the paper but considered as tools.

Since throughout paper we consider two random walks, we extend the Markov property when we have two Markov chains with the following general lemma.

\begin{lemma} \label{Lemma2StoppingTimes}
Consider two Markov chains $(X_n)_{n \geq 0}$ and $(Y_n)_{n \geq 0}$ which take values in a countable set $E$. Consider the product measure $P_{(x,y)} = P_x \otimes P_y$. We consider the canonical markov chain on $E^2$. We define for $k,k' \in \mathbb{N}$ the translation operator
\begin{equation}\label{eqn:DoubleTimeShift}
    \func{\theta_{k,k'}}{(E^2)^{\mathbb{N}}}{(E^2)^{\mathbb{N}}}{((x_n),(y_n))}{((x_{n+k}),(y_{n+k'}))}.
\end{equation}
Consider also $P_{x,y} = P_x \otimes P_y$ almost surely finite stopping times $\tau$ and $\rho$ for $X$ and $Y$. Consider two bounded and measurable functions $f$ and $g$ which are, respectively, measurable with respect to
    \begin{align*}
        \sigma\Big(X_0,...,X_{\tau},Y_0,...,Y_{\rho}\Big) \text{ and }\sigma\Big((X_n)_{n \geq 0},(Y_n)_{n \geq 0}\Big).
    \end{align*}
Then we have the following equality:
    \begin{align*}
        E_{x,y}\left[f g \circ \theta_{\tau,\rho}\right] = E_{x,y}\left[fE_{X_{\tau},Y_{\rho}}[g]\right].
    \end{align*}
\end{lemma}

\begin{proof}
The proof is straightforward using strong Markov property and a Monotone Class Theorem. Indeed, if $f$ is of the form $f((x_n),(y_n)) = f_1((x_n))f_2((y_n))$ with $f_1$ measurable with respect to $\sigma\big(X_0,...,X_{\tau}\big)$ and $f_2$ measurable with respect to $\sigma\big(Y_0,...,Y_{\rho}\big)$ and $g((x_n),(y_n)) = g_1((x_n))g_2((y_n))$ with $g_1$ measurable with respect to $\sigma\big((X_n)_{n \geq 0}\big)$ and $g_2$ measurable with respect to $\sigma\big((Y_n)_{n \geq 0}\big)$. Then we have that
    \begin{align*}
        E_{x,y}\left[f g \circ \theta_{\tau,\rho}\right]&= E_{x,y}\left[f_1 g_1  \theta_{\tau} \times f_2g_2 \circ \theta_{\rho}\right]\\
        & = E_{x}\left[f_1 g_1 \circ \theta_{\tau}\right]\times E_{y}\left[f_2 g_2 \circ \theta_{\rho}\right],
    \end{align*}
 where the last equality comes from the independence of the two walks. Using the Strong Markov property we get
 \begin{align*}
    E_{x}\left[f_1 E_{X_{\tau}}[g_1]\right]\times E_{y}\left[f_2 E_{Y_{\rho}}[g_2]\right] & = E_{x,y}\left[f_1f_2 E_{X_{\tau}}[g_1]E_{Y_{\rho}}[g_2]\right]\\
    & = E_{x,y}\left[f E_{X_{\tau},Y_{\rho}}[g]\right].
 \end{align*}
We conclude using a Monotone Class Theorem argument.
\end{proof}

\section{Polynomial decay}
\begin{proposition} \label{TechnicalInequality}
For any $\varepsilon > 0$, there exists $K_0 > 1$ such that for any $K \ge K_0$, any $n \in \mathbb{N}^{*}$ and any $z \in \mathbb{Z}^d$ we have 
\begin{equation*}
    \Pf_0(X_{\tau_n} = z) \leq Cn^{\varepsilon-\frac{d}{2}},
\end{equation*}
where $C$ only depends on $\lambda,\vec{\ell},d,K$.
\end{proposition}

\begin{remark1}
We recall that the notion of regeneration time introduced in \cite{Kious_Frib}(see the definition in Section~\ref{definition_reg_times}) depends on the value of $K$, but we always take $K$ large enough.
\end{remark1}

\begin{proof}

First, we have that
\begin{equation}
    \begin{split}\label{Xtaunz}
        \Pf_0(X_{\tau_n} = z)& =\Pf_0\left(X_{\tau_1} +\somme{i=1}{n-1}{(X_{\tau_{i+1}}-X_{\tau_i})} = z\right)\\
         &=\Pf_0\left(\somme{i=1}{n-1}{(X_{\tau_{i+1}}-X_{\tau_i} - v )} = z - X_{\tau_1} - (n-1)v\right),
    \end{split}
\end{equation}
where $v = \mathbb{E}[X_{\tau_{2}}-X_{\tau_1}]$, we introduce for $i\geq 1$, $\xi_i = X_{\tau_{i+1}}-X_{\tau_i} - v $.

The family $(\xi_i)_{i \ge 1}$ is an independent family of random variables identically distributed and independent of $X_{\tau_1}$. By \cite[Lemma 6.1]{Kious_Frib}, for any $M > 0$ there exists $K_0 > 1$ such that for any $K \ge K_0$ we have
\begin{equation*}
    \mathbb{E}_{0}\left[\|\xi_1\|^{M}\right] < +\infty.
\end{equation*}
Thus, we can write \eqref{Xtaunz} as
\begin{align}\label{Toboundxi}
\somme{y \in \mathbb{Z}^d}{}{\Pf_0(X_{\tau_1} = y)\Pf_0\bigg(\somme{i=1}{n-1}{\xi_i} = z - y - (n-1)v\bigg)}.
\end{align}
Then we only need a bound for the value of $\Pf_0\bigg(\displaystyle \somme{i=1}{n-1}{\xi_i} = x\bigg)$ with $x = z - y - (n-1)v$ uniformly for $y \in \mathbb{Z}^d$.

We define $\func{H}{\mathbb{R}}{\mathbb{R}}{x}{x^{M}}$, then we have $\mathbb{E}_0(H(\|\xi_i\|)) < +\infty $. $H$ increases faster than $x \mapsto x^{2+\delta}$ and slower than $x \mapsto e^x$. Using \cite[Theorem 3]{Zaitsev} we can construct a probability space and two family of random variables $(Y_i)$ and $(N_i)$ such that:
\begin{itemize}
\item[$\bullet$]  $(Y_i)_{i \geq 1}$ has the same law as $(\xi_i)_{i \geq 1}$ under $\Pf_0$.
\item[$\bullet$]  $(N_i)_{i \geq 1}$ is a family of independent identically distributed centered gaussian vectors with covariance matrix $\Sigma$ where $\Sigma$ is the covariance matrix of $\xi_1$.
\end{itemize}
Moreover, they satisfy
\begin{equation*}
    \Pf\left(\max_{\substack{1 \leq j \leq n}}(\|\somme{i=1}{j}{(Y_i - N_i)}\|) \geq cn^{\varepsilon}\right) \leq C\frac{n}{n^{\varepsilon M}},
\end{equation*}
where the constants $C,c > 0$ only depend on $\lambda,\vec{\ell},d,K$. We write
\begin{align}
\nonumber \Pf_0\bigg(\somme{i=1}{n-1}{\xi_i} = x\bigg) &= \Pf\bigg(\somme{i=1}{n-1}{Y_i} = x\bigg) \\
&\nonumber \leq Cn^{1-\varepsilon M} + \Pf \bigg(\somme{i=1}{n-1}{N_i} \in B_{\| \cdot \|_{1}}(x,cn^{\varepsilon})\bigg) \\
&\label{Gaussian} \leq Cn^{1-\varepsilon M} + \Pf \bigg( \frac{\somme{i=1}{n-1}{N_i}}{\sqrt{n-1}} \in B_{\| \cdot \|_{1}}\bigg(\frac{x}{\sqrt{n-1}},\frac{cn^{\varepsilon}}{\sqrt{n-1}}\bigg) \bigg).
\end{align}
Since $\displaystyle \frac{\somme{i=1}{n-1}{N_i}}{\sqrt{n-1}}$ is a centered Gaussian vector with covariance matrix $\Sigma$, the right term in \eqref{Gaussian} is bounded from above by
\begin{align*}
    C\mu\left(B_{\| \cdot \|_{1}}\bigg(\frac{x}{\sqrt{n-1}},\frac{cn^{\varepsilon}}{\sqrt{n-1}}\bigg)\right),
\end{align*}
where $\mu$ denotes the Lebesgue measure on $\mathbb{R}^d$ and $C$ is a constant which depends on $\Sigma$, hence on $\lambda,\vec{\ell},d,K$. We deduce that last term is dominated by
\begin{align*}
    Cn^{d\varepsilon-\frac{d}{2}},
\end{align*}
where $C$ is a constant which only depends on $\lambda,\vec{\ell},d,K$.

Choosing $M$ large enough such that $1-\varepsilon M < - \frac{d}{2}$ we obtain that \eqref{Toboundxi} is bounded from above by 
\begin{align*}
Cn^{1-\varepsilon M} + Cn^{d\varepsilon-\frac{d}{2}} \leq C n^{\varepsilon d-\frac{d}{2}},
\end{align*}
where the constant $C$ only depends on $\lambda,\vec{\ell},d,K$. We conclude since the result is true for all $\varepsilon > 0$.
\end{proof}

\section*{Acknowledgments}
We would like to thank Daniel Kious for encouraging this collaboration and useful comments at various stages of the project. We would like also to thank Christophe Sabot and Noam Berger for insightful discussions.

AF acknowledges the support of an NSERC Discovery Grant. During the completion of this work CS was a PhD student at the University of Bath supported by a scholarship from the EPSRC Centre for Doctoral Training in Statistical Applied Mathematics at Bath (SAMBa), under the project EP/S022945/1. 

We thank the anonymous Referee for many valuable comments that greatly improved the presentation of the paper.

\printbibliography

\end{document}